\documentclass[a4paper,10pt,leqno]{amsart}

\usepackage{amsmath,amssymb,amsthm}
\usepackage{mathtools}
\usepackage{setpartitions}
\usepackage[all]{xy}
\usepackage{tikz}
\usepackage{tkz-euclide}
\usetikzlibrary{arrows,calc} 
\tikzset{hidden/.style = {thick, dashed}}
\usepackage[margin=3cm]{geometry}
\usepackage{mathdots,enumitem,xcolor,comment}
\usepackage{cancel}
\usepackage{hyperref}
\usepackage[noabbrev,capitalise,nameinlink]{cleveref}
\hypersetup{
    colorlinks,
    linkcolor={teal!50!black},
    citecolor={teal!50!black},
    urlcolor={teal!80!black}
}

\theoremstyle{plain}

\theoremstyle{definition}
\newtheorem{theorem}{Theorem}[section]

\newtheorem{corollary}[theorem]{Corollary}
\newtheorem{definition}[theorem]{Definition}
\newtheorem{lemma}[theorem]{Lemma}
\newtheorem{proposition}[theorem]{Proposition}

\newtheorem{problem}[theorem]{Problem}

\theoremstyle{remark}
\newtheorem{remark}[theorem]{Remark}
\newtheorem{example}[theorem]{Example}

\definecolor{forestgreen}{rgb}{0.0, 0.27, 0.13}

\newcommand{\triang}[1]{
	\foreach \i in {1,...,#1}
	{
		\node at (\i,0) {\i};
	}
	\pgfmathtruncatemacro{\n}{#1 -1};
	\foreach \i in {1,...,\n}
	{
		\foreach \j in {1,...,\i}
		{
			\draw[fill=lightgray,thick] (1+ \i - 1/2*\j,1/2*\j) circle (2pt);
		}
	}
}

\newcommand{\pretriang}[1]{
	\pgfmathtruncatemacro{\n}{#1 -1};
	\foreach \i in {1,...,\n}
	{
		\foreach \j in {1,...,\i}
		{
			\draw[fill=lightgray,thick] ({0.5*(1+ \i - 1/2*\j)},1/4*\j) circle (2pt);
		}
	}
}

\newcommand{\marked}[2]{
	\node at (1/2*#1 + 1/2*#2,1/2*#2 - 1/2*#1) {\scalebox{1.3}{$\textcolor{red}{\varheartsuit}$}};
}

\newcommand{\ctriang}{
\triang{9}
\marked{2}{3}
\marked{2}{4}
\marked{2}{7}
\marked{3}{5}
\marked{3}{6}
\marked{3}{8}
\marked{4}{5}
\marked{4}{6}
\marked{4}{8}
\marked{5}{7}
\marked{6}{7}
\marked{7}{8}
}

\newcommand{\cbot}{
\draw[ultra thick] (1.15,.15) -- (1.25,.25);
\draw[ultra thick] (3.15,.15) -- (3.25,.25);
\draw[ultra thick] (6.15,.15) -- (6.25,.25);
\draw[ultra thick] (5.85,.15) -- (5.75,.25);
\draw[ultra thick] (6.85,.15) -- (6.75,.25);
\draw[ultra thick] (7.85,.15) -- (7.75,.25);
}

\newcommand{\cptriang}{
\triang{9}
\marked{1}{2}
\marked{1}{5}
\marked{1}{7}
\marked{1}{8}
\marked{2}{3}
\marked{2}{4}
\marked{2}{6}
\marked{2}{9}
\marked{3}{5}
\marked{3}{7}
\marked{3}{8}
\marked{4}{5}
\marked{4}{7}
\marked{4}{8}
\marked{5}{6}
\marked{5}{9}
\marked{6}{7}
\marked{6}{8}
\marked{7}{9}
\marked{8}{9}
}

\usepackage{transparent}
\renewcommand{\S}{S}

\DeclareSymbolFont{extraup}{U}{zavm}{m}{n}
\DeclareMathSymbol{\varheartsuit}{\mathalpha}{extraup}{86}

\newcommand\precdot{\mathrel{\ooalign{$\prec$\cr
  \hidewidth\raise.25ex\hbox{$\scalebox{.6}{$\cdot$}\mkern0.5mu$}\cr}}}

\newcommand{\nc}{\pi}%
\newcommand{\nn}{p}%
\newcommand{\Kr}{\mathrm{Krew}}
\newcommand{\Cat}{\mathrm{Cat}}
\newcommand{\NC}{\mathrm{NC}}
\newcommand{\NN}{\mathrm{NN}}

\newcommand{\Camb}{\mathrm{Camb}}
\newcommand{\Roww}{\mathrm{Row}}
\newcommand{\Supp}{\mathrm{Supp}}
\renewcommand{\L}{\mathcal{L}}
\renewcommand{\a}{w}
\renewcommand{\r}{r}
\newcommand{\Out}{\mathrm{Out}}
\newcommand{\In}{\mathrm{In}}
\renewcommand{\sc}{\mathsf{c}}
\newcommand{\w}{\mathsf{w}}
\newcommand{\T}{\mathsf{T}}
\newcommand{\sT}{\mathsf{T}'}
\renewcommand{\t}{\mathrm{tog}}
\newcommand{\p}{P}

\newcommand{\Bal}{\mathrm{Bal}}
\newcommand{\ARc}{\mathrm{AR}(c)}
\newcommand{\AR}{\mathrm{AR}}
\newcommand{\co}{c_1}
\newcommand{\rot}{\mathrm{rot}}
\newcommand{\row}{\mathrm{Row}}
\newcommand{\Charm}{\mathrm{Charm}}
\newcommand{\ARcp}{\mathrm{AR}(c')}

\newcommand{\Qc}{Q_c}
\newcommand{\inv}{\mathrm{inv}}

\newcommand{\M}{\varheartsuit}
\newcommand{\bigheart}{\scalebox{3.5}{\textcolor{red}{\heart}}}
\newcommand{\heart}{\ensuremath\varheartsuit}
\usepackage{scalerel,adjustbox}

\newcommand{\lheart}{L}
\newcommand{\rheart}{R}

\DeclareMathOperator{\bijK}{\phi}
\DeclareMathOperator{\dyck}{\mathrm{Dyck}}
\DeclareMathOperator{\thicken}{\mathrm{Match}}
\DeclareMathOperator{\mrho}{\mathrm{Krow}} %
\DeclareMathOperator{\Krow}{\mrho}
\DeclareMathOperator{\Krew}{\Kr}

\newcommand{\orbsupp}{\mathrm{OrbitSupp}}

\definecolor{darkblue}{rgb}{0.0,0,0.7}

\newcommand{\defn}[1]{\emph{\color{red!50!black} #1}}

\usepackage{tikz-cd} 
\usetikzlibrary{tikzmark}

\author[Dequ\^ene]{Benjamin Dequ\^ene}
\address[B.~Dequ\^ene]{D\'epartement de math\'ematiques, LaCIM, Universit\'e du Qu\'ebec \`a Montr\'eal}
\email{dequene.benjamin@courrier.uqam.ca}
\author[Frieden]{Gabriel Frieden}
\address[G.~Frieden]{McGill University, Montr\'eal, Canada}
\email{gabefri@gmail.com}
\author[Iraci]{Alessandro Iraci}
\address[A.~Iraci]{Universit\`a di Pisa, Italy}
\email{alessandro.iraci@unipi.it}
\author[Schreier-Aigner]{Florian Schreier-Aigner}
\address[F.~Schreier-Aigner]{University of Vienna, Austria}
\email{florian.schreier-aigner@univie.ac.at}
\author[Thomas]{Hugh Thomas}
\address[H.~Thomas]{D\'epartement de math\'ematiques, Universit\'e du Qu\'ebec \`a Montr\'eal}
\email{thomas.hugh\_r@uqam.ca}
\author[Williams]{Nathan Williams}
\address[N.~Williams]{The University of Texas at Dallas}
\email{nathan.f.williams@gmail.com}

\begin{document}

\title{Charmed roots and the Kroweras complement}

\begin{abstract}
	Although both noncrossing partitions and nonnesting partitions are uniformly enumerated for Weyl groups, the exact relationship between these two sets of combinatorial objects remains frustratingly mysterious.  In this paper, we give a precise combinatorial answer in the case of the symmetric group: for any standard Coxeter element, we construct an equivariant bijection between noncrossing partitions under the \emph{Kreweras complement} and nonnesting partitions under a Coxeter-theoretically natural cyclic action we call the \emph{Kroweras complement}.  Our equivariant bijection is the unique bijection that is both equivariant and support-preserving, and is built using local rules depending on a new definition of \emph{charmed roots}.  Charmed roots are determined by the choice of Coxeter element---in the special case of the linear Coxeter element $(1,2,\ldots,n)$, we recover one of the standard bijections between noncrossing and nonnesting partitions.
\end{abstract}

\makeatletter
\@namedef{subjclassname@2020}{%
	\textup{2020} Mathematics Subject Classification}
\makeatother

\subjclass[2020]{
	Primary: 05E18
}

\keywords{Catalan combinatorics, noncrossing partitions, nonnesting partitions, Kreweras complement}

\date{\today}

\maketitle

\section{Introduction}
\label{sec:intro}

\subsection{Noncrossing and nonnesting partitions}
Let $W \subseteq \mathrm{GL}(V)$ be a finite complex reflection group acting in its reflection representation on a complex vector space $V$ of dimension $\r$ with reflections $T$~\cite{Humphreys1990,lehrer2009unitary}.  It is well-known that the ring of $W$-invariants $\mathbb{C}[V]^W$ is a polynomial ring generated by invariants of degrees $d_1\leq d_2 \leq \cdots \leq d_\r$.  The \defn{Coxeter number} of a well-generated $W$ (that is, $W$ is generated by $r$ reflections) is $h=d_\r$ and the \defn{$W$-Catalan number} is

\begin{equation}
	\label{eq:catalan}
	\mathrm{Cat}(W) \coloneqq \prod_{i=1}^\r \frac{h+d_i}{d_i}.
\end{equation}

\subsubsection{Noncrossing partitions}
The \defn{absolute order} on $W$ is the poset defined as the oriented Cayley graph of $W$ generated by $T$, where the identity of $W$ is the minimal element.  For $W$ well-generated, a \defn{Coxeter element} $c$ is defined to be a regular element of order $h=d_\r$;
any Coxeter element has rank $\r$ in the absolute order.   The \defn{$c$-noncrossing partition lattice} $\NC(W,c)$ is the interval $[e,c]_T$ in the absolute order (it is indeed a lattice)~\cite{brady2002k,BradyWatt2008,bessis2003dual}.  For any Coxeter elements $c,c'$, we have $\NC(W,c) \simeq \NC(W,c')$ as posets because there exists an automorphism of $W$ (sometimes outer) fixing the set of reflections and sending $c$ to $c'$~\cite{reiner2017non}.  The \defn{support} of a noncrossing partition $\nc \in \NC(W,c)$ is the set $\Supp(\nc)$ of simple reflections required to write a reduced word in simple reflections for $\nc$.

Until very recently, the number of noncrossing partitions had only been computed case-by-case; a uniform proof was found in the case of real $W$ in~\cite{galashin2022rational}.

\begin{theorem}[{\cite[Section 13]{bessis2015finite}}] %
	For $W$ a well-generated finite complex reflection group with Coxeter element $c$, $|\NC(W,c)| = \mathrm{Cat}(W)$.
\end{theorem}
\Cref{sec:comb} further discusses $c$-noncrossing partitions.

\subsubsection{Nonnesting partitions}
Let now $W$ be a Weyl group (a crystallographic real reflection group), with positive roots $\Phi^+$.  The positive root poset is the partial order on $\Phi^+$ defined by $\alpha \leq \beta$ iff $\beta-\alpha$ is a nonnegative sum of positive roots.  The \defn{nonnesting partitions} $\NN(W)$ are the order ideals in the positive root poset~\cite[Remark 2]{reiner1997non}.  The \emph{support} of a nonnesting partition $\nn \in \NN(W)$ is the set $\Supp(\nn)$ of simple roots that lie in $\nn$ (as an order ideal of $\Phi^+$).

The number of nonnesting partitions has been computed uniformly by combining a bijection of Cellini and Papi~\cite[Theorem 1]{cellini2002ad} with an argument of Haiman~\cite[Theorem 7.4.2]{haiman1994conjectures} and an observation of Thiel~\cite[Lemma 8.2]{thiel2016anderson}.

\begin{theorem}[{\cite{cellini2002ad,haiman1994conjectures,thiel2016anderson}}]
	For $W$ a Weyl group, $|\NN(W)| = \Cat(W)$.
\end{theorem}
\Cref{sec:nnps} further discusses nonnesting partitions.

\subsection{History and Incongruity}
\label{sec:history}

Despite the fact that they are both counted by $\mathrm{Cat}(W)$, there are several Incongruities between $\NC(W,c)$ and $\NN(W)$:

\begin{enumerate}
	\item $\NC(W,c)$ is defined for well-generated complex reflection groups, while $\NN(W)$ is only defined for Weyl groups;
	      \item\label{it:2} the definition of $\NC(W,c)$ requires the choice of a Coxeter element, while $\NN(W)$ has no such dependence;
	      \item\label{it:3} When $c$ is a \defn{standard Coxeter element}\footnote{Outside of this introduction, we will always assume that $c$ is a standard Coxeter element.} in a real reflection group---that is, when $c$ is a product of the simple reflections in some order---and $s$ is initial in $c$, the map $w \mapsto s^{-1} w s$ gives a bijection $\NC(W,c) \simeq \NC(W,s^{-1}cs)$ and leads to the \emph{Kreweras complement} and the \emph{Cambrian recurrence} for real $W$, while there is no obvious similar action of the initial $s$ on $\NN(W)$.
\end{enumerate}
Two additional Incongruities have only been very recently resolved in~\cite{galashin2022rational}:
\begin{enumerate}[resume]
	\item\label{it:4} $\NC(W,c)$ had only \emph{Fuss--Catalan} and \emph{Fuss--Dogolon} generalizations, while $\NN(W)$ was easily generalized to rational parameters $p$ coprime to the Coxeter number $h$; %
	\item the enumeration of $\NC(W,c)$ was only case-by-case, while the enumeration of $\NN(W)$ was uniform.
\end{enumerate}

The exact relationship between noncrossing and nonnesting partitions, therefore, remains frustratingly mysterious, and is perhaps the biggest open question in Coxeter--Catalan combinatorics:

\begin{problem}
\label{prob:main_question}
Fix $W$ a finite Weyl group and $c$ a Coxeter element.  Find a ``natural'' bijection between $\NC(W,c)$ and $\NN(W)$.
\end{problem}

Historically---if one interprets ``natural'' as meaning ``uniform''---finding such a bijection was viewed as a promising way to give a uniform proof that $|\NC(W,c)|=\Cat(W)$.  But even with the recent uniform enumeration of $\NC(W,c)$,~\Cref{prob:main_question} remains interesting in its own right.

To our taste, there are two approaches to~\Cref{prob:main_question}: the first approach is based on the case-by-case combinatorial models available in the classical types $A,B,D$~\cite{giraldo2009bijections,stump2010crossings,conflitti2011noncrossing,kim2011new,stump2008non,athanasiadis1998noncrossing}; the second approach was pioneered in~\cite{armstrong2013uniform} based on observations in~\cite{panyushev2009orbits,BessisReiner2011}, and uses a coincidence of cyclic actions (partially resolving Incongruity~\eqref{it:3} above) to induce a bijection.  Our main theorem will refine both of these approaches in the special case of the symmetric group $\S_n$, recovering one of the standard bijections between noncrossing and nonnesting partitions in the case of the linear Coxeter element $(1,2,\ldots,n)$.

\subsection{Cyclic actions on noncrossing and nonnesting partitions}%

\subsubsection{Armstrong-Stump-Thomas}
For $W$ a well-generated complex reflection group, we define the \defn{$c$-Kreweras complement} on the noncrossing partition lattice as the anti-automorphism~\cite[Section 4.2]{Armstrong2009}, \cite{kreweras1972partitions} \begin{align}\label{eq:krew} \Kr_c: \NC(W,c) &\to \NC(W,c) \\ \nonumber \nc &\mapsto \nc^{-1}c. \end{align}  Since $\Kr_c^2(\nc) = c^{-1} \nc c$ and $c$ has order $h$, $\Kr_c$ has order $h$ if $-1 \in W$, and $2h$ otherwise.%

For $W$ a Weyl group, \defn{rowmotion} on nonnesting partitions is the map \begin{align}\label{eq:row} \Roww: \NN(W) &\to \NN(W) \\ \nonumber \nn &\mapsto \min_{\Phi^+} \{\alpha \mid \alpha \not \leq \beta \text{ for any } \beta \in \nn\}.\end{align} Panyushev conjectured that the order of $\Roww$ on $\NN(W)$ was $h$ if $-1 \in W$ and $2h$ otherwise~\cite{panyushev2009orbits}, and Bessis and Reiner refined Panyushev's conjecture by observing that $\Roww$ had the same orbit structure on $\NN(W)$ as $\Kr$ on $\NC(W,c)$~\cite{BessisReiner2011}.  This conjecture was proven by Armstrong, Stump, and Thomas, using a uniformly-stated---but only case-by-case verified---inductive bijection between nonnesting and noncrossing partitions~\cite{armstrong2013uniform}.%

Thus, partitioning the sets $\NC(W,c)$ and $\NN(W)$ into orbits under $\Kr_c$ and $\Roww$ refines the problem of finding a ``natural'' bijection.  But Incongruity~\eqref{it:2} remains---while the definitions of $\NC(W,c)$ and $\Kr_c$ depend on the choice of a Coxeter element, the set $\NN(W)$ and its action $\Roww$ do not depend on any such choice.

\subsubsection{Flips and the Kreweras Complement on Noncrossing Partitions}
Let $W$ be a finite Coxeter group. For $c$ a standard Coxeter element and $\sc$ a particular choice of reduced word for $c$, the \defn{$c$-sorting word} $\w(\sc)$ for $w \in W$ is the leftmost reduced word in simple reflections for $w$ in $\sc^\infty$.  An element is \defn{$c$-sortable} if its $c$-sorting word uses a weakly decreasing subset of simple reflections in each successive copy of $\sc$.  The \defn{weak order} is the poset given by the Cayley graph of $W$ generated by $S$, with the longest element $w_\circ$.  The restriction of the weak order to the $c$-sortable elements is Nathan Reading's $c$-Cambrian lattice, with edges labeled by roots.  The inversion sequence $\inv(\w_\circ(\sc))=[\alpha^{(1)},\alpha^{(2)},\ldots,\alpha^{(N)}]$ gives total ordering on the set of roots of $W$ (see~\Cref{sec:sort_ar} for more details on this ordering).

In~\cite{thomas2019rowmotion}, a subset of the authors showed that it is possible to compute the $c$-Kreweras complement in ``slow motion'' as a sequence of local \emph{flips} on the edges of the $c$-Cambrian lattice in $\inv(\w_\circ(\sc))$ order.

\begin{theorem}[{\cite[Theorem 1.2]{thomas2019rowmotion}}] \label{thm:slow}
	For $W$ a real reflection group, $\Kr_c(\nc)$ can be computed as a sequence of flips on $\Camb(W,c)$.
\end{theorem}

\subsubsection{Toggles and the Kroweras Complement on Nonnesting Partitions}
As with the $c$-Kreweras complement, rowmotion can also be written as a sequence of local moves~\cite{cameron1995orbits,striker2012promotion}.  A \emph{toggle} $\t_\alpha(\nn)$ of a nonnesting partition $\nn$ at a positive root $\alpha$ either adds $\alpha$ to $\nn$ (when $\alpha \not \in \nn$) or removes $\alpha$ from $\nn$ (when $\alpha \in \nn$), provided that the result is again a nonnesting partition.   For nonnesting partitions, $\row$ can be computed by toggling each root of the root poset in order of height (or by \emph{row}).

We are now in a position to address Incongruity~\eqref{it:2}: in contrast to the $c$-Kreweras complement, $\row$ has no dependence on the Coxeter element $c$.  We resolve Incongruity~\eqref{it:2} by mirroring the sequence of local moves used in the ``slow motion'' $c$-Kreweras complement of~\Cref{thm:slow}---but replacing flips by toggles~\cite{williams2013cataland,williams2014bijactions}:
\begin{align}\label{eq:krow} \Krow_c: \NN(W) &\to \NN(W) \\ \nonumber \nn &\mapsto \left(\t_{\alpha^{(N)}}\circ\cdots\circ\t_{\alpha^{(2)}}\circ\t_{\alpha^{(1)}}\right)(\nn).\end{align}
As further explained in~\Cref{sec:Kroweras}, we call this map the \defn{$c$-Kroweras complement}.

\subsection{Main Theorem}

In~\Cref{sec:bij}, we construct a novel bijection between $c$-noncrossing partitions and nonnesting partitions for the symmetric group $\S_n$ using the combinatorics of certain kissing families of lattice paths whose local behavior is dependent on a novel definition of \emph{$c$-charmed roots}.  When $c$ is a standard Coxeter element of $\S_n$, $c$ consists of a single cycle $(\a_1,\a_2, \ldots, \a_m,\a_{m+1} \dots, \a_n)$ with an initial increasing subsequence $1 = \a_1 < \a_2 < \dots < \a_m = n$, followed by a decreasing sequence $n = \a_m > \a_{m+1} > \dots > \a_n$ of the remaining unused entries.  We say that the reflection $(\a_i,\a_j)$ is \defn{$c$-charmed} if $1 < i < m$ and $m < j \leq n$ (see~\Cref{sec:coxeter_els} for a more complete discussion), and \defn{$c$-ordinary} otherwise.

Our bijection is the unique equivariant support-preserving bijection between $\NC(W,c)$ under the $c$-Kreweras complement and $\NN(W)$ under the $c$-Kroweras complement, and is based on \emph{intimate families} of lattice paths (see \Cref{sec:intimate}).

\begin{theorem}
	\label{thm:main_theorem}
	Let $\S_n$ be the symmetric group, and fix a standard Coxeter element $c \in \S_n$. Then there is a unique bijection $\Charm_c \colon \NC(\S_n,c) \to \NN(\S_n)$ satisfying
	\begin{itemize}
		\item $\Charm_c \circ \Krew_c = \Krow_c \circ \; \Charm_c$ and
		\item $\Supp = \Supp \circ \Charm_c$.
	\end{itemize}
\end{theorem}

As we describe in~\Cref{cor:tails}, it is easy to directly read the blocks of the noncrossing partition from the corresponding intimate family.  We prove \Cref{thm:main_theorem} in~\Cref{sec:proof} using \defn{Cambrian induction}---that is, we will show
\begin{itemize}
	\item the \emph{base case}: the theorem holds for the linear Coxeter element $\co = s_1s_2 \cdots s_{n-1}$.  In this case, $\Krow_{\co}$ is order ideal promotion, as defined by Striker and Williams \cite{striker2012promotion}, and $\Charm_{\co}$ recovers a standard bijection; and
	\item the \emph{inductive step}: suppose \Cref{thm:main_theorem} holds for $c$.  We argue that if $s$ is initial in $c$ and $c'=scs$, then \Cref{thm:main_theorem} also holds for $c'$---we understand how $\Charm_{c'}$ differs from $\Charm_{c}$ using the relation between the set of $c$-charmed and $c'$-charmed roots in~\Cref{prop:c_charmed}.
\end{itemize}
Since all standard Coxeter elements are conjugate by a sequence of conjugations by initial simple reflections, the theorem, therefore, holds for all standard Coxeter elements.  Our proof strategy is summarized in~\Cref{fig:outline}.

A priori, there is no reason for an arbitrary product of toggles to have well-behaved order.  As a consequence of~\Cref{thm:main_theorem}, we have the following surprising result.
\begin{corollary}
	For any standard Coxeter element $c \in \S_n$, the order of $\Krow_c$ on $\NN(W)$ is $2h$.
\end{corollary}

\begin{remark}
	The statement of the main theorem in Armstrong-Stump-Thomas~\cite{armstrong2013uniform} can be obtained from the statement of our~\Cref{thm:main_theorem} by replacing $\Krow_c$ by $\row$, and replacing the symmetric group by any finite Weyl group. In particular, in the result from~\cite{armstrong2013uniform}, only one toggle sequence is ever employed in the poset of positive roots.   Thus, our main theorem is much richer than the~\cite{armstrong2013uniform} result (when restricted to the symmetric group) because we consider a different toggle sequence for each standard Coxeter element.  Indeed, the fact that the Coxeter element only plays a role on the noncrossing side of the bijection means that in~\cite{armstrong2013uniform}, it is sufficient to consider one preferred Coxeter element (a \emph{bipartite} Coxeter element), and find one corresponding bijection---whereas here we construct truly different bijections for each Coxeter element.
\end{remark}

\subsection{Future work}

Using initial and final simple reflections, the notion of $c$-charmed root can be uniformly extended to all finite Weyl groups.  We expect that our $c$-charmed bijections can also be extended; the first step is to find analogs of~\Cref{prop:nn_induction}---that is, to determine the correct toggle group elements that conjugate $\Krow_c$ to $\Krow_{c'}$ for $c'=scs$.  As a simple first case, we show in~\Cref{sec:typec} that we can use a simple folding argument to extend our work to type $C_n$.

Although there are now explicit constructions of noncrossing and nonnesting partitions for Weyl groups and any rational parameter $p$ coprime to $h$~\cite{galashin2022rational,haiman1994conjectures},  only in type $A$ for the linear Coxeter element do we have explicit diagrammatic models coming from rational Dyck paths~\cite{ArmstrongRhoadesWilliams2013,bodnar2016cyclic,bodnar2018rational,bodnar2019rational}.  It would be natural to try to extend our charmed bijections to these rational models.

\begin{figure}[htbp]
	\begin{tikzpicture}[scale=4.3]
		\node[align=center,rectangle,draw] (00) at (0,0) {$\NN(\S_n)$\\ Def.~\ref{def:nonnesting}};
		\node[align=center,rectangle,draw] (30) at (3,0) {$\NN(\S_n)$\\ Def.~\ref{def:nonnesting}};
		\node[align=center,rectangle,draw] (11) at (1,1) {$\NN(\S_n)$\\ Def.~\ref{def:nonnesting}};
		\node[align=center,rectangle,draw] (21) at (2,1) {$\NN(\S_n)$\\ Def.~\ref{def:nonnesting}};
		\node[align=center,rectangle,draw] (03) at (0,3) {$\NC(\S_n,c)$ \\ Def.~\ref{def:noncrossing}};
		\node[align=center,rectangle,draw] (33) at (3,3) {$\NC(\S_n,c')$\\ Def.~\ref{def:noncrossing}};
		\node[align=center,rectangle,draw] (12) at (1,2) {$\NC(\S_n,c)$\\ Def.~\ref{def:noncrossing}};
		\node[align=center,rectangle,draw] (22) at (2,2) {$\NC(\S_n,c')$\\ Def.~\ref{def:noncrossing}};

		\node at (1.5,2.5) {Prop.~\ref{prop:nc_induction}};
		\node at (1.5,0.5) {Prop.~\ref{prop:nn_induction}};
		\node at (1.5,1.5) {Prop.~\ref{prop:inductive step}};
		\node at (.5,1.5) {Sec.~\ref{sec:cambrian_induction}};
		\node at (2.5,1.5) {Thm.~\ref{thm:charm}};

		\draw[->] (00) -- node[rectangle,fill=white,align=center] {$\beta_{c,k}$ \\ Eq.~\eqref{eq:beta}} (30);
		\draw[->] (11) -- node[rectangle,fill=white,align=center] {$\beta_{c,k}$\\ Eq.~\eqref{eq:beta}} (21);
		\draw[->] (12) -- node[rectangle,fill=white,align=center] {$\alpha_{c,k}$ \\ Eq.~\eqref{eq:alpha}} (22);
		\draw[->] (03) -- node[rectangle,fill=white,align=center] {$\alpha_{c,k}$ \\ Eq.~\eqref{eq:alpha}} (33);
		\draw[->] (03) -- node[rectangle,fill=white,align=center] {$\Charm_c$ \\ Def.~\ref{def:bij}} (00);
		\draw[->] (12) -- node[rectangle,fill=white,align=center] {$\Charm_c$\\ Def.~\ref{def:bij}} (11);
		\draw[->] (33) -- node[rectangle,fill=white,align=center] {$\Charm_{c'}$\\ Def.~\ref{def:bij}} (30);
		\draw[->] (22) -- node[rectangle,fill=white,align=center] {$\Charm_{c'}$\\ Def.~\ref{def:bij}} (21);
		\draw[->] (12) -- node[rectangle,fill=white,align=center] {$\Krew_c$ \\ Def.~\ref{def:krew}} (03);
		\draw[->] (22) -- node[rectangle,fill=white,align=center] {$\Krew_{c'}$\\ Def.~\ref{def:krew}} (33);
		\draw[->] (11) -- node[rectangle,fill=white,align=center] {$\Krow_c$ \\ Def.~\ref{def:krow}} (00);
		\draw[->] (21) -- node[rectangle,fill=white,align=center] {$\Krow_{c'}$\\ Def.~\ref{def:krew}} (30);

	\end{tikzpicture}
	\caption{The commutative cube underlying the structure of the paper.  Each vertex of the cube is the set of noncrossing partitions with respect to some Coxeter element, or the set of nonnesting partitions.  The edges of the cube are bijections between these sets.  Each face of the cube is commutative---\Cref{thm:charm} uses Cambrian induction to infer commutativity of the rightmost face from commutativity of all other faces.  \Cref{ex:running_example} replaces each vertex of this commutative cube by an element of the corresponding set used in the running example of the paper.}
	\label{fig:outline}
\end{figure}
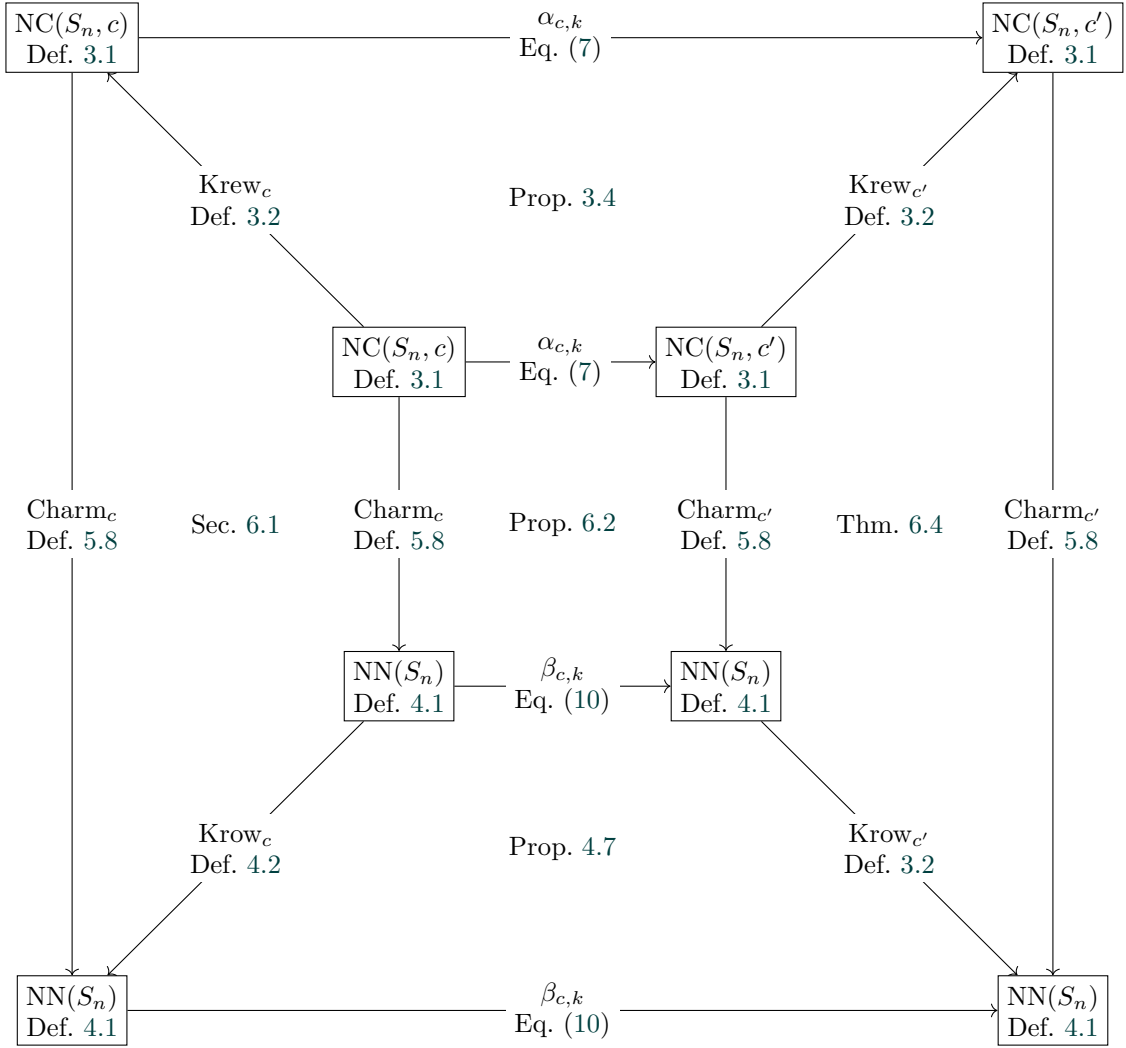

\subsection*{Acknowledgements}
B.D.~thanks the Institut des Sciences Math\'ematiques of Canada for its partial support.
G.F. was partially supported by the Canada Research Chairs program.
F.S.-A.\ acknowledges the financial support from the Austrian Science Foundation FWF, grant J 4387.
H.T.~was partially supported by the Canada Research Chairs program and an NSERC Discovery Grant.
N.W.~was partially supported by Simons Foundation Collaboration Grant No.~585380.

\section{Background}
\label{sec:background}

Write $[n] \coloneqq \{1,2,\ldots,n\}$; we work modulo $n$ so that for $0< i<n$, we may write $-i$ for $n-i$ and $n+1$ for $1$.  Since this paper only deals with the permutation group $\S_n$, we will specialize our explanations unless the general case causes no additional complications.  For a more general introduction on Coxeter systems, we refer the reader to~\cite{Humphreys1990}.%

\subsection{Coxeter groups}
Fix $(W,S)$ a finite Coxeter system, where $S$ is the set of \defn{simple reflections}.  For $W=\S_n$, we reserve the symbols $s_i$ for the simple transpositions $(i,i+1)$.

The \defn{length} of an element $w \in W$ is \[ \ell_S(w) \coloneqq \min \{\ell \in \mathbb{N} \mid w = t_1 \cdots t_\ell \text{ with } t_i \in S \}. \]  For $W=\S_n$ and $w \in W$, $\ell_S(w)$ is equal to the number of \defn{inversions} of $w$: \[\ell_S(w) = \{(i,j) \mid 1\leq i<j\leq n \text{ and } w(i)>w(j)\}.\]
For $w = t_1 \cdots t_\ell$ with $t_i \in S$ and $\ell = \ell_S(w)$, the expression $\mathsf{w}=[t_1, \ldots, t_k]$ is a \defn{reduced $S$-word} for $w$.
The \defn{(right) weak order} on $W$ is defined by \[ v \leq_S w \iff \ell_S(w) = \ell_S(v) + \ell_S(v^{-1}w).\] There is a unique \defn{longest element} $w_\circ$ of maximum length, which is the reverse permutation in $\S_n$; this longest element is greater than all other elements in weak order.

The \defn{reflections} in $W$ are the elements of the set $T \coloneqq \{ wsw^{-1} \mid s \in \S, \; w \in W \}.$
We use the reflections of $W$ to define the \defn{absolute length} of an element $w \in W$ to be \[ \ell_T(w) \coloneqq \min \{\ell \in \mathbb{N} \mid w = r_1 \cdots r_\ell \colon r_i \in T \}. \] It is immediate that $\ell_T(w) \leq \ell_S(w)$.  The \defn{absolute order} on $W$ is the partial order on $W$ induced by absolute length (in the same way as length induces weak order):
\[ v \leq_T w \iff \ell_T(w) = \ell_T(v) + \ell_T(v^{-1}w).\]

For $\S_{n}$, the set of reflections is the set of all transpositions $(i,j)$.  The absolute length of a permutation $w \in \S_{n}$ is $n$ minus the number of cycles of $w$.

The \defn{support} of an element $w \in W$ is the set of simple reflections below $w$ in absolute order: \[\Supp(w) = \{s \in S \mid s \leq_T w\}.\]

\subsection{Roots}
Let $\{e_i\}_{i=1}^n$ be the standard basis of $\mathbb{R}^n$.  The \defn{positive roots} of type $A_{n-1}$ are the following normal vectors to the reflecting hyperplanes in the braid arrangement \begin{equation}\label{eq:roots}
	\alpha_{i,j} \coloneqq e_j-e_i \text{ with } 1 \leq i < j \leq n.
\end{equation}

To reduce subscripts, we will often denote $\alpha_{i,j}$ by its corresponding reflection $(i,j) \in S_n$; we will always use the convention that $i<j$ with this notation.  The \defn{root poset} of type $A_{n-1}$ is the poset $\Phi^+ \coloneqq \Phi^+(A_{n-1})$ defined on the positive roots with covering relations $(i+1,j) \lessdot (i,j) \lessdot (i,j+1)$.  The minimal elements of $\Phi^+$ are the \defn{simple roots} $\alpha_i$ for $1 \leq i < n$, which we also denote by $(i,i+1)$.  The \defn{support} of a root $(i,j)$ is the set of simple roots $\Supp(i,j)$ less than $(i,j)$ in the root poset.  The root poset of type $A_8$ is illustrated in~\Cref{fig: root poset}.%

\begin{figure}[htbp]
	\begin{center}
		\begin{tikzpicture}[scale=.95]
			\pgfmathtruncatemacro{\n}{9};
			\pgfmathtruncatemacro{\none}{\n-1};
			\pgfmathtruncatemacro{\ntwo}{\n-2};
			\pgfmathtruncatemacro{\nthree}{\n-3};
			\foreach \yy in {0,...,\ntwo}{
					\pgfmathtruncatemacro{\y}{\none-\yy};
					\foreach \x in {1,...,\y}{
							\pgfmathtruncatemacro{\i}{2*\x+\yy};
							\pgfmathtruncatemacro{\j}{\yy};
							\pgfmathtruncatemacro{\k}{\x+\yy+1};
							\node (\x\k) at (\i,\j) {$(\x\k)$};
						};
				};
			\foreach \yy in {0,...,\nthree}{
					\pgfmathtruncatemacro{\y}{\ntwo-\yy};
					\foreach \x in {1,...,\y}{
							\pgfmathtruncatemacro{\i}{2*\x+\yy};
							\pgfmathtruncatemacro{\j}{\yy};
							\pgfmathtruncatemacro{\k}{\x+\yy+1};
							\pgfmathtruncatemacro{\xone}{\x+1};
							\pgfmathtruncatemacro{\xmone}{\x-2};
							\pgfmathtruncatemacro{\kone}{\k+1};
							\draw (\x\k) -- (\x\kone);
							\draw (\xone\kone) -- (\x\kone);
						};
				};
			\foreach \i in {3,4,7}{
					\foreach \j in {2,5,6,8}{
							\pgfmathtruncatemacro{\x}{min(\i,\j)};
							\pgfmathtruncatemacro{\y}{max(\i,\j)};
							\draw (\x\y) node[opacity=0.5,yshift=-.3em] {\bigheart};
						};
				};
		\end{tikzpicture}
	\end{center}
	\caption{\label{fig: root poset} The Hasse diagram of the positive root poset $\Phi^+$ of type $A_8$.  The label $(ij)$ represents the positive root $e_j - e_i$.  For $c=s_2s_1s_3s_6s_5s_4s_8s_7$, the $c$-charmed roots from~\Cref{def:charmed} are marked using hearts.}
\end{figure}
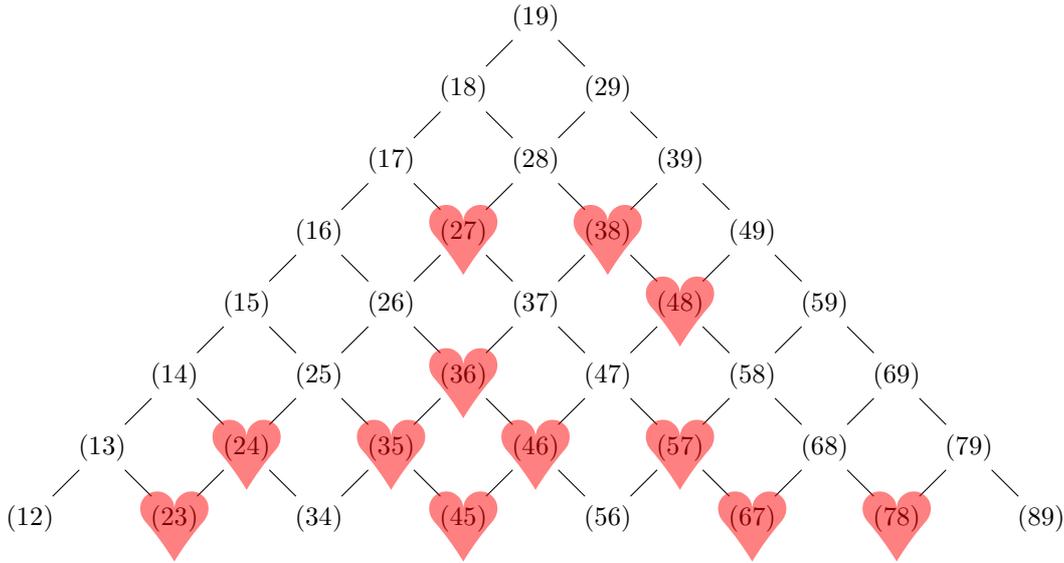

\subsection{Coxeter elements and charmed roots}
\label{sec:coxeter_els}

\begin{definition}
	\label{def:coxeter_element}
	A \defn{standard Coxeter element} $c$ is a product of the simple reflections in any order.
\end{definition}

To ease notation, we reserve the symbols $r_i$ to refer to a fixed ordering $\sc \coloneqq [r_1,r_2,\ldots,r_{n-1}]$ of $S$ and define $c \coloneqq r_1 r_2 \cdots r_{n-1}$.  We also reserve the symbol $s_k = r_1$ for the first simple reflection in the chosen reduced word, and write $c' = r_2 \cdots r_{n-1} r_1$.  We will drop the word ``standard'' when talking about standard Coxeter elements, as we only use standard Coxeter elements in this paper.

We say that a simple reflection $s$ is \defn{initial} in $c$ if $\ell_S(s c) \leq \ell_S(c)$, and that a simple reflection $s$ is \defn{final} in $c$ if $\ell_S(c s) \leq \ell_S(c)$.   If $s=s_k=(k,k+1)$ is initial in $c$, then $c' = scs$ is also a Coxeter element of $\S_n$, and we will denote this by the writing $c \xrightarrow{k} c'$.  It is a basic fact for finite Coxeter groups---relying only on the fact that the underlying Dynkin diagram is a tree---that all standard Coxeter elements are conjugate~\cite[Section 3.16]{Humphreys1990}.

\begin{lemma}[{\cite[Lemma 1.7]{reading2007clusters}}]\label{lem:initial_conj}
	Let $c$ and $c'$ be two Coxeter elements of a finite Coxeter group $W$.  Then $c$ and $c'$ are conjugate by a sequence of conjugations by initial simple reflections: \[c=c^{(0)} \xrightarrow{k_1} c^{(1)} \xrightarrow{k_2} c^{(2)} \xrightarrow{k_3} \cdots \xrightarrow{k_m} c^{(m)} = c'.\]
\end{lemma}

In $\S_n$ and for $1 \leq k < n$, we write $c_k$ for the Coxeter element with $s_k$ as its unique initial simple reflection---that is, $c_k=s_k s_{k+1}\cdots s_{n-1} s_{k-1}s_{k-2} \cdots s_1$.  As a special case, the \defn{linear Coxeter element} $c_1=s_1 s_2\cdots s_{n-1}$ has cycle notation $(1,2,\ldots,n)$.

More generally, it is easy to show that the cycle notation of any Coxeter element in the symmetric group has a particularly simple form: $c \in \S_n$ consists of a single cycle with an initial increasing subsequence starting at $1$ and ending at $n$, followed by a decreasing sequence of the remaining unused entries.%

\begin{proposition}
	An element $w \in \S_n$ is a Coxeter element if and only if it has cycle notation \[w = (\a_1,\a_2, \ldots, \a_m,\a_{m+1} \dots, w_n),\] where $1 = \a_1 < \a_2 < \dots < \a_m = n$ and $n = \a_i > \a_{m+1} > \dots > \a_n>\a_1=1$.
\end{proposition}

Let $c$ be a Coxeter element with cycle notation $(\a_1,\a_2, \ldots, \a_m,\a_{m+1} \dots, \a_n)$, where $1 = \a_1 < \a_2 < \dots < \a_m = n$ and $n = \a_m > \a_{m+1} > \dots > \a_n>\a_1=1$.
Write \begin{align*}
	\lheart_c \coloneqq \{\a_2,\ldots,\a_{m-1}\} \text{  and  }
	\rheart_c \coloneqq \{\a_{m+1},\ldots,\a_n\}.
\end{align*}

\begin{definition}
	\label{def:charmed}
	For $1<i<j<n$, we say that a root $(i,j)$ is \defn{$c$-charmed} if $i \in \lheart_c$ and $j \in \rheart_c$ or if $i \in \rheart_c$ and $j \in \lheart_c$ and \defn{$c$-ordinary} otherwise.  We write $\heart_c$ for the set of $c$-charmed roots.
\end{definition}

\begin{example}
	\label{ex:1}
	Consider the Coxeter element $c=s_2s_1s_3s_6s_5s_4s_8s_7$ in $\S_9$; $s_2$, $s_6$, and $s_8$ are initial in $c$, while $s_1$, $s_4$, and $s_7$ are final.  The element $c$ has cycle notation $(1,3,4,7,9,8,6,5,2)$, so that \[\lheart_c = \{3,4,7\} \text{ and } \rheart_c = \{8,6,5,2\}.\]
\end{example}

As a special case, a simple root $(i,i+1)$ for $1<i<n-1$ is $c$-charmed if and only if it is initial or final in $c$; more precisely, $(i,i+1)$ is initial when $i+1 \in \lheart_c$ and $i \in \rheart_c$, and $(i,i+1)$ is final when $i \in \lheart_c$ and $i+1 \in \rheart_c$.  In figures, we depict $c$-charmed roots with a $\textcolor{red}{\heart}$ and ordinary roots by a circle.

We visualize the cycle notation of $c$ by drawing it as points labeled $w_1,w_2,\ldots,w_n$ %
counter-clockwise around a circle.  We visualize a root $(i,j)$ by connecting the vertices labeled $i$ and $j$ by a line segment.  For any $a,b,c,d \in [n]$, we say that $(a,b)$ \defn{crosses} $(c,d)$ if and only if $(a,b)$ and $(c,d)$ are crossing in their interior (note that $(a,b)$ does not cross itself).  For $1<i<j<n$, it is easy to check that a root $(i,j)$ is $c$-charmed if and only if $(i,j)$ crosses $(i-1,j+1)$.

\begin{example}
	\Cref{fig:c_visualization} illustrates this visualization of the cycle notation $(1,3,4,7,9,8,6,5,2)$ of $c=s_2s_1s_3s_6s_5s_4s_8s_7$.
	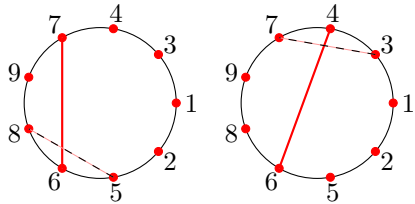
\begin{figure}[htbp]
		\begin{tikzpicture}
			\def\c{{1,3,4,7,9,8,6,5,2}};
			\pgfmathtruncatemacro{\n}{dim(\c)};
			\pgfmathtruncatemacro{\no}{\n-1};
			\draw (0,0) circle (1cm);
			\foreach \i in {0,...,\no}{
					\pgfmathtruncatemacro{\j}{\c[\i]};
					\node (\j) at ({cos(360*\i/\n)},{sin(360*\i/\n)}){};
					\filldraw[red] (\j) circle (1.5pt);
					\node (p\j) at ({1.2*cos(360*\i/\n)},{1.2*sin(360*\i/\n)}){\j};
				};
			\draw[red, thick, fill=red!40!white] (6.center) -- (7.center);
			\draw[black, dashed, fill=red!40!white] (5.center) -- (8.center);
		\end{tikzpicture}
		\begin{tikzpicture}
			\def\c{{1,3,4,7,9,8,6,5,2}};
			\pgfmathtruncatemacro{\n}{dim(\c)};
			\pgfmathtruncatemacro{\no}{\n-1};
			\draw (0,0) circle (1cm);
			\foreach \i in {0,...,\no}{
					\pgfmathtruncatemacro{\j}{\c[\i]};
					\node (\j) at ({cos(360*\i/\n)},{sin(360*\i/\n)}){};
					\filldraw[red] (\j) circle (1.5pt);
					\node (p\j) at ({1.2*cos(360*\i/\n)},{1.2*sin(360*\i/\n)}){\j};
				};
			\draw[red, thick, fill=red!40!white] (4.center) -- (6.center);
			\draw[black, dashed, fill=red!40!white] (3.center) -- (7.center);
		\end{tikzpicture}
		\caption{The visualization of the cycle notation $(1,3,4,7,9,8,6,5,2)$ of $c=s_2s_1s_3s_6s_5s_4s_8s_7$.  {\it Left:} the initial $c$-charmed simple root $(6,7)$ (in red) intersecting the root $(5,8)=(6-1,8+1)$ (dashed).  {\it Right:} the $c$-charmed root $(4,6)$ (in red) intersecting the root $(3,7)=(4-1,6+1)$ (dashed).}
		\label{fig:c_visualization}
	\end{figure}%
\end{example}

For $s=s_k=(k,k+1)$, conjugating $c$ by $s$ interchanges the positions of $k$ and $k+1$ in the cycle notation of $c$.  When $s$ is initial in $c$, then $c'=scs$ is again a Coxeter element and we recall that we write this using the notation $c \xrightarrow{k} c'$.  We obtain the following relationship between $c$-charmed roots and $c'$-charmed roots.  %

\begin{proposition}
	\label{prop:c_charmed}
	Let $c \xrightarrow{k} c'$.  Then for $1<i<j<n$,
	\[ (i,j) \text{ is $c$-charmed if and only if it is }
		\begin{cases}
			\text{$c'$-charmed}  & \text{and } |\{i,j\} \cap \{k,k+1\}| = 0 \mod 2,  \\
			\text{$c'$-ordinary} & \text{ and } |\{i,j\} \cap \{k,k+1\}| = 1 \mod 2.
		\end{cases}
	\]
\end{proposition}

\subsection{Coxeter-sorting words and Auslander--Reiten quivers}
\label{sec:sort_ar}

Let $c \in \S_n$ be a Coxeter element, and fix $\sc = [r_{1},r_2 \ldots, r_{n-1}]$
a reduced word for $c$. Define the \defn{$c$-sorting word} of the long element $w_\circ$ to be the leftmost reduced word $\w_\circ(\sc)$ for $w_\circ$ in $\sc^\infty$.  Up to commutations, $\w_\circ(\sc)$ does not depend on the choice of reduced word for $c$, and so we shall denote it by $\w_\circ(c)$.
Write the $c$-sorting word for the long element $\w_\circ(c) = [t_1,t_2, \ldots, t_N]$, with each $t_i \in S$ a simple reflection and $N=|T|=\binom{n}{2}$ the number of reflections in $\S_n$. For $a = 1,2, \ldots, N$, define the inversion
\begin{equation}
	\label{eq:root_order}
	\alpha^{(a)} \coloneqq t_{1} t_{2} \cdots t_{a-1}(\alpha_{t_a}),
\end{equation}
where $\alpha_{t_a}$ is the simple root corresponding to the simple reflection $t_a$.  Since $w_\circ$ has every positive root as an inversion, each positive root appears exactly once in the inversion sequence $\inv(\w_\circ(c)) \coloneqq [\alpha^{(1)},\alpha^{(2)}, \ldots, \alpha^{(N)}]$.

\begin{example}
	\label{ex:c_sorting}
	For $\sc=[s_2,s_1,s_3,s_6,s_5,s_4,s_8,s_7]$, we have \begin{align*} \w_\circ(c) & = s_2s_1s_3s_6s_5s_4s_8s_7s_2s_1s_3s_6s_5s_4s_8s_7s_2s_1s_3s_6s_5s_4s_8s_7s_2s_1s_3s_6s_5s_4s_8s_7s_2s_6s_5s_8 \text{ and } \\ \inv(\w_\circ(c)) & =[(23), (13), (24), (67), (57), (27), (89), (69), (14), (34), (17), (59),\\ &\phantom{=[} (29), (19), (68), (58), (37), (47), (39), (28), (18), (38), (56), (26),\\ &\phantom{=[} (49), (79), (48), (16), (36), (46), (25), (15), (78), (35), (45), (12)].
	\end{align*}
\end{example}

It is helpful to arrange the letters in the Coxeter-sorting word $\w_\circ(c)= [t_1,t_2, \ldots, t_N]$ in a quiver $\ARc$---if we replace each letter $t_a$ by the reflection corresponding to the positive root $\alpha^{(a)}$ in the inversion sequence, this is known as the \defn{Auslander--Reiten quiver} $\ARc$~\cite[Sections 3.1.1 and 3.1.2.2]{schiffler2014quiver}. %
A combinatorial description of $\ARc$ is given as follows:
\begin{itemize}
	\item the vertices are the roots $(i,j)$ for $1 \leq i <j \leq n$;
	\item if $(i,i+1)$ is initial in $c$, then $(i,i+1)$ appears as the leftmost vertex in the $i$th row from the top;
	\item if $(i,i+1)$ is final in $c$, then $(i,i+1)$ appears as the rightmost vertex in the $i$th row from the bottom.
	\item for $i<j$ we draw a south-east arrow $(i,j) \longrightarrow (i,c(j))$ if $i<c(j)$; and
	\item for $i<j$ we draw a north-east arrow $(i,j) \longrightarrow (c(i),j)$ if $c(i)<j$.
\end{itemize}
It will be useful to consider the example Auslander--Reiten quiver shown in Figure \ref{fig:tri9} while digesting this description.

Representation-theoretically, the vertices of $\ARc$ are the indecomposable representations of $\Qc$, but these are naturally labelled by positive roots using Gabriel's theorem. We now provide some more details on the structure of $\ARc$.

It follows immediately from the above description that the labels $(i,j)$ appearing on each SW-NE diagonal have $j$ fixed, while those appearing
on each NW-SE diagonal have $i$ fixed; we therefore index the diagonals by the corresponding label. One SW-NE diagonal appears for each value of $j$ except 1; one NW-SE diagonal appears for each value of $i$ except $n$.
It is also immediate that the NW-SE diagonals appear in left-to-right order given by $c$, and similarly for the SW-NE diagonals. For $j$ in $\lheart\cap [2,n-1]$, the righthand end of the SW-NE diagonal labelled by $j$ is $(c^{-1}(j),j)$. The lefthand end of the SW-NE diagonal
labelled by $j$ begins one step to the right from $(c^{-1}(j),j)$, at $(j,c(j))$. For $i$ in $\rheart\cap [2,n-1]$, the righthand end of the NW-SE diagonal labelled by $i$ ends at $(i,c^{-1}(i))$ and the SW-NE-diagonal labelled by $i$ starts one step to the right, at $(c(i),i)$.

Let $s=s_k$ be initial in $c$. It is convenient to write $c=(a_1, \dots a_n)$ where $a_1=k$. We have that $\{a_1,\dots,a_k\}=\{1,\dots,k\}$, $a_{k+1}=k+1$, and $\{a_{k+1},\dots,a_n\}=\{k+1,\dots,n\}$. The $(k+1)$ SW-NE diagonals starting with the diagonal containing containing $(k,k+1)$ are numbered $k=a_1,\dots,a_k,a_{k+1}=k+1$ in that order. The $(n-k+2)$ NW-SE diagonals starting with the diagonal containing $(k,k+1)$ are numbered $k+1=a_{k+1},\dots,a_n,a_1=k$ in that order. In addition to Figure \ref{fig:tri9}, it will also be useful to consult Figure \ref{fig: AR arround k}, which provides a schematic description of a generic AR quiver with $s_k$ initial.

The diagonals labelled by $k$ and $k+1$ form a rectangle with three missing corners. We have already observed that the top and bottom corners of the rectangle are missing; these are the positions where the two diagonals labelled $k$, or the two diagonals labelled $k+1$, should cross. We see that the righthand corner is also missing; it is at the intersection of the SW-NE diagonal labelled $k$ and the NW-SE diagonal labelled $k+1$, so that its label would be $(k+1,k)$, which is not a valid label; however, the labels to its NW and SW are the (valid) labels $(c^{-1}(k+1),k)$ and $(k+1,c^{-1}(k))$, so all the rest of the roots along the edges of the rectangle are indeed present.

We will need the following lemma, which characterizes the labels whichappear outside this rectangle.

\begin{lemma}
	\label{lem:outside}
	Let $s=s_k$ be initial in $c$.
	In $\ARc$, all the roots above the SW-NE diagonal labelled $k+1$ or below the SW-NE diagonal labelled by $k$ have both entries less than $k$.
	Similarly, all the roots above the NW-SE diagonal labelled by $k+1$ or below the NW-SE diagonal labelled $k$ have both entries greater than $k+1$.
\end{lemma}

\begin{proof}
	Recall that each set of diagonals appears in the order given by $c$. This implies that the NW-SE diagonals with labels greater than $k$ are between the one labelled $k+1$ and the one labelled $k$. The SW-NE diagonals less than $k+1$ are between the one labelled $k+1$ and the one labelled $k$.
	It follows that the roots in the region above the SW-NE diagonal labelled by $k+1$ lie on diagonals which are both labelled by numbers less than $k$. A similar argument applies to the three other regions.
\end{proof}

Let $c'=s_kcs_k$, so $c \xrightarrow{k} c'$. In this case, we can relate the Auslander--Reiten quivers $\AR(c)$ and $\AR(c')$.

\begin{proposition}
	\label{prop:ARc}
	Let $c \xrightarrow{k} c'$. Then $\AR(c')$ is obtained from $\ARc$ by replacing every label $k$ by $k+1$ and vice versa, while also removing the root labelled $(k,k+1)$ and replacing it where the SW-NE diagonal labelled $k+1$ and the NW-SE diagonal labelled $k$ now cross.
\end{proposition}

\begin{proof}
	Another way to describe the transformation of the roots (other than $(k,k+1)$) is that we replace the label $(i,j)$ by $(s(i),s(j))$. Since $c'=scs$, if $c(j)=\tilde j$, then $c'(s(j))=s(\tilde j)$. All the edges of the quiver except those involving $(k,k+1)$ will remain, because the action of $s$ will not change whether or not the inequalities from the description of the Auslander--Reiten quiver hold. In $\ARcp$, the label $(k,k+1)$ must lie at the intersection of the SW-NE diagonal labeled by $k+1$ and the NW-SE diagonal labelled by $k$; as we have already observed, this position was empty in the Auslander--Reiten quiver of $c$.
\end{proof}

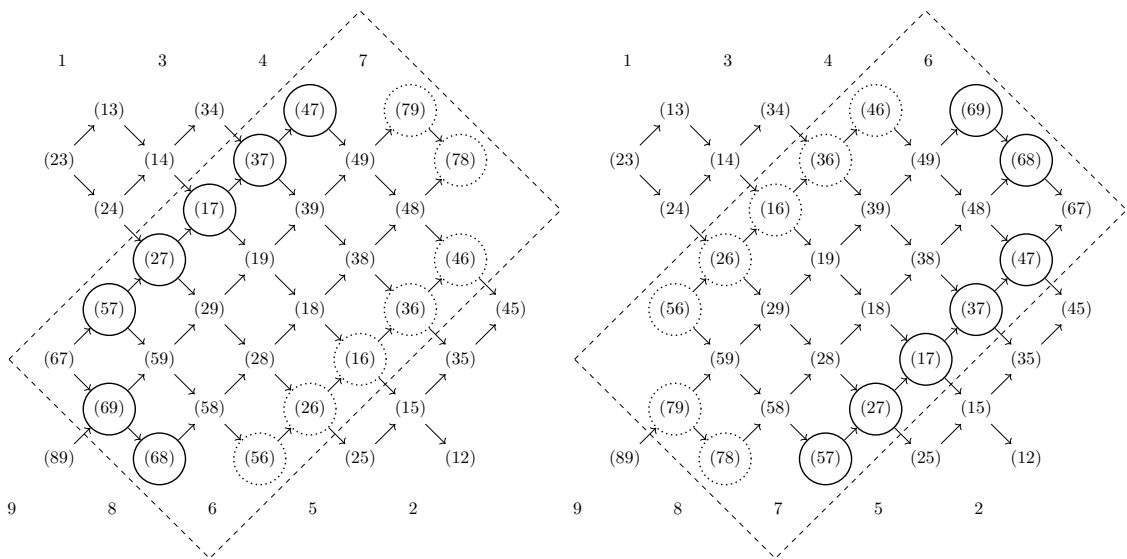
\begin{figure}[htbp]
	\[
		\scalebox{0.66}{
			\begin{tikzpicture}
				\node (1) at (0,8) {{  $1$}};
				\node (3) at (2,8) {{  $3$}};
				\node (4) at (4,8) {{  $4$}};
				\node (7) at (6,8) {{  $7$}};
				\node (9) at (-1,-1) {{  $9$}};
				\node (8) at (1,-1) {{  $8$}};
				\node (6) at (3,-1) {{  $6$}};
				\node (5) at (5,-1) {{  $5$}};
				\node (2) at (7,-1) {{  $2$}};

				\draw[dashed] (-1,2) -- (6,9) -- (10,5) -- (3,-2) -- (-1,2);

				\node (67) at (0,2) {(67)};
				\node[circle,draw,thick] (57) at (1,3) {(57)};
				\node[circle,draw,thick] (27) at (2,4) {(27)};
				\node[circle,draw,thick] (17) at (3,5) {(17)};
				\node[circle,draw,thick] (37) at (4,6) {(37)};
				\node[circle,draw,thick] (47) at (5,7) {(47)};
				\node[circle,draw,thick] (69) at (1,1) {(69)};
				\node (59) at (2,2) {(59)};
				\node (29) at (3,3) {(29)};
				\node (19) at (4,4) {(19)};
				\node (39) at (5,5) {(39)};
				\node (49) at (6,6) {(49)};
				\node[circle,draw,thick] (68) at (2,0) {(68)};
				\node (58) at (3,1) {(58)};
				\node (28) at (4,2) {(28)};
				\node (18) at (5,3) {(18)};
				\node (38) at (6,4) {(38)};
				\node (48) at (7,5) {(48)};
				\node[circle,draw,thick,dotted] (78) at (8,6) {(78)};
				\node (89) at (0,0) {(89)};
				\node[circle,draw,thick,dotted] (79) at (7,7) {(79)};
				\node[circle,draw,thick,dotted] (56) at (4,0) {(56)};
				\node (25) at (6,0) {(25)};
				\node (12) at (8,0) {(12)};
				\node (13) at (1,7) {(13)};
				\node (34) at (3,7) {(34)};
				\node[circle,draw,thick,dotted] (26) at (5,1) {(26)};
				\node (15) at (7,1) {(15)};
				\node (23) at (0,6) {(23)};
				\node (14) at (2,6) {(14)};
				\node[circle,draw,thick,dotted] (16) at (6,2) {(16)};
				\node (35) at (8,2) {(35)};
				\node (24) at (1,5) {(24)};
				\node[circle,draw,thick,dotted] (36) at (7,3) {(36)};
				\node (45) at (9,3) {(45)};
				\node[circle,draw,thick,dotted] (46) at (8,4) {(46)};

				\draw[->] (13)--(14);\draw[->] (14)--(34);\draw[->] (14)--(17);\draw[->] (15)--(35);\draw[->] (15)--(12);\draw[->] (16)--(36);\draw[->] (16)--(15);\draw[->] (17)--(37);\draw[->] (17)--(19);\draw[->] (18)--(38);\draw[->] (18)--(16);\draw[->] (19)--(39);\draw[->] (19)--(18);\draw[->] (23)--(13);\draw[->] (23)--(24);\draw[->] (24)--(14);\draw[->] (24)--(27);\draw[->] (25)--(15);\draw[->] (26)--(16);\draw[->] (26)--(25);\draw[->] (27)--(17);\draw[->] (27)--(29);\draw[->] (28)--(18);\draw[->] (28)--(26);\draw[->] (29)--(19);\draw[->] (29)--(28);\draw[->] (34)--(37);\draw[->] (35)--(45);\draw[->] (36)--(46);\draw[->] (36)--(35);\draw[->] (37)--(47);\draw[->] (37)--(39);\draw[->] (38)--(48);\draw[->] (38)--(36);\draw[->] (39)--(49);\draw[->] (39)--(38);\draw[->] (46)--(45);\draw[->] (47)--(49);\draw[->] (48)--(78);\draw[->] (48)--(46);\draw[->] (49)--(79);\draw[->] (49)--(48);\draw[->] (56)--(26);\draw[->] (57)--(27);\draw[->] (57)--(59);\draw[->] (58)--(28);\draw[->] (58)--(56);\draw[->] (59)--(29);\draw[->] (59)--(58);\draw[->] (67)--(57);\draw[->] (67)--(69);\draw[->] (68)--(58);\draw[->] (69)--(59);\draw[->] (69)--(68);\draw[->] (79)--(78);\draw[->] (89)--(69);
			\end{tikzpicture}%
			\begin{tikzpicture}
				\node (1) at (0,8) {{  $1$}};
				\node (3) at (2,8) {{  $3$}};
				\node (4) at (4,8) {{  $4$}};
				\node (7) at (6,8) {{  $6$}};
				\node (9) at (-1,-1) {{  $9$}};
				\node (8) at (1,-1) {{  $8$}};
				\node (6) at (3,-1) {{  $7$}};
				\node (5) at (5,-1) {{  $5$}};
				\node (2) at (7,-1) {{  $2$}};
				\draw[dashed] (-1,2) -- (6,9) -- (10,5) -- (3,-2) -- (-1,2);
				\node (67) at (9,5) {(67)};
				\node[circle,draw,thick,dotted] (57) at (1,3) {(56)};
				\node[circle,draw,thick,dotted] (27) at (2,4) {(26)};
				\node[circle,draw,thick,dotted] (17) at (3,5) {(16)};
				\node[circle,draw,thick,dotted] (37) at (4,6) {(36)};
				\node[circle,draw,thick,dotted] (47) at (5,7) {(46)};

				\node[circle,draw,thick,dotted] (69) at (1,1) {(79)};
				\node (59) at (2,2) {(59)};
				\node (29) at (3,3) {(29)};
				\node (19) at (4,4) {(19)};
				\node (39) at (5,5) {(39)};

				\node (49) at (6,6) {(49)};

				\node[circle,draw,thick,dotted] (68) at (2,0) {(78)};
				\node (58) at (3,1) {(58)};
				\node (28) at (4,2) {(28)};
				\node (18) at (5,3) {(18)};
				\node (38) at (6,4) {(38)};
				\node (48) at (7,5) {(48)};
				\node[circle,draw,thick] (78) at (8,6) {(68)};
				\node (89) at (0,0) {(89)};
				\node[circle,draw,thick] (79) at (7,7) {(69)};

				\node[circle,draw,thick] (56) at (4,0) {(57)};
				\node (25) at (6,0) {(25)};
				\node (12) at (8,0) {(12)};
				\node (13) at (1,7) {(13)};
				\node (34) at (3,7) {(34)};

				\node[circle,draw,thick] (26) at (5,1) {(27)};
				\node (15) at (7,1) {(15)};
				\node (23) at (0,6) {(23)};
				\node (14) at (2,6) {(14)};

				\node[circle,draw,thick] (16) at (6,2) {(17)};
				\node (35) at (8,2) {(35)};
				\node (24) at (1,5) {(24)};

				\node[circle,draw,thick] (36) at (7,3) {(37)};
				\node (45) at (9,3) {(45)};

				\node[circle,draw,thick] (46) at (8,4) {(47)};

				\draw[->] (13)--(14);\draw[->] (14)--(34);\draw[->] (14)--(17);\draw[->] (15)--(35);\draw[->] (15)--(12);\draw[->] (16)--(36);\draw[->] (16)--(15);\draw[->] (17)--(37);\draw[->] (17)--(19);\draw[->] (18)--(38);\draw[->] (18)--(16);\draw[->] (19)--(39);\draw[->] (19)--(18);\draw[->] (23)--(13);\draw[->] (23)--(24);\draw[->] (24)--(14);\draw[->] (24)--(27);\draw[->] (25)--(15);\draw[->] (26)--(16);\draw[->] (26)--(25);\draw[->] (27)--(17);\draw[->] (27)--(29);\draw[->] (28)--(18);\draw[->] (28)--(26);\draw[->] (29)--(19);\draw[->] (29)--(28);\draw[->] (34)--(37);\draw[->] (35)--(45);\draw[->] (36)--(46);\draw[->] (36)--(35);\draw[->] (37)--(47);\draw[->] (37)--(39);\draw[->] (38)--(48);\draw[->] (38)--(36);\draw[->] (39)--(49);\draw[->] (39)--(38);\draw[->] (46)--(45);\draw[->] (47)--(49);\draw[->] (48)--(78);\draw[->] (48)--(46);\draw[->] (49)--(79);\draw[->] (49)--(48);\draw[->] (56)--(26);\draw[->] (57)--(27);\draw[->] (57)--(59);\draw[->] (58)--(28);\draw[->] (58)--(56);\draw[->] (59)--(29);\draw[->] (59)--(58);\draw[->] (46)--(67);\draw[->] (78)--(67);\draw[->] (68)--(58);\draw[->] (69)--(59);\draw[->] (69)--(68);\draw[->] (79)--(78);\draw[->] (89)--(69);
			\end{tikzpicture}}
	\]

	\caption{The Coxeter element $c=s_2s_1s_3s_6s_5s_4s_8s_7$ has cycle notation $(1,3,4,7,9,8,6,5,2)$.  The simple reflection $s=s_6$ is initial in $c$, and we set $c'=scs$, with cycle notation $(1,3,4,6,9,8,7,5,2)$.  The quivers $\AR(c)$ ({\it left}) and $\AR(c')$ ({\it right}) are illustrated above.  The roots in $\T_6$ defined in~\Cref{eq:tk} are marked by solid circles in $\AR(c)$ and $\AR(c')$, while the roots in $\sT_6$ are marked by dotted circles.}%
	\label{fig:tri9}
\end{figure}

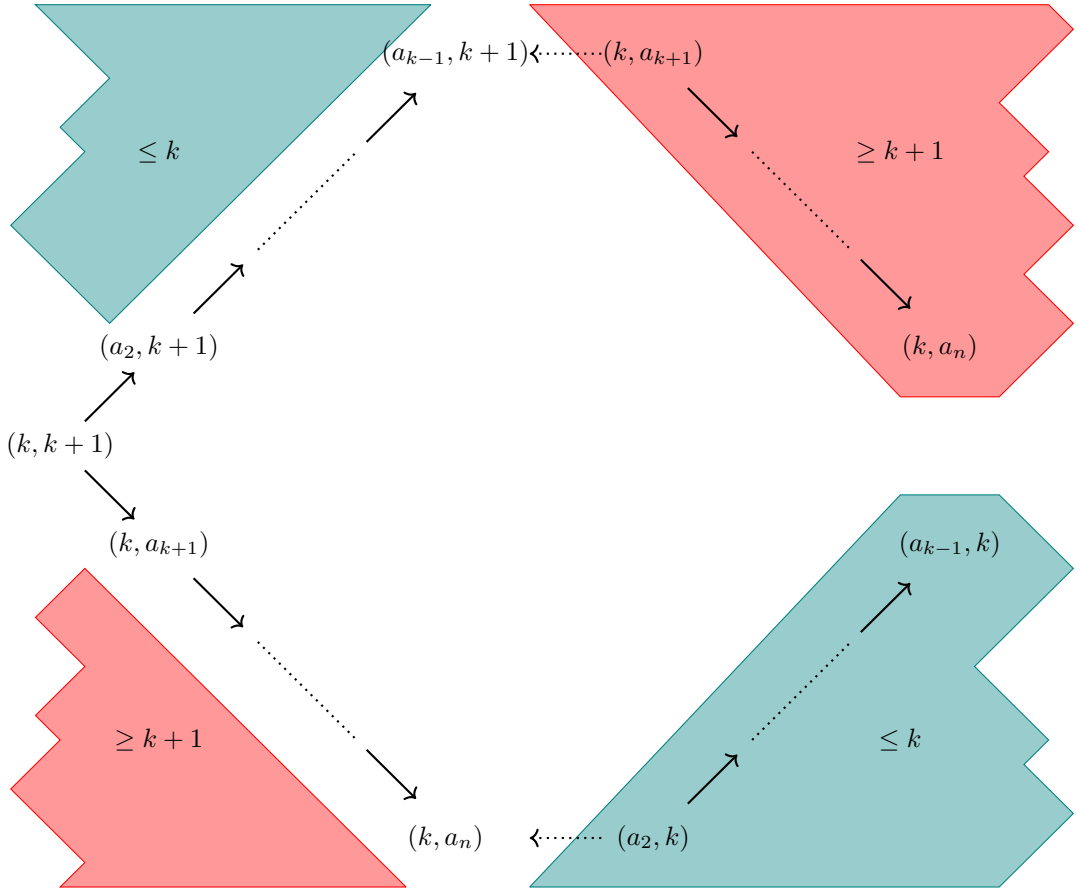
\begin{figure}[htbp]
	\begin{center}
		\begin{tikzpicture}[scale=1.3]
			\draw[teal, fill=teal!40!white] (3.75,4.5) -- (.5,1.25) -- (-.5,2.25) -- (.25,3) -- (0,3.25) -- (.5,3.75) -- (-.25,4.5) -- (3.75,4.5);
			\node at (1,3) {$\leq k$};%
			\draw[red, fill=red!40!white] (3.5,-4.5) -- (0.25,-1.25) -- (-.25,-1.75) -- (.25,-2.25) -- (-.25,-2.75) -- (0,-3) -- (-.5,-3.5) -- (.25,-4.25) -- (0,-4.5) -- (3.5,-4.5);
			\node at (1,-3) {$\geq k+1$};%

			\draw[red, fill=red!40!white] (4.75,4.5) -- (8.5,.5) -- (9.5,.5) -- (10.25,1.25) -- (9.75,1.75) -- (10.25,2.25) -- (9.75,2.75) -- (10,3) -- (9.5,3.5) -- (10.25,4.25) -- (10,4.5) -- (4.75,4.5);
			\node at (8.5,3) {$\geq k+1$};%
			\draw[teal, fill=teal!40!white] (4.75,-4.5) -- (8.5,-.5) -- (9.5,-.5) -- (10.25,-1.25) -- (9.25,-2.25) -- (10,-3) -- (9.75,-3.25) -- (10.25,-3.75) -- (9.5,-4.5) -- (4.75,-4.5);
			\node at (8.5,-3) {$\leq k$};%

			\node at (4.5,0) {};%
			\node at (0,0) {$(k,k+1)$};
			\draw[->, thick] (.25,.25) -- (.75,.75);
			\node at (1,1) {$(a_2,k+1)$};
			\draw[->, thick] (1.35,1.35) -- (1.85,1.85);

			\draw[dotted, thick] (2,2) -- (3,3);
			\draw[->, thick] (3.1,3.1) -- (3.6,3.6);
			\node at (4,4) {$(a_{k-1},k+1)$};
			\draw[<-, thick, dotted] (4.75,4) -- (5.5,4);
			\node at (6,4) {$(k,a_{k+1})$};
			\draw[->, thick] (6.35,3.65) -- (6.85,3.15);
			\draw[dotted, thick] (7,3) -- (8,2);
			\draw[->, thick] (8.1,1.9) -- (8.6,1.4);
			\node at (8.9,1) {$(k,a_{n})$};

			\draw[->, thick] (0.25,-.25) -- (0.75,-0.75);
			\node at (1,-1) {$(k,a_{k+1})$};
			\draw[->, thick] (1.35,-1.35) -- (1.85,-1.85);
			\draw[dotted, thick] (2,-2) -- (3,-3);
			\draw[->, thick] (3.1,-3.1) -- (3.6,-3.6);
			\node at (3.9,-4) {$(k,a_n)$};
			\draw[<-, thick, dotted] (4.75,-4) -- (5.5,-4);
			\node at (6,-4) {$(a_{2},k)$};
			\draw[->, thick] (6.35,-3.65) -- (6.85,-3.15);
			\draw[dotted, thick] (7,-3) -- (8,-2);
			\draw[->, thick] (8.1,-1.9) -- (8.6,-1.4);
			\node at (9,-1) {$(a_{k-1},k)$};
		\end{tikzpicture}
	\end{center}
	\caption{This provides a generic description of $\ARc$ where $s_k$ is initial in $c$. Here $c$ has cycle notation $(a_1,\ldots,a_k,a_{k+1},\ldots,a_{n})$ with $a_1=k$, $a_{k+1}=k+1$, $\{a_1,\ldots,a_k\} = \{1,\ldots,k\}$, and $\{a_{k+1},\ldots,a_{n}\} = \{k+1,\ldots,n\}$. The regions shaded in green consist of roots with
	both labels $\leq k$, while the regions shaded in red consist of roots with both labels $\geq k+1$, as guaranteed by Lemma \ref{lem:outside}.}
	\label{fig: AR arround k}
\end{figure}

\section{Noncrossing partitions and the Kreweras complement}
\label{sec:comb}

In this section, we define \emph{noncrossing partitions} and the \emph{Kreweras complement}.

\subsection{Noncrossing partitions}

\begin{definition}
	\label{def:noncrossing}
	Let $(W,S)$ be a Coxeter system and $c$ be a Coxeter element. The \defn{$c$-noncrossing partition lattice} is the absolute order interval \[\NC(W,c) \coloneqq [1,c]_T = \{ \nc \in W \mid 1 \leq_T \nc \leq_T c \}.\]  The \defn{support} of a noncrossing partition $\nc \in \NC(W,c)$ is the set $\Supp(\nc)$ of simple reflections required to write a reduced word in simple reflections for $\nc$.
\end{definition}

We now recall the usual graphical realization of noncrossing partitions in $\S_n$.
When $c$ is a Coxeter element of $\S_n$ with cycle notation $(a_1,a_2, \ldots, a_n)$, we depict a $c$-noncrossing partition $\nc \in \NC(\S_n,c)$ using the visualization of the cycle notation $(a_1,a_2, \ldots, a_n)$ by drawing the convex hull of the sets of vertices which belong to the same cycle of $\nc$.  Examples are illustrated on the left and right of~\Cref{fig:Kreweras}.

As the name suggests, the noncrossing partitions $\NC(\S_n,c)$ are exactly the set partitions of $[n]$ with the property that the convex hulls of their blocks are disjoint using the cyclic ordering $(a_1,a_2, \ldots, a_n)$; the absolute order interval is obtained by partially ordering these set partitions by reverse refinement.

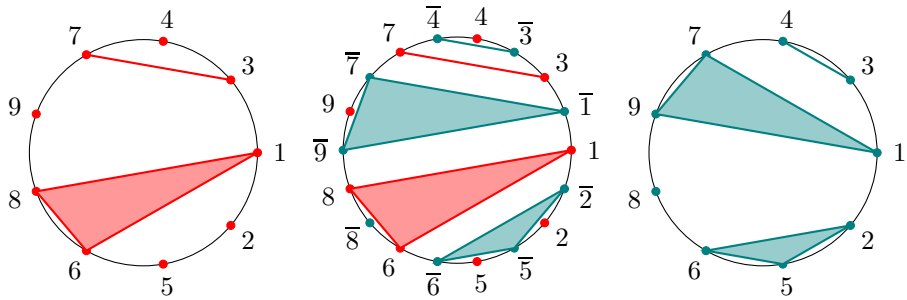
\begin{figure}[htbp]
	\begin{tikzpicture}[scale=1.5]
		\def\c{{1,3,4,7,9,8,6,5,2}};
		\pgfmathtruncatemacro{\n}{dim(\c)};
		\pgfmathtruncatemacro{\no}{\n-1};
		\draw (0,0) circle (1cm);
		\foreach \i in {0,...,\no}{
				\pgfmathtruncatemacro{\j}{\c[\i]};
				\node (\j) at ({cos(360*\i/\n)},{sin(360*\i/\n)}){};
				\filldraw[red] (\j) circle (1pt);
				\node (p\j) at ({1.2*cos(360*\i/\n)},{1.2*sin(360*\i/\n)}){\j};
			};
		\draw[red, thick, fill=red!40!white] (1.center) -- (6.center) -- (8.center) -- (1.center);
		\draw[red, thick, fill=red!40!white] (3.center) -- (7.center);
	\end{tikzpicture}
	\begin{tikzpicture}[scale=1.5]
		\def\c{{1,3,4,7,9,8,6,5,2}};
		\pgfmathtruncatemacro{\n}{dim(\c)};
		\pgfmathtruncatemacro{\no}{\n-1};
		\draw (0,0) circle (1cm);
		\foreach \i in {0,...,\no}{
				\pgfmathtruncatemacro{\j}{\c[\i]};
				\node (\j) at ({cos(360*(2*\i)/(2*\n))},{sin(360*(2*\i)/(2*\n))}){};
				\filldraw[red] (\j) circle (1pt);
				\node (p\j) at ({1.2*cos(360*(2*\i)/(2*\n))},{1.2*sin(360*(2*\i)/(2*\n))}){\j};
				\node (a\j) at ({cos(360*(2*\i+1)/(2*\n))},{sin(360*(2*\i+1)/(2*\n))}){};
				\filldraw[teal] (a\j) circle (1pt);
				\node (ap\j) at ({1.2*cos(360*(2*\i+1)/(2*\n))},{1.2*sin(360*(2*\i+1)/(2*\n))}){$\overline{\j}$};
			};
		\draw[red, thick, fill=red!40!white] (1.center) -- (6.center) -- (8.center) -- (1.center);
		\draw[red, thick, fill=red!40!white] (3.center) -- (7.center);
		\draw[teal, thick, fill=teal!40!white] (a1.center) -- (a7.center) -- (a9.center) -- (a1.center);
		\draw[teal, thick, fill=teal!40!white] (a2.center) -- (a5.center) -- (a6.center) -- (a2.center);
		\draw[teal, thick, fill=teal!40!white] (a3.center) -- (a4.center);
	\end{tikzpicture}
	\begin{tikzpicture}[scale=1.5]
		\def\c{{1,3,4,7,9,8,6,5,2}};
		\pgfmathtruncatemacro{\n}{dim(\c)};
		\pgfmathtruncatemacro{\no}{\n-1};
		\draw (0,0) circle (1cm);
		\foreach \i in {0,...,\no}{
				\pgfmathtruncatemacro{\j}{\c[\i]};
				\node (\j) at ({cos(360*\i/\n)},{sin(360*\i/\n)}){};
				\filldraw[teal] (\j) circle (1pt);
				\node (p\j) at ({1.2*cos(360*\i/\n)},{1.2*sin(360*\i/\n)}){\j};
			};
		\draw[teal, thick, fill=teal!40!white] (1.center) -- (7.center) -- (9.center) -- (1.center);
		\draw[teal, thick, fill=teal!40!white] (2.center) -- (5.center) -- (6.center) -- (2.center);
		\draw[teal, thick, fill=teal!40!white] (3.center) -- (4.center);
	\end{tikzpicture}
	\caption[Set partitions]{Fix the Coxeter element $c=(134798652)$.  {\it Left: } the $c$-noncrossing partition $\nc=\{\{1,6,8\},\{3,7\}\}$. {\it Right: } the $c$-noncrossing partition $\Krew_c(\nc)=\{\{1,7,9\},\{2,5,6\},\{3,4\}\}.$ {\it Middle: } the combinatorial construction of the Kreweras complement.}
	\label{fig:Kreweras}
\end{figure}

\subsection{The Kreweras complement}
\label{sec:Kreweras}

\begin{definition}
	\label{def:krew}
	As in~\Cref{eq:krew}, we define the $c$-Kreweras complement on $\NC(\S_n,c)$ by \[\Krew_c(\nc) = \nc^{-1} c.\]
\end{definition}

We compute the $c$-Kreweras complement using the graphical realization of noncrossing partitions as follows:
between each point $a_i$ and $a_{i+1}$, add a new ``barred'' point labeled $\overline{a}_i$.  Draw the graphical depiction of $\nc \in \NC(\S_n,c)$ using only unbarred points, and define $\overline{\nc}'$ to be the coarsest noncrossing partition using only barred points such that $\overline{\nc}'\cup \nc$ is noncrossing.  Finally, define $\Kr_c(\nc) = \nc'$ to be the combinatorial noncrossing partition obtained from $\overline{\nc}'$ by dropping the bars and rotating clockwise by $\pi/n$ radians.  This procedure is illustrated in~\Cref{fig:Kreweras}.

\begin{proposition}
	The square of the $c$-Kreweras complement $\Kr_c$ acts as clockwise rotation by $2\pi/n$ radians.  Furthermore, $\Krew_c$ is a poset anti-automorphism of $\NC(\S_n,c)$.
\end{proposition}%

\begin{proof} %
	Obviously, $\Kr_c^2$ is the desired rotation, since $\Kr_c$ is an involution up to relabeling $a_i \to a_{i-1}$. Since $\rot$ is bijective, $\Kr_c$ is bijective.
	$\Kr_c$ reverses the ordering of $\NC(\S_n,c)$ because if a given noncrossing partition $\nc_1$ is coarser than some other partition $\nc_2$, then the coarsest element $\overline{\nc}_1'$ such that $\nc_1 \cup \overline{\nc}_1'$ is still noncrossing must be finer than the coarsest element $\overline{\nc}_2'$ such that $\nc_2 \cup \overline{\nc}_2'$ is still noncrossing.
\end{proof}

There is an easy relation between the Kreweras complement on $\NC(\S_n,c)$ and the Kreweras complement on $\NC(\S_n,c')$.  %

\begin{proposition}
	\label{prop:nc_induction}
	Let $c \xrightarrow{k} c'$ and define \begin{align}\label{eq:alpha}\alpha_{c,k} \colon \nonumber\NC(\S_n,c) &\to \NC(\S_n,c') \\ \alpha_{c,k}(\nc) & \coloneqq s_k\nc s_k.\end{align} Then $\Krew_{c'} \circ \alpha_{c,k} = \alpha_{c,k} \circ \Krew_c$.
\end{proposition}

\begin{proof}
	Write $s=s_k$.  Since $\Krew_c(\nc) = \nc^{-1} c$, we have \begin{align*}
		\alpha_{c,k}(\Krew_c(\nc)) & = s\Krew_c(\nc)s=s\nc^{-1}c s = (s\nc^{-1}s) (scs)            \\
		                           & = \Krew_{c'}(s\nc s) = \Krew_{c'}(\alpha_{c,k}(\nc)).\qedhere
	\end{align*}
\end{proof}

\section{Nonnesting partitions and the Kroweras complement}
\label{sec:nnps}

In this section, we define \emph{nonnesting partitions} and the \emph{Kroweras complement}.

\subsection{Nonnesting partitions}
\label{sec:nonnesting}

Let $(P,\leq)$ be a poset. An \defn{order ideal} of $P$ is a subset $I \subseteq P$ such that if $x \in I$ and $y \leq x$, then $y \in I$. Let $J(P)$ denote the set of order ideals of $P$.

\begin{definition}
	\label{def:nonnesting}
	For $W$ a Weyl group with positive root poset $\Phi^+$, the \defn{nonnesting partitions} are the order ideals in the set \[\NN(\S_n) \coloneqq J(\Phi^+).\]  The \defn{support} of a nonnesting partition $\nn \in \NN(W)$ is the set $\Supp(\nn)$ of simple roots that lie in $\nn$.
\end{definition}

\Cref{def:nonnesting} illustrates Incongruity~\eqref{it:2} from~\Cref{sec:history}: unlike $\NC(W,c)$, the definition of nonnesting partitions does \emph{not} require the choice of a Coxeter element (or, equivalently, an ordering of the simple roots).

\subsection{The Kroweras complement}
\label{sec:Kroweras}

There is a natural map \defn{rowmotion} defined on $J(P)$ that takes an order ideal $I$ to the order ideal $\row(I)$ generated by the minimal elements of $P$ not in $I$~\cite{striker2012promotion} (see~\Cref{eq:row}).  In~\cite{panyushev2009orbits}, Panyushev conjectured that rowmotion had order $h$ or $2h$ on the set of nonnesting partitions.  This conjecture was refined by Bessis and Reiner to the statement that there should be an equivariant bijection between noncrossing partitions under $\Krew_c$ and nonnesting partitions under $\row$~\cite{BessisReiner2011}.  Such a bijection was recursively constructed by Armstrong, Stump, and Thomas for bipartite Coxeter element~\cite{armstrong2013uniform}.

Let $c$ be a Coxeter element and let $\inv(\w_\circ(c)) = [\alpha^{(1)},\alpha^{(2)}, \ldots, \alpha^{(N)}]$ be the inversion sequence associated to the $c$-sorting word for the long element $w_\circ$ in~\Cref{eq:root_order}.    In~\cite{thomas2019rowmotion}, a subset of the authors showed that it is possible to compute the Kreweras complement as a sequence of flips in $\inv(\w_\circ(c))$ order on the $c$-Cambrian lattice under the tag line ``rowmotion in slow motion.''

For a poset $(P,\leq)$, define the \defn{toggle} at $x\in P$ to be the bijection $\t_x \colon J(P) \rightarrow J(P)$ defined by
\begin{equation}
	\label{def:toggle}
	\t_x(I) \coloneqq \begin{cases}
		I \cup \{x\}      & \text{ if } x \not \in I \text{ and } I \cup \{x\} \in J(P) \\
		I \setminus \{x\} & \text{ if } x \in I \text{ and } I \setminus \{x\} \in J(P) \\
		I                 & \text{ otherwise}
	\end{cases}.
\end{equation}
Any toggle $\t_x$ is an involution, and $\t_x$ commutes with $\t_y$ if $x$ does not cover $y$ and $y$ does not cover $x$.  Write \[\t_{[x_1,x_2,\ldots,x_m]} = \t_{x_m} \circ \t_{x_{m-1}} \circ \cdots \circ \t_{x_1}.\]%

It turns out that $\row$ can also be written as a product of toggles~\cite{cameron1995orbits,striker2012promotion}; for nonnesting partitions, $\row$ can be computed by toggling each root of the root poset in order of height (or \emph{row}).  It is Coxeter-theoretically natural to consider what happens if one instead applies the ``slow motion'' Kreweras complement for nonnesting partitions---using toggles instead of flips.

\begin{definition}
	\label{def:krow}
	As in~\Cref{eq:krew}, we define the \defn{$c$-Kroweras complement} on $\NN(\S_n)$ by
	\[
		\mrho_c(\nn) \coloneqq \t_{\inv(\w_\circ(c))}(\nn).
	\]
\end{definition}

\begin{example}
	Using the data from~\Cref{ex:c_sorting}, we have that
	\begin{align*}
		\mrho_c & = \t_{[(23),(13),(24),\ldots,(15),(78),(35),(45),(12)]}                                                                                  \\
		        & = \t_{(12)} \circ \t_{(45)} \circ \t_{(35)} \circ \t_{(78)} \circ \t_{(15)} \circ \cdots \circ t_{(24)} \circ \t_{(13)} \circ \t_{(23)}.
	\end{align*}
\end{example}

\begin{remark}
	\label{rem:commutation}
	The map $\mrho_c$ does not depend on the choice of reduced word $\sc$ for $c$:
	all reduced words for $c$ are related by a sequence of commutations of non-adjacent transpositions, and each commutation only affects the sequence $\alpha^{(k)}$ by interchanging pairs of roots that are not adjacent in the root poset.  By the remark after~\Cref{def:toggle}, this, therefore, does not change the composition of toggles.  We may therefore view $\mrho_c$ as the composition of toggles according to any linear extension of the poset obtained from the transitive closure of the Auslander--Reiten quiver $\AR(c)$.
\end{remark}

\begin{remark}
	\label{rem: promotion}
	For the linear Coxeter element $\co = (1, 2, \ldots, n)$, $\mrho_{\co}$ toggles roots in lexicographic order $[(1,2),(1,3),\ldots,(n-1,n)]$, which is equivalent under commuting toggles to the sequence that toggles the elements of the root poset from left to right. In \cite{striker2012promotion}, this map was termed \emph{promotion} because---viewing a nonnesting partition as a standard Young tableau of shape $2\times n$---$\mrho_{\co}$ recovers Sch\"utzenberger promotion on standard Young tableaux.  The main result of \cite{striker2012promotion} showed that promotion on nonnesting partitions was conjugate by toggles to rowmotion.  Our K-{\it row}-eras complement is a reference to this second map.
\end{remark}

\begin{example}
	\label{ex:NN Kreweras}%
	Let $c = s_1s_3s_2 \in \S_4$. The $c$-sorting word for $w_\circ$ is $s_1s_3s_2s_1s_3s_2$, and the corresponding ordering of positive roots is
	\[
		\alpha^{(1)} = (1,2), \quad \alpha^{(2)} = (3,4), \quad \alpha^{(3)} = (1,4), \quad \alpha^{(4)} = (2,4), \quad \alpha^{(5)} = (1,3), \quad \alpha^{(6)} = (2,3).
	\]
	\Cref{fig:c promotion example} illustrates the orbits of $\mrho_c$ on $\NN(4)$.
\end{example}

\begin{figure}[!ht]
	\begin{center}
		\begin{tikzpicture}[scale=0.9]
			\begin{scope}[xshift=0cm]
				\filldraw[teal!40!white] (7/16,1/16) -- (6/8,3/8) -- (8/8,1/8) -- (10/8,3/8) -- (12/8,1/8) -- (14/8,3/8) -- (33/16,1/16);
				\pretriang{4}
				\draw[<-] (4/4,7/4) -- (4/4,4/4);
				\draw[->] (6/4,7/4) -- (6/4,4/4);
			\end{scope}
			\begin{scope}[xshift=3cm]
				\filldraw[teal!40!white] (7/16,1/16) -- (6/8,3/8) -- (17/16,1/16);
				\filldraw[teal!40!white] (23/16,1/16) -- (14/8,3/8) -- (33/16,1/16);
				\pretriang{4}
				\draw[<-] (5/4,7/4) -- (5/4,4/4);
				\draw[->] (7.5/4,9/4) -- (10.5/4,9/4);
				\draw[<-] (7.5/4,2/4) -- (10.5/4,2/4);
			\end{scope}
			\begin{scope}[xshift=5cm]
				\filldraw[teal!40!white] (7/16,1/16) -- (10/8,7/8) -- (33/16,1/16);
				\pretriang{4}
				\draw[->] (5/4,7/4) -- (5/4,4/4);
			\end{scope}
			\begin{scope}[xshift=8cm]
				\filldraw[teal!40!white] (7/16,1/16) -- (8/8,5/8) -- (12/8,1/8) -- (14/8,3/8) -- (33/16,1/16);
				\pretriang{4}
				\draw[<-] (5/4,7/4) -- (5/4,4/4);
				\draw[->] (7.5/4,9/4) -- (10.5/4,9/4);
				\draw[<-] (7.5/4,2/4) -- (10.5/4,2/4);
			\end{scope}
			\begin{scope}[xshift=10cm]
				\filldraw[teal!40!white] (15/16,1/16) -- (12/8,5/8) -- (33/16,1/16);
				\pretriang{4}
				\draw[->] (7.5/4,9/4) -- (10.5/4,9/4);
				\draw[<-] (7.5/4,2/4) -- (10.5/4,2/4);
			\end{scope}
			\begin{scope}[xshift=12cm]
				\filldraw[teal!40!white] (7/16,1/16) -- (6/8,3/8) -- (8/8,1/8) -- (10/8,3/8) -- (25/16,1/16);
				\pretriang{4}
				\draw[->] (7.5/4,9/4) -- (10.5/4,9/4);
				\draw[<-] (7.5/4,2/4) -- (10.5/4,2/4);
			\end{scope}
			\begin{scope}[xshift=14cm]
				\filldraw[teal!40!white] (23/16,1/16) -- (14/8,3/8) -- (33/16,1/16);
				\pretriang{4}
				\draw[->] (5/4,7/4) -- (5/4,4/4);
			\end{scope}

			\begin{scope}[yshift=1.8cm]
				\begin{scope}[xshift=0cm]
					\pretriang{4}
				\end{scope}
				\begin{scope}[xshift=3cm]
					\filldraw[teal!40!white] (15/16,1/16) -- (10/8,3/8) -- (25/16,1/16);
					\pretriang{4}
				\end{scope}
				\begin{scope}[xshift=5cm]
					\filldraw[teal!40!white] (7/16,1/16) -- (8/8,5/8) -- (10/8,3/8) -- (12/8,5/8) -- (33/16,1/16);
					\pretriang{4}
				\end{scope}
				\begin{scope}[xshift=8cm]
					\filldraw[teal!40!white] (7/16,1/16) -- (6/8,3/8) -- (17/16,1/16);
					\pretriang{4}
				\end{scope}
				\begin{scope}[xshift=10cm]
					\filldraw[teal!40!white] (15/16,1/16) -- (10/8,3/8) -- (12/8,1/8) -- (14/8,3/8) -- (33/16,1/16);
					\pretriang{4}
				\end{scope}
				\begin{scope}[xshift=12cm]
					\filldraw[teal!40!white] (7/16,1/16) -- (8/8,5/8) -- (25/16,1/16);
					\pretriang{4}
				\end{scope}
				\begin{scope}[xshift=14cm]
					\filldraw[teal!40!white] (7/16,1/16) -- (6/8,3/8) -- (8/8,1/8) -- (12/8,5/8) -- (33/16,1/16);
					\pretriang{4}
				\end{scope}
			\end{scope}
		\end{tikzpicture}
	\end{center}

	\caption{Orbits of $\mrho_{s_1s_3s_2} = \t_{(2,3)}\t_{(1,3)}\t_{(2,4)}\t_{(1,4)}\t_{(3,4)}\t_{(1,2)}$ on the nonnesting partitions $\NN(4)$.}
	\label{fig:c promotion example}
\end{figure}

In contrast to the noncrossing side, we must work harder on the nonnesting side to relate the toggle sequence $\mrho_c$ with the toggle sequence $\mrho_{c'}$.   Our goal is to find a conjugating sequence $\beta$ of toggles with $\beta \circ \mrho_c \circ \beta^{-1} = \mrho_{c'}$. %

As $s=s_k$ is initial in $c$, we fix the cycle notation for $c$ as $(a_1, \ldots, a_k, a_{k+1}, \ldots, a_n)$ with $a_1=k$, $a_{k+1}=k+1$, $\{a_1,\ldots,a_k\}=\{1,\ldots,k\}$, and $\{a_{k+1},\ldots,a_n\}=\{k+1,\ldots,n\}$.

Define the sequences of roots
\begin{align}
	\label{eq:tk}
	\nonumber\T_{(k,\bullet)}   & \coloneqq [(k,a_n),\ldots,(k,a_{k+3}),(k,a_{k+2})],                                              \\
	\nonumber\T_{(\bullet,k+1)} & \coloneqq [(a_k,k+1),\ldots,(a_3,k+1),(a_2,k+1)],                                                \\
	\T_k                        & \coloneqq \T_{(k,\bullet)} \sqcup \T_{(\bullet,k+1)}, \text{ and }                               \\
	\nonumber \sT_k             & \coloneqq s\T_k = [(k+1,a_n),\ldots,(k+1,a_{k+3}),(k+1,a_{k+2}),(a_k,k),\ldots,(a_3,k),(a_2,k)],
\end{align}  where we use $\sqcup$ to denote concatenation.  For example, the roots in $\T_6$ and $\T'_6$ are circled in~\Cref{fig:tri9}.

\begin{proposition}
	\label{prop:nn_induction}
	Let $c \xrightarrow{k} c'$ and define \begin{align}\label{eq:beta}\nonumber\beta_{c,k}: \NN(\S_n) &\to \NN(\S_n) \\ \beta_{c,k} & \coloneqq \t_{(k,k+1)} \circ \t_{\T_k} \circ \t_{(k,k+1)}.\end{align} Then $\Krow_{c'} \circ \beta_{c,k} = \beta_{c,k} \circ \Krow_c$.
\end{proposition}

\begin{proof}
	We show that $\Krow_c \circ \beta_{c,k}^{-1} = \beta_{c,k}^{-1} \circ \Krow_{c'}$. By the description of $\AR(c)$ in~\Cref{sec:sort_ar} and~\Cref{rem:commutation}, $\beta_{c,k}$ is the restriction of $\mrho_c$ to the roots that lie in $\T_{k}$.   Let $\mrho_{c,k}$ be the toggle sequence obtained from $\mrho_c$ by removing the toggles of the form $\t_t$ for $t \in \T_k$.  Then $\mrho_c \circ (\beta_{c,k})^{-1}=\mrho_{c,k}$, since (as shown in Figure \ref{fig: AR arround k} and proved in \cref{lem:outside}) the roots which precede those of $\T_k$ (other than $(k, k+1)$) are of the form $(i,j)$ with $i,j\leq k$ or $i,j\geq k+1$, and consequently none of these roots cover or are covered by roots in $\T_k$, so their toggles commute with those from $\T_k$.

	Write $s=s_k$.  By~\Cref{prop:ARc}, $\mrho_{c'}$ is obtained from $\mrho_c$ by moving
	$\t_{(k,k+1)}$ from the beginning of the toggle sequence to the end, and replacing
	$\t_{(i,j)}$ by $\t_{(s(i),s(j))}$.  Almost all roots are unchanged by this replacement, which has the simple effect of swapping the roots in $\T_k$ with the roots in $\sT_k$.

	So if $\mrho_{c',k}$ is the toggle sequence obtained from $\mrho_{c'}$ by removing the toggles of the form $\t_t$ for $t \in \T_k$, then the same reasoning as above shows that $(\beta_{c,k})^{-1} \circ \mrho_{c'}=\mrho_{c',k}$.

	It remains to show that $\mrho_{c,k}=\mrho_{c',k}$.
	Every $\t_t$ for $t\in \sT_k$ commutes with any $\t_r$ with $r$ in the rectangle defined by the diagonals labeled by $k$ and $k+1$--- except those $\t_r$ with $r \in \T_k$.  But these toggles no longer appear in either $\mrho_{c,k}$ or $\mrho_{c',k}$, so that we can commute the toggles from $\sT_k$ in $\mrho_{c,k}$ to their positions in $\mrho_{c',k}$.  Finally, since $\t_{(k,k+2)}$ and $\t_{(k-1,k+1)}$ have been removed from $\mrho_{c,k}$ and $\mrho_{c',k}$, $\t_{(k,k+1)}$ commutes with all remaining toggles.
\end{proof}

\section{Charmed bijections}
\label{sec:bij}

In this section, we define a general family of \emph{charmed bijections} between \emph{balanced pairs} of subsets and nonnesting partitions.  Our charmed bijections depend on a choice of decoration of the roots in $\Phi^+=\Phi^+(A_{n-1})$, and use certain \emph{intimate} families of lattice paths as intermediate objects.
Specializing these decorations to those arising from the choice of Coxeter element $c$, we obtain the \emph{$c$-charmed bijections} between $\NC(\S_n,c)$ and $\NN(\S_n)$.

\subsection{Balanced pairs and noncrossing partitions}
\begin{definition}
	Say that a pair of sets $(O,I)$ with $O,I \subseteq [n]$ is \defn{balanced} if  $|O|=|I|$ and $|O\cap [k]| \geq |I \cap [k]|$ for all $1\leq k \leq n$.  Write $\Bal(n)$ for all balanced pairs of subsets of $[n]$.
\end{definition}

We first show that balanced pairs are naturally in bijection with $c$-noncrossing partitions.  Let $\nc \in \NC(\S_n,c)$.
We define $O(\nc)$ to be the set of integers $i$ for which there exists an $j>i$ in the same block as $i$, and we define the set $I(\nc)$ to be the set of integers $j$ such that there exists an $i<j$ in the same block as $j$. It is immediate from the definition that $(O,I)\in \Bal(n)$.

\begin{proposition}
	\label{prop:bal_to_nc}
	The map\footnote{To be precise, both $O$ and $I$ depend on $c$, however since $c$ will always be clear from the context we omit it in the notation.} $\nc \mapsto (O(\nc),I(\nc))$ is a bijection between $\NC(\S_n,c)$ and $\Bal(n)$.
\end{proposition}

\begin{proof}
	We construct its inverse.  For a given pair $(O,I) \in \Bal(n)$, we can construct a $\nc \in \NC(\S_n,c)$ with $O(\nc)=O$ and $I(\nc)=I$ as follows.  The singletons of $\nc$ are the integers that are neither in $O$ nor $I$. We place each integer in $O\setminus I$ in its own block and call these blocks \defn{open}. Then we add iteratively the remaining integers of $O \cup I$ to the open blocks, starting with the smallest integer, such that the intermediate partition is always noncrossing.  This is achieved by adding an integer $x$ to the first open block we visit when walking from $x$ towards $n$ via $1$ in the cycle notation of $c$.  If an integer in $I\setminus O$ is added to a block we call this block \defn{closed} and thereafter do not add any integers to it. By construction we have $O(\nc)=O$ and $I(\nc)=I$.
\end{proof}

The bijection of~\Cref{prop:bal_to_nc} is illustrated in~\Cref{fig: pi to OI}.

\begin{figure}[htbp]
	\begin{tikzpicture}[scale=1.5]
		\def\c{{1,3,4,7,9,8,6,5,2}};
		\pgfmathtruncatemacro{\n}{dim(\c)};
		\pgfmathtruncatemacro{\no}{\n-1};
		\draw (0,0) circle (1cm);
		\foreach \i in {0,...,\no}{
				\pgfmathtruncatemacro{\j}{\c[\i]};
				\node (\j) at ({cos(360*\i/\n)},{sin(360*\i/\n)}){};
				\filldraw (\j) circle (1pt);
				\node (p\j) at ({1.2*cos(360*\i/\n)},{1.2*sin(360*\i/\n)}){\j};
			};
		\filldraw[teal] (1.center) circle (1pt);
		\filldraw[teal] (3.center) circle (1pt);
		\draw[teal, dashed, very thick] (1.center) circle (3pt);
		\draw[teal, dashed, very thick] (3.center) circle (3pt);
	\end{tikzpicture} \hspace{2cm}
	\begin{tikzpicture}[scale=1.5]
		\def\c{{1,3,4,7,9,8,6,5,2}};
		\pgfmathtruncatemacro{\n}{dim(\c)};
		\pgfmathtruncatemacro{\no}{\n-1};
		\draw (0,0) circle (1cm);
		\foreach \i in {0,...,\no}{
				\pgfmathtruncatemacro{\j}{\c[\i]};
				\node (\j) at ({cos(360*\i/\n)},{sin(360*\i/\n)}){};
				\filldraw (\j) circle (1pt);
				\node (p\j) at ({1.2*cos(360*\i/\n)},{1.2*sin(360*\i/\n)}){\j};
			};
		\filldraw[teal] (1.center) circle (1pt);
		\filldraw[teal] (3.center) circle (1pt);
		\filldraw[teal] (6.center) circle (1pt);
		\draw[teal, dashed, very thick] (3.center) circle (3pt);
		\draw[teal, dashed, thick, fill=red!40!white] (1.center) -- (6.center);
	\end{tikzpicture}

	\begin{tikzpicture}[scale=1.5]
		\def\c{{1,3,4,7,9,8,6,5,2}};
		\pgfmathtruncatemacro{\n}{dim(\c)};
		\pgfmathtruncatemacro{\no}{\n-1};
		\draw (0,0) circle (1cm);
		\foreach \i in {0,...,\no}{
				\pgfmathtruncatemacro{\j}{\c[\i]};
				\node (\j) at ({cos(360*\i/\n)},{sin(360*\i/\n)}){};
				\filldraw (\j) circle (1pt);
				\node (p\j) at ({1.2*cos(360*\i/\n)},{1.2*sin(360*\i/\n)}){\j};
			};
		\filldraw[teal] (1.center) circle (1pt);
		\filldraw[red] (3.center) circle (1pt);
		\filldraw[teal] (6.center) circle (1pt);
		\filldraw[red] (7.center) circle (1pt);
		\draw[teal, dashed, thick] (1.center) -- (6.center);
		\draw[red, thick] (3.center) -- (7.center);
	\end{tikzpicture}\hspace{2cm}
	\begin{tikzpicture}[scale=1.5]
		\def\c{{1,3,4,7,9,8,6,5,2}};
		\pgfmathtruncatemacro{\n}{dim(\c)};
		\pgfmathtruncatemacro{\no}{\n-1};
		\draw (0,0) circle (1cm);
		\foreach \i in {0,...,\no}{
				\pgfmathtruncatemacro{\j}{\c[\i]};
				\node (\j) at ({cos(360*\i/\n)},{sin(360*\i/\n)}){};
				\filldraw[red] (\j) circle (1pt);
				\node (p\j) at ({1.2*cos(360*\i/\n)},{1.2*sin(360*\i/\n)}){\j};
			};
		\draw[red, thick, fill=red!40!white] (1.center) -- (6.center) -- (8.center) -- (1.center);
		\draw[red, thick, fill=red!40!white] (3.center) -- (7.center);
	\end{tikzpicture}
	\caption{\label{fig: pi to OI} The construction of a noncrossing partition in $\NC(\S_9,c)$ for $c= s_2s_1s_3s_6s_5s_4s_8s_7 = (1\, 3\, 4\, 7\, 9\, 8\, 6\, 5\, 2)$ starting with the outgoing set $O=\{1,3,6\}$ and incoming set $I=\{6,7,8\}$. We depict open blocks in teal and dashed, and closed blocks in red. For  better visibility, we circled  open and closed singletons using the same color code.}
\end{figure}
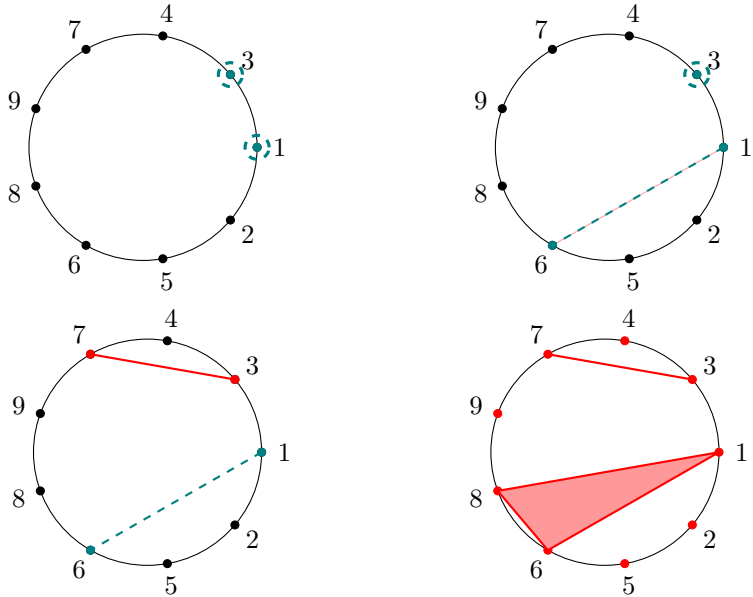

\subsection{Intimate families}
\label{sec:intimate}

Fix $n$ a positive integer and consider the type $A_{n-1}$ positive root poset $\Phi^+ \coloneqq \Phi^+(A_{n-1})$ on the positive roots $\{(i,j) \mid 1 \leq i < j \leq n \}$.  We draw the root $(i,j)$ in the plane with coordinates $(i+(j-i)/2,j-i)$ and---since the label $(i,j)$ is implied by the position---we may omit the labels on the roots. For $1 \leq i \leq n$, we draw $n$ additional points labeled by $i$ at coordinates $(i-1,0)$ and call these extra points \defn{integral vertices}. For $\M \subseteq \Phi^+$, we call a root $(i,j)$ \defn{charmed} if $(i,j) \in \M$ and \defn{ordinary} otherwise.  We depict charmed roots using hearts $\textcolor{red}{\varheartsuit}$ and ordinary roots using circles---an example of $c$-charmed roots is illustrated in \Cref{fig: root poset}.

A \defn{path} is a lattice path with step set $\{(.5,1),(.5,-1)\}$ that starts and ends at an integral vertex and stays strictly above the $x$-axis.  We call $(.5,1)$-steps \defn{up}, and $(.5,-1)$-steps \defn{down}; a \defn{peak} (resp. \defn{valley}) of a path is a root contained in an up step to its left (resp.~right) and a down step to its right (resp.~left).  Two paths are \defn{kissing} if they do not cross or share edges---they may \emph{meet} at a vertex, where they are said to \defn{kiss}. (Note, though, that kisses are not required for two paths to be kissing.) A family of paths is \defn{kissing} if they are pairwise kissing.  A path \defn{feints} at a root $(i,j)$ if $(i,j)$ is a valley, but the path does not kiss any path at $(i,j)$.  These definitions are illustrated on the right of~\Cref{fig: path configurations}: the top configuration is a feint at a charmed root, while the bottom one is a kiss at an ordinary root.

For $\M \subseteq \Phi^+$, a family $\L$ of kissing paths is called \defn{$\M$-charmed} if:
\begin{itemize}
	\item paths only kiss at charmed roots and
	\item paths only feint at ordinary roots.%
\end{itemize}

In other words, a family of paths is charmed if it avoids the two local configurations shown on the right of~\Cref{fig: path configurations}.%

\begin{definition}
	A family $\L$ of $\M$-charmed kissing paths is called \defn{$\M$-intimate} if:
	\begin{itemize}
		\item every ordinary root either lies above all paths in $\L$ or is contained in some path in $\L$ and
		\item no path contains a root above a charmed peak of a path in $\L$, unless that charmed peak is the location of a kiss.
	\end{itemize}
\end{definition}

\begin{figure}
	\begin{center}
		\begin{tikzpicture}
			\draw (1,-.5) -- (1.5,0) -- (2,-.5);
			\draw (2.5,-.5) -- (3,0) -- (3.5,.5);
			\draw (4,.5) -- (4.5,0) -- (5,-.5);
			\draw (5.5,.5) -- (6,0) -- (6.5,.5);
			\draw[very thick] (5.5,-.5) -- (6,0) -- (6.5,-.5);
			\draw (9.5,.5) -- (10,0) -- (10.5,.5);
			\marked{0}{0}
			\marked{1.5}{1.5}
			\marked{3}{3}
			\marked{4.5}{4.5}
			\marked{6}{6}
			\marked{10}{10}

			\begin{scope}[yshift=-2cm]
				\draw (1,-.5) -- (1.5,0) -- (2,-.5);
				\draw (2.5,-.5) -- (3,0) -- (3.5,.5);
				\draw (4,.5) -- (4.5,0) -- (5,-.5);
				\draw (5.5,.5) -- (6,0) -- (6.5,.5);
				\draw (9.5,.5) -- (10,0) -- (10.5,.5);
				\draw[very thick] (9.5,-.5) -- (10,0) -- (10.5,-.5);
				\draw[fill=lightgray,thick] (0,0) circle (2pt);
				\draw[fill=lightgray,thick] (1.5,0) circle (2pt);
				\draw[fill=lightgray,thick] (3,0) circle (2pt);
				\draw[fill=lightgray,thick] (4.5,0) circle (2pt);
				\draw[fill=lightgray,thick] (6,0) circle (2pt);
				\draw[fill=lightgray,thick] (10,0) circle (2pt);
			\end{scope}
		\end{tikzpicture}
	\end{center}
	\caption{{\it Left:} the five allowed local configurations for charmed families. {\it Right:} the two local configurations are forbidden for charmed families.}
	\label{fig: path configurations}
\end{figure}
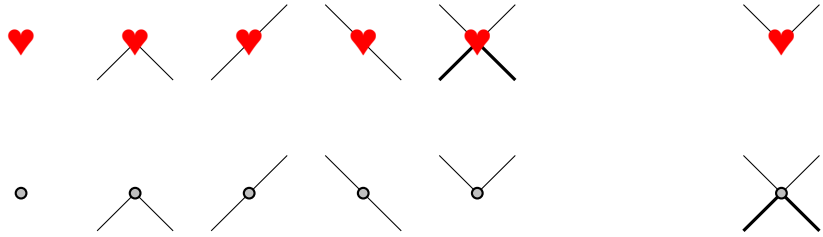

\subsection{Balanced pairs and intimate families}
We now relate balanced pairs and intimate families of paths.  For a family $\L$ of paths, we call an integral vertex on the $x$-axis \defn{outgoing} if it is incident to an up step and \defn{incoming} if it is incident to a down step (an integral vertex can be both outgoing and incoming).  Denote by $\Out(\L)$  (resp.~$\In(\L)$) the set of labels of outgoing (resp.~incoming) vertices of $\L$.  It is clear that $(\Out(\L),\In(\L))$ is balanced.

\begin{lemma}
	\label{lem: bij help in-out}
	Let $\M \subseteq \Phi^+$ and let $(O,I) \in \Bal(n)$. Then there is a unique $\M$-inimate family $\L_{(O,I)}$ with $\Out(\L_{(O,I)}) = O$ and $\In(\L_{(O,I)})=I$.
\end{lemma}
\begin{proof}
	We first construct a well-formed word of parentheses from the subsets $O$ and $I$. For $i$ from $1$ to $n$, write:
	\begin{itemize}
		\item a $(_i$ parenthesis if $i \in O\setminus I$,
		\item a $)_i$ parenthesis if $i \in I \setminus O$, and
		\item $)_i(_i$ parentheses if $i \in O \cap I$.
	\end{itemize}
	We now construct an $\M$-intimate family $\L$ recursively, starting with the empty family of paths. At each step, we pick neighbouring parentheses of the form $(_i)_j$, delete them, and add the path $\p$ that starts at $i$ and ends at $j$ that takes a down step whenever possible without violating the condition that the family is charmed, and an up step otherwise.  Then $\L \cup \{\p\}$ is intimate:
	\begin{itemize}
		\item If there were an ordinary root below $\L \cup \{\p\}$ that wasn't part of a path, then that root would lie between $p$ and $\L$, since $\L$ was intimate.  But then $\p$ took an up step instead of a possible down step, contradicting the definition of $\p$.
		\item If a previously constructed path $\p'$ in $\L$ started at an integral vertex after $i$, ended before $j$, and had a charmed peak which is not the location of a kiss, then our new path $\p$ will kiss $\p'$ at that charmed peak.
	\end{itemize}
	The order of choosing two neighbouring parentheses is irrelevant.  The family produced is unique, since if at any point a path uses a step different from those prescribed by the algorithm above, then the resulting family of paths will be non-intimate.  This non-intimacy will persist, regardless of how the family is extended.
	\qedhere
\end{proof}

An example of the algorithm used in the proof of~\Cref{lem: bij help in-out} is given in~\Cref{fig: ex for path const 1}.

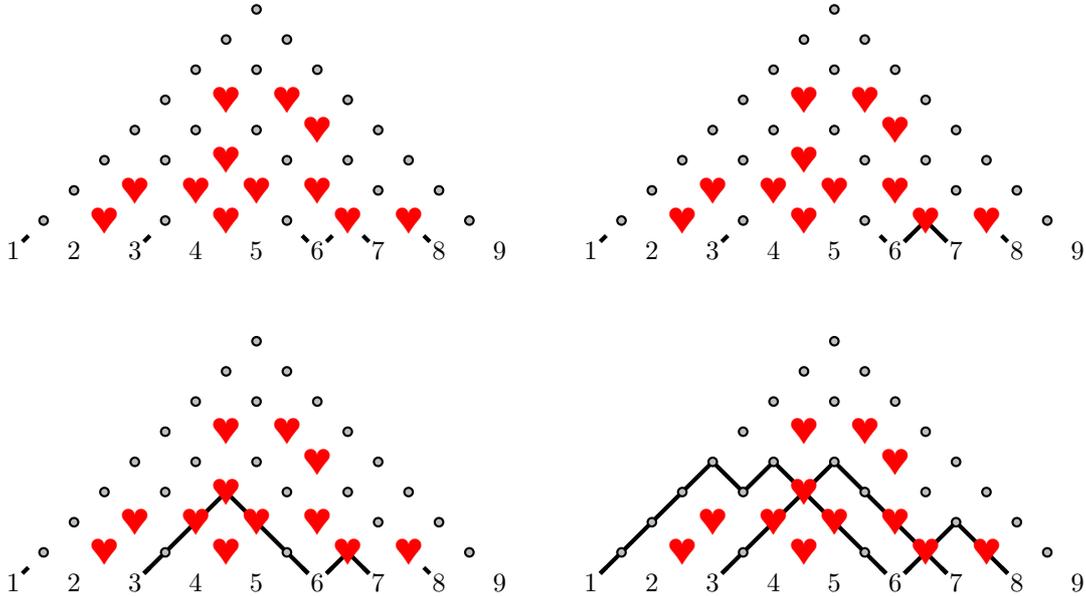
\begin{figure}[htbp]
	\begin{center}
		\begin{tikzpicture}[scale=.8]
			\ctriang
			\cbot

			\begin{scope}[xshift = 9.5cm]
				\draw[ultra thick] (6.15,.15) -- (6.5,.5) -- (6.85,.15);
				\ctriang
				\cbot
			\end{scope}

			\begin{scope}[yshift = -5.5cm]
				\draw[ultra thick] (3.15,.15) -- (4.5,1.5) -- (5.85,.15);
				\draw[ultra thick] (6.15,.15) -- (6.5,.5) -- (6.85,.15);
				\ctriang
				\cbot
			\end{scope}

			\begin{scope}[yshift = -5.5cm, xshift=9.5cm]
				\draw[ultra thick] (1.15,.15) -- (3,2) -- (3.5,1.5) -- (4,2) -- (4.5,1.5) -- (5,2) -- (6.5,.5) -- (7,1) -- (7.85,.15);
				\draw[ultra thick] (3.15,.15) -- (4.5,1.5) -- (5.85,.15);
				\draw[ultra thick] (6.15,.15) -- (6.5,.5) -- (6.85,.15);
				\ctriang
				\cbot
			\end{scope}
		\end{tikzpicture}
		\caption{
			The construction of an intimate family of paths with outgoing set $\{1,3,6\}$ and incoming set $\{6,7,8\}$.  The corresponding word of parentheses is $(_1 (_3 )_6(_6)_7)_8$.
		}
		\label{fig: ex for path const 1}
	\end{center}
\end{figure}

\subsection{Intimate families and nonnesting partitions}
Let $\L$ be an $\M$-intimate family of paths. We define the order ideal $J(\L)$ of $\L$ to be the set of all roots $(i,j)$ which lie on or below a path in $\L$. It is clear that $J(\L)$ is an order ideal and hence is in $\NN(\S_n)$.

\begin{lemma}
	\label{lem: bij help order ideal}
	Let $\M\subseteq \Phi^+$ and $J \in \NN(\S_n)$. Then there exists a unique $\M$-intimate family $\L_J$ with $J(\L_J)=J$.
\end{lemma}
\begin{proof}
	We construct an $\M$-intimate family $\L$ recursively, starting with the empty family of paths.   At each step, we add a maximal path $p$ to $\L$ such that all roots contained in $p$ lie in $J$.  We then replace $J$ by the order ideal generated by all ordinary roots in $J$ not contained in a path of $\L$, all charmed feints of paths of $\L$, and roots in $J$ not lying below a path of $\L$.  The recursion stops when $J$ is empty.  It is clear that the resulting family $\L$ is the unique $\M$-intimate family of paths with order ideal $J$.
\end{proof}%
An example of the algorithm used in the proof of~\Cref{lem: bij help order ideal} is given in~\Cref{fig: ex for path const 2}.

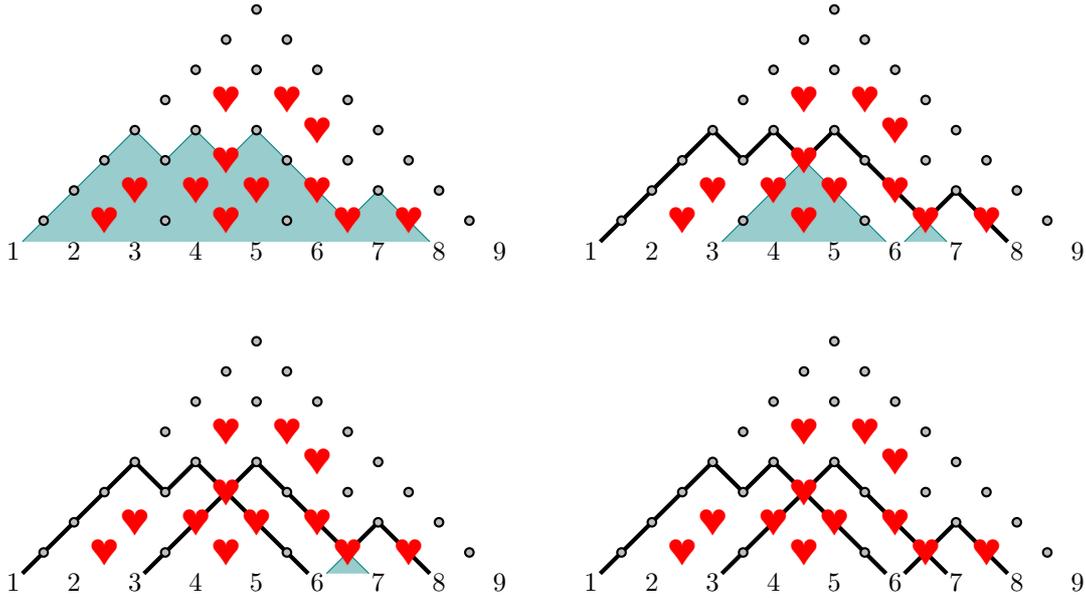
\begin{figure}[htbp]
	\begin{center}
		\begin{tikzpicture}[scale=.8]
			\draw[teal, fill=teal!40!white] (1.15,.15) -- (3,2) -- (3.5,1.5) -- (4,2) -- (4.5,1.5) -- (5,2) -- (6.5,.5) -- (7,1) -- (7.85,.15);
			\ctriang

			\begin{scope}[xshift = 9.5cm]
				\draw[teal, fill=teal!40!white] (3.15,.15) -- (4.5,1.5) -- (5.85,.15);
				\draw[teal, fill=teal!40!white] (6.15,.15) -- (6.5,.5) -- (6.85,.15);
				\draw[ultra thick] (1.15,.15) -- (3,2) -- (3.5,1.5) -- (4,2) -- (4.5,1.5) -- (5,2) -- (6.5,.5) -- (7,1) -- (7.85,.15);
				\ctriang
			\end{scope}

			\begin{scope}[yshift = -5.5cm, xshift=0cm]
				\draw[teal, fill=teal!40!white] (6.15,.15) -- (6.5,.5) -- (6.85,.15);
				\draw[ultra thick] (1.15,.15) -- (3,2) -- (3.5,1.5) -- (4,2) -- (4.5,1.5) -- (5,2) -- (6.5,.5) -- (7,1) -- (7.85,.15);
				\draw[ultra thick] (3.15,.15) -- (4.5,1.5) -- (5.85,.15);
				\ctriang
			\end{scope}

			\begin{scope}[yshift = -5.5cm, xshift=9.5cm]
				\draw[ultra thick] (1.15,.15) -- (3,2) -- (3.5,1.5) -- (4,2) -- (4.5,1.5) -- (5,2) -- (6.5,.5) -- (7,1) -- (7.85,.15);
				\draw[ultra thick] (3.15,.15) -- (4.5,1.5) -- (5.85,.15);
				\draw[ultra thick] (6.15,.15) -- (6.5,.5) -- (6.85,.15);
				\ctriang
			\end{scope}

		\end{tikzpicture}
	\end{center}
	\caption{
		The construction of the intimate family of lattice paths with order ideal $J$, where $J$ contains the roots in the grey-shaded region of the top left picture.  At each step, the order ideal under consideration consists of the roots contained in the teal-shaded region.
	}
	\label{fig: ex for path const 2}
\end{figure}

\subsection{Charmed bijections between balanced pairs and nonnesting partitions}
As a direct consequence of~\Cref{lem: bij help in-out,lem: bij help order ideal}, we obtain the following family of bijections between balanced pairs and nonnesting partitions.
\begin{proposition}
	\label{prop: J_S bijection}
	Fix a collection of charmed roots $\M \subseteq \Phi^+$. Then the map $J_\M \colon \Bal(n) \to \NN(\S_n)$ defined by $J_\M(O,I) = J(\L_{(O,I)})$ is a bijection.
\end{proposition}

\begin{remark}
	Charmed roots along the upper boundary of $\Phi^+$ do not affect the bijection of Proposition \ref{prop: J_S bijection}.  On the other hand, each of the $2^{\binom{n-2}{2}}$ charming choices for the the roots $(i,j)$ with $1 < i < j < n$ gives rise to a distinct bijection between $\NN(\S_n)$ and $\Bal(n)$, as we now prove.

	Let $\M \subseteq \Phi^+$ and $\nn\in \NN(\S_n)$. We need the following two observations.
	\begin{enumerate}
		\item The preimage $J_\M^{-1}(\nn)$ depends only on $\nn$ and $\M \cap \nn$, i.e., $J_\M^{-1}(\nn)=J_{\M\cap \nn}^{-1}(\nn)$.

		\item Let $(O,I)$ be the preimage $J_\M^{-1}(\nn)$ and $\p$ be the first path in the recursive construction of the $\M$-intimate family of paths as in the proof of~\Cref{lem: bij help order ideal} and $\nn^\prime$ be the corresponding order ideal.  Denote by $(O^\prime,I^\prime) =  J_\M^{-1}(\nn^\prime) = J_{\M\cap \nn^\prime}^{-1}(\nn^\prime)$. Then $O$ is the union of  $O^\prime$ and the starting point of $\p$, and $I$ is the union of $I^\prime$ and the end point of $\p$.
	\end{enumerate}

	Let $\M_1$ and $\M_2$ be different subsets of $\{(i,j) \mid 1 <i<j<n\}$ and assume that the bijections $J_{\M_1}$ and $J_{\M_2}$ are equal. Denote by $(i_0,j_0)$ a root with $j_0-i_0$ minimal such that it is in exactly one of the two sets $\M_1$ or $\M_2$; without loss of generality, we assume $(i_0,j_0) \in \M_1 \setminus \M_2$. Let $\nn$ be the order ideal generated by the roots $(i_0-1,j_0)$ and $(i_0,j_0+1)$ and let $(O,I)=J_{\M_1}^{-1}(\nn)=J_{\M_2}^{-1}(\nn)$.
	Denote by $\p_1$ and $\nn_1^\prime$ (resp. $\p_2$ and $\nn_2^\prime$) the path and order ideal appearing in the recursive construction of $J^{-1}_{\M_1}(\nn)$ (resp. $J^{-1}_{\M_2}(\nn)$) as described above. The paths $\p_1$ and $\p_2$ coincide since both are maximal. This implies by the second observation $J_{\M_1}^{-1}(\nn_1^\prime)=J_{\M_2}^{-1}(\nn_2^\prime)$. Since $(i_0,j_0) \notin \M_2$,%
	the order ideal $\nn_2^\prime$ is by construction contained in the set of roots $\{(i,j) \in \Phi^+; (i,j) < (i_0,j_0)\}$. By minimality of $(i_0,j_0)$ we therefore have $\nn_2^\prime \cap \M_1 = \nn_2^\prime \cap \M_2$. The first above observation implies
	\[
		J^{-1}_{\M_2}(\nn_2^\prime) = J^{-1}_{\M_2 \cap \nn_2^\prime}(\nn_2^\prime) = J^{-1}_{\M_1 \cap \nn_2^\prime}(\nn_2^\prime) = J^{-1}_{\M_1}(\nn_2^\prime),
	\]
	and hence $J_{\M_1}^{-1}(\nn_1^\prime)=J_{\M_1}^{-1}(\nn_2^\prime)$. However since $J_{\M_1}$ is a bijection and $\nn_1^\prime \neq \nn_2^\prime$ this is a contradiction.
\end{remark}

\subsection{Charmed bijections between noncrossing and nonnesting partitions}
\label{sec: the bijection}
Let $c \in \S_n$ be a Coxeter element.
Using~\Cref{prop:bal_to_nc} and~\Cref{prop: J_S bijection}, we now construct a bijection between $c$-noncrossing partitions and nonnesting partitions \emph{that depends on the choice of Coxeter element $c$}.  This addresses Incongruity~\eqref{it:2} from~\Cref{sec:history}.
\begin{definition}
	\label{def:bij}
	The \defn{$c$-charmed bijection} between $c$-noncrossing partitions and nonnesting partitions is given by
	\begin{align*}
		\Charm_c \colon \NC(\S_n,c) & \rightarrow \NN(\S_n)            \\
		\nc                         & \mapsto J_{\M_c}(O(\nc),I(\nc)),
	\end{align*}
	where the set $\heart_c$ is defined in~\Cref{def:charmed} and the map $J_{\M_c}$ is defined in \Cref{prop: J_S bijection}.
\end{definition}

In~\Cref{sec:proof}, we will show that $\Charm_c$ is the desired bijection from the introduction. We give also an alternative description for the inverse in \Cref{cor:tails} which allows one to directly read the blocks of the noncrossing partition from the corresponding intimate family.

\section{Proof of the main theorem}
\label{sec:proof}

In this section, we prove \Cref{thm:main_theorem}, which states that the bijection $\Charm_c$ from~\Cref{sec: the bijection} is the unique support-preserving bijection between noncrossing and nonnesting partitions that is equivariant with respect to $\Krew_c$ and $\Krow_c$.

\begin{figure}[htbp]
	\begin{tikzpicture}[scale=4]

		\node (00) at (0,0) {\scalebox{0.3}{\begin{tikzpicture}
					\draw[teal, fill=teal!40!white] (1.15,.15) -- (3.5,2.5) -- (4,2) -- (4.5,2.5) -- (6.85,.15);
					\draw[teal, fill=teal!40!white] (7.15,.15) -- (8,1) -- (8.85,.15);
					\draw[ultra thick] (3.15,.15) -- (3.5,.5) -- (3.85,.15);
					\draw[ultra thick] (5.15,.15) -- (5.5,.5) -- (5.85,.15);
					\draw[ultra thick] (7.15,.15) -- (8,1) -- (8.85,.15);
					\draw[ultra thick] (2.15,.15) -- (3.5,1.5) -- (4.85,.15);
					\draw[ultra thick] (1.15,.15) -- (3.5,2.5) -- (4,2) -- (4.5,2.5) -- (6.85,.15);

					\ctriang
				\end{tikzpicture}}};

		\node (30) at (3,0) {\scalebox{0.3}{\begin{tikzpicture}
					\draw[teal, fill=teal!40!white] (1.15,.15) -- (3.5,2.5) --  (5.5,.5) -- (6,1) -- (6.5,.5) -- (7,1) -- (7.5,.5) -- (8,1) -- (8.85,.15);
					\draw[ultra thick] (3.15,.15) -- (3.5,.5) -- (3.85,.15);
					\draw[ultra thick] (5.15,.15) -- (5.5,.5) -- (5.85,.15);
					\draw[ultra thick] (6.15,.15) -- (6.5,.5) -- (6.85,.15);
					\draw[ultra thick] (2.15,.15) -- (3.5,1.5) -- (4.85,.15);
					\draw[ultra thick] (1.15,.15) -- (3.5,2.5) --  (5.5,.5) -- (6,1) -- (6.5,.5) -- (7,1) -- (7.5,.5) -- (8,1) -- (8.85,.15);

					\cptriang
				\end{tikzpicture}}};

		\node (21) at (2,1) {\scalebox{0.3}{\begin{tikzpicture}
					\draw[teal, fill=teal!40!white] (7.15,.15) -- (7.5,.5) -- (7.85,.15);
					\draw[teal, fill=teal!40!white] (1.15,.15) -- (3,2) -- (3.5,1.5) -- (4.5,2.5) -- (6.85,.15);
					\draw[ultra thick] (7.15,.15) -- (7.5,.5) -- (7.85,.15);
					\draw[ultra thick] (3.15,.15) -- (4.5,1.5) -- (5.85,.15);
					\draw[ultra thick] (1.15,.15) -- (3,2) -- (3.5,1.5) -- (4.5,2.5) -- (6.85,.15);
					\cptriang
				\end{tikzpicture}}};

		\node (33) at (3,3) {\scalebox{0.7}{\begin{tikzpicture}
					\def\c{{1,3,4,6,9,8,7,5,2}};
					\pgfmathtruncatemacro{\n}{dim(\c)};
					\pgfmathtruncatemacro{\no}{\n-1};
					\draw (0,0) circle (1cm);
					\foreach \i in {0,...,\no}{
							\pgfmathtruncatemacro{\j}{\c[\i]};
							\node (\j) at ({cos(360*\i/\n)},{sin(360*\i/\n)}){};
							\filldraw[red] (\j) circle (1pt);
							\node (p\j) at ({1.2*cos(360*\i/\n)},{1.2*sin(360*\i/\n)}){\j};
						};
					\draw[red, thick, fill=red!40!white] (1.center) -- (6.center) -- (9.center) -- (1.center);
					\draw[red, thick, fill=red!40!white] (2.center) -- (5.center) -- (7.center) -- (2.center);
					\draw[red, thick, fill=red!40!white] (3.center) -- (4.center);
				\end{tikzpicture}}};

		\node (03) at (0,3) {\scalebox{0.7}{\begin{tikzpicture}
					\def\c{{1,3,4,7,9,8,6,5,2}};
					\pgfmathtruncatemacro{\n}{dim(\c)};
					\pgfmathtruncatemacro{\no}{\n-1};
					\draw (0,0) circle (1cm);
					\foreach \i in {0,...,\no}{
							\pgfmathtruncatemacro{\j}{\c[\i]};
							\node (\j) at ({cos(360*\i/\n)},{sin(360*\i/\n)}){};
							\filldraw[red] (\j) circle (1pt);
							\node (p\j) at ({1.2*cos(360*\i/\n)},{1.2*sin(360*\i/\n)}){\j};
						};
					\draw[red, thick, fill=red!40!white] (1.center) -- (7.center) -- (9.center) -- (1.center);
					\draw[red, thick, fill=red!40!white] (2.center) -- (5.center) -- (6.center) -- (2.center);
					\draw[red, thick, fill=red!40!white] (3.center) -- (4.center);
				\end{tikzpicture}}};

		\node (11) at (1,1) {\scalebox{0.3}{\begin{tikzpicture}
					\draw[teal, fill=teal!40!white] (1.15,.15) -- (3,2) -- (3.5,1.5) -- (4,2) -- (4.5,1.5) -- (5,2) -- (6.5,.5) -- (7,1) -- (7.85,.15);
					\draw[ultra thick] (1.15,.15) -- (3,2) -- (3.5,1.5) -- (4,2) -- (4.5,1.5) -- (5,2) -- (6.5,.5) -- (7,1) -- (7.85,.15);
					\draw[ultra thick] (3.15,.15) -- (4.5,1.5) -- (5.85,.15);
					\draw[ultra thick] (6.15,.15) -- (6.5,.5) -- (6.85,.15);
					\ctriang
				\end{tikzpicture}}};

		\node (12) at (1,2) {\scalebox{0.7}{\begin{tikzpicture}
					\def\c{{1,3,4,7,9,8,6,5,2}};
					\pgfmathtruncatemacro{\n}{dim(\c)};
					\pgfmathtruncatemacro{\no}{\n-1};
					\draw (0,0) circle (1cm);
					\foreach \i in {0,...,\no}{
							\pgfmathtruncatemacro{\j}{\c[\i]};
							\node (\j) at ({cos(360*\i/\n)},{sin(360*\i/\n)}){};
							\filldraw[red] (\j) circle (1pt);
							\node (p\j) at ({1.2*cos(360*\i/\n)},{1.2*sin(360*\i/\n)}){\j};
						};
					\draw[red, thick, fill=red!40!white] (1.center) -- (6.center) -- (8.center) -- (1.center);
					\draw[red, thick, fill=red!40!white] (3.center) -- (7.center);
				\end{tikzpicture}}};

		\node (22) at (2,2) {\scalebox{0.7}{\begin{tikzpicture}
					\def\c{{1,3,4,6,9,8,7,5,2}};
					\pgfmathtruncatemacro{\n}{dim(\c)};
					\pgfmathtruncatemacro{\no}{\n-1};
					\draw (0,0) circle (1cm);
					\foreach \i in {0,...,\no}{
							\pgfmathtruncatemacro{\j}{\c[\i]};
							\node (\j) at ({cos(360*\i/\n)},{sin(360*\i/\n)}){};
							\filldraw[red] (\j) circle (1pt);
							\node (p\j) at ({1.2*cos(360*\i/\n)},{1.2*sin(360*\i/\n)}){\j};
						};
					\draw[red, thick, fill=red!40!white] (1.center) -- (7.center) -- (8.center) -- (1.center);
					\draw[red, thick, fill=red!40!white] (3.center) -- (6.center);
				\end{tikzpicture}}};

		\draw[->] (00) -- node[rectangle,fill=white,align=center] {$\beta_{c,k}$} (30);
		\draw[->] (11) -- node[rectangle,fill=white,align=center] {$\beta_{c,k}$} (21);
		\draw[->] (12) -- node[rectangle,fill=white,align=center] {$\alpha_{c,k}$} (22);
		\draw[->] (03) -- node[rectangle,fill=white,align=center] {$\alpha_{c,k}$} (33);
		\draw[->] (03) -- node[rectangle,fill=white,align=center] {$\Charm_c$} (00);
		\draw[->] (12) -- node[rectangle,fill=white,align=center] {$\Charm_c$} (11);
		\draw[->] (33) -- node[rectangle,fill=white,align=center] {$\Charm_{c'}$} (30);
		\draw[->] (22) -- node[rectangle,fill=white,align=center] {$\Charm_{c'}$} (21);
		\draw[->] (12) -- node[rectangle,fill=white,align=center] {$\Krew_c$} (03);
		\draw[->] (22) -- node[rectangle,fill=white,align=center] {$\Krew_{c'}$} (33);
		\draw[->] (11) -- node[rectangle,fill=white,align=center] {$\Krow_c$} (00);
		\draw[->] (21) -- node[rectangle,fill=white,align=center] {$\Krow_{c'}$} (30);
	\end{tikzpicture}
	\caption{\Cref{fig:outline}, where the vertices previously labeled by the sets of noncrossing and nonnesting partitions are now replaced by our running example for the choice of Coxeter element $c= s_2s_1s_3s_6s_5s_4s_8s_7\in \S_9$ with cycle notation $(1\, 3\, 4\, 7\, 9\, 8\, 6\, 5\, 2)$.}
	\label{ex:running_example}
\end{figure}
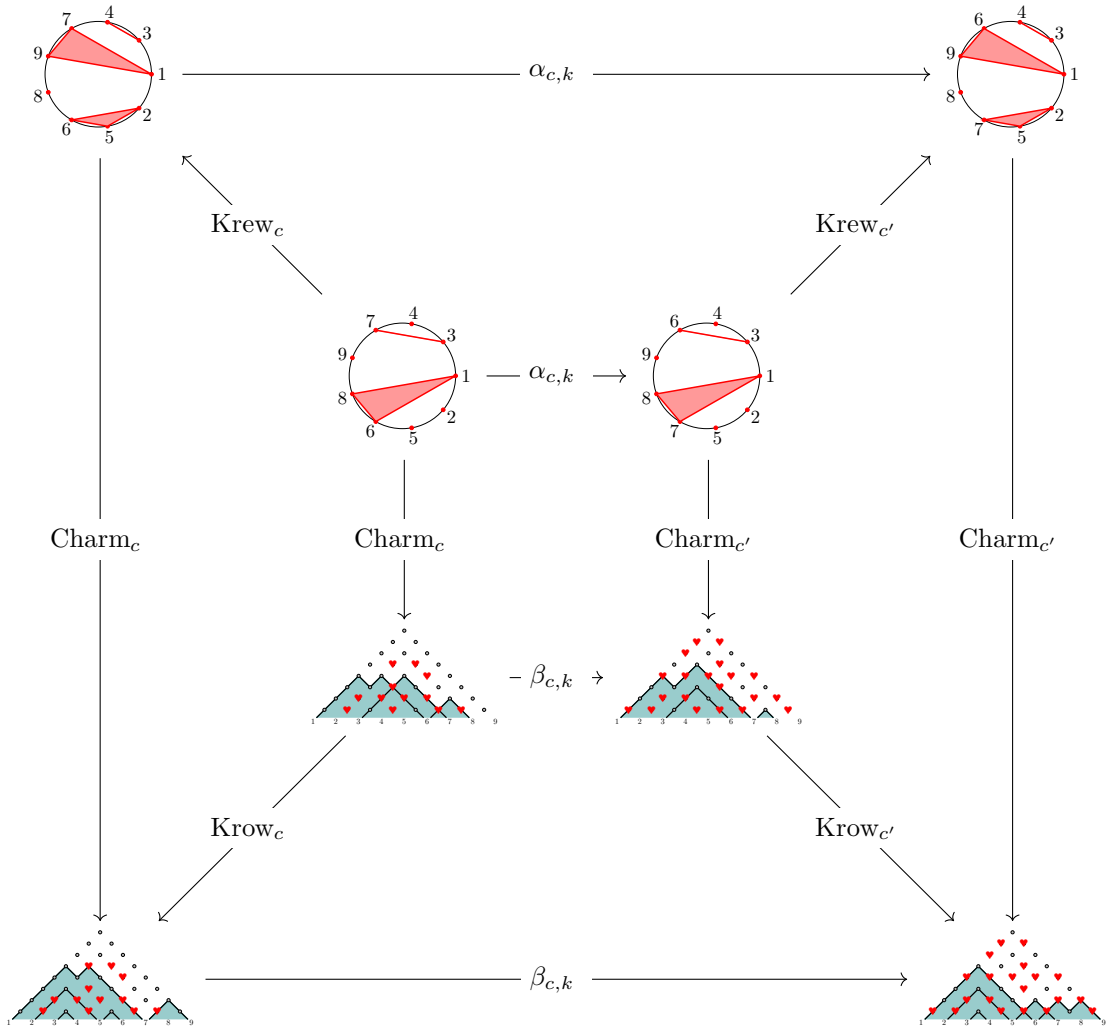

\subsection{Cambrian induction}
\label{sec:cambrian_induction}

Fix the symmetric group $\S_n$, let $c$ be a Coxeter element, and let $s=s_k=(k,k+1)$ be initial in $c$. Then $c' = scs$ is also a Coxeter element of $\S_n$, and recall from~\Cref{sec:coxeter_els} that we write $c \xrightarrow{k} c'$.%

We will prove \Cref{thm:main_theorem} using \defn{Cambrian induction}---that is, we will show that the theorem holds for a particular Coxeter element $\co$ (the \emph{base case}), and then we will show that if $c \rightarrow c'$ and the theorem holds for $c$, then the theorem also holds for $c'$ (the \emph{inductive step}).  By~\Cref{lem:initial_conj}, all Coxeter elements in $\S_n$ are conjugate by a sequence of conjugations by initial simple reflections, so the theorem holds for all Coxeter elements.

\begin{remark}
	Note that this proof technique differs from the usual Cambrian induction used in~\cite{reading2007clusters,armstrong2013uniform}: the base case of our proof is the statement for the linear Coxeter element in $\S_n$, and not the statement for the Coxeter group $\S_2$---that is, we do not need to descend to parabolic subgroups.
\end{remark}

\subsection{Base case}

For the base case, we use the linear Coxeter element $\co = s_1s_2 \cdots s_{n-1}$, with one-line notation $(1\, 2\, \dots\, n)$.   We show that $\Charm_c$ is an equivariant bijection between $\NC(\S_n,\co)$ under $\Krew_{\co}$ and $\NN(\S_n)$ under $\Krow_{\co}$.  In this case, $\Krow_{\co}$ is order ideal promotion, as defined by Striker and Williams~\cite{striker2012promotion} (see \Cref{rem: promotion}), which toggles roots in lexicographic order.  There are well-known equivariant bijections (see \cite[proof of Theorem 8.2]{Rhoades2010} and \cite[Theorems 3.9 and 3.10]{striker2012promotion})
\[
	\binom{\NC(\S_n,\co),}{\text{Kreweras complement}} \xrightarrow{\thicken} \binom{\text{noncrossing matchings of } [2n],}{\text{clockwise rotation}} \xrightarrow{\dyck} \binom{\NN(\S_n),}{\text{promotion}}.
\]
The bijection $\thicken$ from a noncrossing partition $\pi$ on $\{1,2,\ldots,n\}$ to a noncrossing matching on $\{1,\overline{1},2,\overline{2},\ldots,n,\overline{n}\}$ goes by a ``thickening'' operation, as illustrated at the top left of \Cref{fig:known bijections}; thickening produces even-length cycles, and to get a matching from an even-length cycle, we take the edges starting from a non-barred letter and going clockwise (which end on a barred letter by construction). %
To produce a nonnesting partition from a noncrossing matching, the bijection $\dyck$ reads each vertex $1,\overline{1},2,\overline{2},\ldots,n,\overline{n}$ in counter-clockwise order---each letter that opens a new block in the noncrossing matching becomes a NE step in a Dyck path, while each letter closing a block becomes a SE step.  For $\co$, there are no charmed roots, so the paths in the corresponding intimate family are shifts of the initial Dyck path, removing the lower two strips of steps each time.  It is clear from this description that the resulting balanced pair $(O,I)$ recovers the balanced pair $(O(\pi),I(\pi))$ for the original noncrossing partition $\pi$, so that $\dyck \circ \thicken = \bijK_{\co}$.  This is illustrated on the bottom of \Cref{fig:known bijections}.

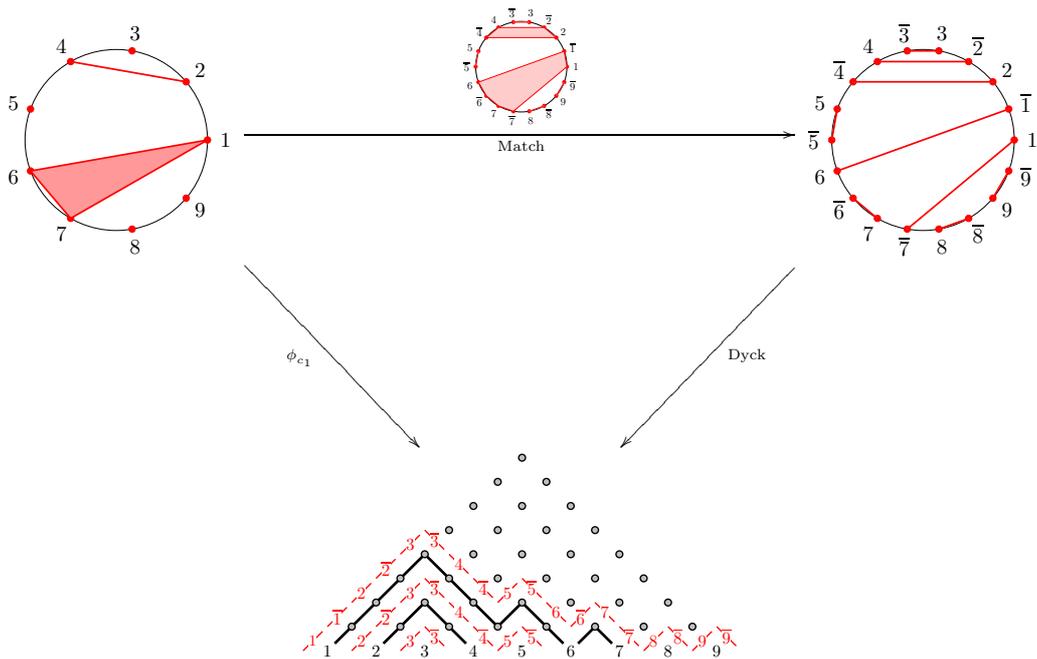
\begin{figure}[htbp]
	\centering

	\scalebox{.8}{
	\xymatrix{
	\raisebox{-.5\height}{\begin{tikzpicture}[scale=1.5]%
			\def\c{{1,2,3,4,5,6,7,8,9}};
			\pgfmathtruncatemacro{\n}{dim(\c)};
			\pgfmathtruncatemacro{\no}{\n-1};
			\draw (0,0) circle (1cm);
			\foreach \i in {0,...,\no}{
					\pgfmathtruncatemacro{\j}{\c[\i]};
					\node (\j) at ({cos(360*\i/\n)},{sin(360*\i/\n)}){};
					\filldraw[red] (\j) circle (1pt);
					\node (p\j) at ({1.2*cos(360*\i/\n)},{1.2*sin(360*\i/\n)}){\j};
				};
			\draw[red, thick, fill=red!40!white] (1.center) -- (6.center) -- (7.center) -- (1.center);
			\draw[red, thick, fill=red!40!white] (2.center) -- (4.center);
		\end{tikzpicture}} \ar[rddd]_{\bijK_{\co}} \ar[rr]_{\thicken}^{\scalebox{0.5}{\raisebox{-.5\height}{\begin{tikzpicture}[scale=1.5]
				\def\c{{1,2,3,4,5,6,7,8,9}};
				\pgfmathtruncatemacro{\n}{dim(\c)};
				\pgfmathtruncatemacro{\no}{\n-1};
				\draw (0,0) circle (1cm);
				\foreach \i in {0,...,\no}{
						\pgfmathtruncatemacro{\j}{\c[\i]};
						\node (\j) at ({cos(360*(2*\i)/(2*\n))},{sin(360*(2*\i)/(2*\n))}){};
						\filldraw[red] (\j) circle (1pt);
						\node (p\j) at ({1.2*cos(360*(2*\i)/(2*\n))},{1.2*sin(360*(2*\i)/(2*\n))}){\j};
						\node (a\j) at ({cos(360*(2*\i+1)/(2*\n))},{sin(360*(2*\i+1)/(2*\n))}){};
						\filldraw[red] (a\j) circle (1pt);
						\node (ap\j) at ({1.2*cos(360*(2*\i+1)/(2*\n))},{1.2*sin(360*(2*\i+1)/(2*\n))}){$\overline{\j}$};
					};
				\draw[red, thick, fill=red!20!white] (a1.center) -- (6.center) -- (a6.center) -- (7.center) -- (a7.center) -- (1.center) -- (a1.center);
				\draw[red, thick, fill=red!20!white] (a2.center) -- (4.center) -- (a4.center) -- (2.center) -- (a2.center);
				\draw[red, thick, fill=red!20!white] (3.center) -- (a3.center);
				\draw[red, thick, fill=red!20!white] (5.center) -- (a5.center);
				\draw[red, thick, fill=red!20!white] (8.center) -- (a8.center);
				\draw[red, thick, fill=red!20!white] (9.center) -- (a9.center);
			\end{tikzpicture}}}} & &
	\raisebox{-.5\height}{\begin{tikzpicture}[scale=1.5]
			\def\c{{1,2,3,4,5,6,7,8,9}};
			\pgfmathtruncatemacro{\n}{dim(\c)};
			\pgfmathtruncatemacro{\no}{\n-1};
			\draw (0,0) circle (1cm);
			\foreach \i in {0,...,\no}{
					\pgfmathtruncatemacro{\j}{\c[\i]};
					\node (\j) at ({cos(360*(2*\i)/(2*\n))},{sin(360*(2*\i)/(2*\n))}){};
					\filldraw[red] (\j) circle (1pt);
					\node (p\j) at ({1.2*cos(360*(2*\i)/(2*\n))},{1.2*sin(360*(2*\i)/(2*\n))}){\j};
					\node (a\j) at ({cos(360*(2*\i+1)/(2*\n))},{sin(360*(2*\i+1)/(2*\n))}){};
					\filldraw[red] (a\j) circle (1pt);
					\node (ap\j) at ({1.2*cos(360*(2*\i+1)/(2*\n))},{1.2*sin(360*(2*\i+1)/(2*\n))}){$\overline{\j}$};
				};
			\draw[red, thick, fill=red!20!white] (a1.center) -- (6.center);
			\draw[red, thick, fill=red!20!white] (a6.center) -- (7.center);
			\draw[red, thick, fill=red!20!white] (a7.center) -- (1.center);
			\draw[red, thick, fill=red!20!white] (a2.center) -- (4.center);
			\draw[red, thick, fill=red!20!white] (a4.center) -- (2.center);
			\draw[red, thick, fill=red!20!white] (3.center) -- (a3.center);
			\draw[red, thick, fill=red!20!white] (5.center) -- (a5.center);
			\draw[red, thick, fill=red!20!white] (8.center) -- (a8.center);
			\draw[red, thick, fill=red!20!white] (9.center) -- (a9.center);
		\end{tikzpicture}} \ar[lddd]^{\dyck} \\ & & \\ & & \\ &
	\scalebox{.8}{\raisebox{-.5\height}{\begin{tikzpicture}
				\draw[ultra thick] (1.15,.15) -- (3,2) -- (4.5,.5) -- (5,1) -- (5.85,.15);%
				\draw[ultra thick] (6.15,.15) -- (6.5,.5) -- (6.85,.15);
				\draw[ultra thick] (2.15,.15) -- (3,1) -- (3.85,.15);
				\triang{9};
				\draw[thick,red] (.5,0) -- (.9,.4) node[midway,fill=white,inner sep=0pt] {\small $1$};
				\draw[thick,red] (1,.5) -- (1.4,.9) node[midway,fill=white,inner sep=0pt] {$\overline{1}$};
				\draw[thick,red] (1.5,1) -- (1.9,1.4) node[midway,fill=white,inner sep=0pt] {$2$};
				\draw[thick,red] (2,1.5) -- (2.4,1.9) node[midway,fill=white,inner sep=0pt] {$\overline{2}$};
				\draw[thick,red] (2.5,2) -- (2.9,2.4) node[midway,fill=white,inner sep=0pt] {$3$};
				\draw[thick,red] (3,2.5) -- (3.4,2.1) node[midway,fill=white,inner sep=0pt] {$\overline{3}$};
				\draw[thick,red] (3.5,2) -- (3.9,1.6) node[midway,fill=white,inner sep=0pt] {$4$};
				\draw[thick,red] (4,1.5) -- (4.4,1.1) node[midway,fill=white,inner sep=0pt] {$\overline{4}$};
				\draw[thick,red] (4.5,1) -- (4.9,1.4) node[midway,fill=white,inner sep=0pt] {$5$};
				\draw[thick,red] (5,1.5) -- (5.4,1.1) node[midway,fill=white,inner sep=0pt] {$\overline{5}$};
				\draw[thick,red] (5.5,1) -- (5.9,.6) node[midway,fill=white,inner sep=0pt] {$6$};
				\draw[thick,red] (6,.5) -- (6.4,.9) node[midway,fill=white,inner sep=0pt] {$\overline{6}$};
				\draw[thick,red] (6.5,1) -- (6.9,.6) node[midway,fill=white,inner sep=0pt] {$7$};
				\draw[thick,red] (7,.5) -- (7.4,.1) node[midway,fill=white,inner sep=0pt] {$\overline{7}$};
				\draw[thick,red] (7.5,0) -- (7.9,.4) node[midway,fill=white,inner sep=0pt] {$8$};
				\draw[thick,red] (8,.5) -- (8.4,.1) node[midway,fill=white,inner sep=0pt] {$\overline{8}$};
				\draw[thick,red] (8.5,0) -- (8.9,.4) node[midway,fill=white,inner sep=0pt] {$9$};
				\draw[thick,red] (9,.5) -- (9.4,.1) node[midway,fill=white,inner sep=0pt] {$\overline{9}$};
				\draw[thick,red] (1.5,0) -- (1.9,0.4) node[midway,fill=white,inner sep=0pt] {$2$};
				\draw[thick,red] (2,.5) -- (2.4,.9) node[midway,fill=white,inner sep=0pt] {$\overline{2}$};
				\draw[thick,red] (2.5,1) -- (2.9,1.4) node[midway,fill=white,inner sep=0pt] {$3$};
				\draw[thick,red] (3,1.5) -- (3.4,1.1) node[midway,fill=white,inner sep=0pt] {$\overline{3}$};
				\draw[thick,red] (3.5,1) -- (3.9,.6) node[midway,fill=white,inner sep=0pt] {$4$};
				\draw[thick,red] (4,.5) -- (4.4,.1) node[midway,fill=white,inner sep=0pt] {$\overline{4}$};
				\draw[thick,red] (4.5,0) -- (4.9,.4) node[midway,fill=white,inner sep=0pt] {$5$};
				\draw[thick,red] (5,.5) -- (5.4,.1) node[midway,fill=white,inner sep=0pt] {$\overline{5}$};
				\draw[thick,red] (2.5,0) -- (2.9,.4) node[midway,fill=white,inner sep=0pt] {$3$};
				\draw[thick,red] (3,0.5) -- (3.4,.1) node[midway,fill=white,inner sep=0pt] {$\overline{3}$};
			\end{tikzpicture}}} &}}
	\caption{The equivariant bijection $\thicken$ from noncrossing partitions for $\co=(1,2,\ldots,n)$ under the Kreweras complement to noncrossing matchings under clockwise rotation and the equivariant bijection $\dyck$ from noncrossing matchings under clockwise rotation to nonnesting partitions under order ideal promotion.  As there are no charmed roots for $\co$, this recovers our $\co$-charmed bijection $\bijK_{\co}$.}
	\label{fig:known bijections}
\end{figure}

\subsection{Inductive step}
For the inductive step, we show that if $c \xrightarrow{k} c'$ and $\Charm_c$ is an equivariant bijection between $\NC(\S_n,c)$ under $\Krew_c$ and $\NN(\S_n)$ under $\Krow_c$, then the statement also holds for $c'$.

\begin{proposition}
	\label{prop:inductive step}
	Let $c \xrightarrow{k} c'$.  Then $\beta_{c,k} \circ \Charm_c = \Charm_{c'} \circ \alpha_{c,k}$.
\end{proposition}

\begin{proof}
	Fix $\nc \in \NC(\S_n,c)$ and write $\nc'=\alpha_{c,k}(\nc) \in \NC(\S_n,c')$. We will write $(O(\nc),I(\nc))$ and $(O(\nc'),I(\nc'))$ for the corresponding balanced pairs from~\Cref{prop:bal_to_nc}. By definition, we see that $(O(\nc'),I(\nc'))$ is obtained from $(O(\nc),I(\nc))$ by simply interchanging $k$ and $k+1$.  Denote by $\L$ (resp. $\L'$) the intimate family of lattice paths obtained when calculating $\Charm_{c}(\nc)$ (resp. $(\Charm_{c'} \circ \alpha_{c,k})(\nc)$).  By~\Cref{lem: bij help in-out}, $\L'$ is the unique $\heart_{c'}$-intimate family with $\Out(\L')=O(\nc^\prime)$ and $\In(\L')=I(\nc')$.  %
	Denote by $\L^t$ the unique $\heart_{c'}$-intimate family with $J(\L^t) = \beta_{c,k} \circ \Charm_c(\nc)$. Since $\L^\prime$ is by ~\Cref{lem: bij help order ideal} further the unique $\heart_{c'}$-intimate family with $J(\L')= \Charm_{c^\prime} \circ \alpha_{c,k}(\nc)$, it suffices to show that $\Out(\L^t) = O(\nc') = \Out(\L^\prime)$ and $\In(\L^t)=I(\nc')=\In(\L^\prime)$.

	Let $\T_k, \sT_k$ be the sequences of roots defined in~\Cref{eq:tk}. We prove the above assertion by showing that
	\begin{itemize}
		\item $\L$ and $\L^t$ coincide on all roots outside of $\T_k \cup \sT_k$ and that
		\item $\Out(\L^t)\cap \{k,k+1\}$ (resp. $\In(\L^t)\cap \{k,k+1\}$) is obtained from $\Out(\L)\cap \{k,k+1\}$ (resp. $\In(\L)\cap \{k,k+1\}$) by interchanging $k$ and $k+1$.
	\end{itemize}
	Since $\heart_{c}$ and $\heart_{c'}$ coincide on all roots outside of $\T_k \cup \sT_k$ by \Cref{prop:c_charmed}, it suffices by the construction of $\L$ and $\L^t$ as described in \Cref{lem: bij help order ideal} to show that $\L$ and $\L^t$ ``enter'' and ``leave'' the region of roots in $\T_k \cup \sT_k$ at the same roots; by symmetry, it suffices to show this only for the region of roots in $(\T_k \cup \sT_k) \cap \{(i,j) \mid i<j \leq k+1\}$. We show this by induction on the number of paths in $\L$ which include roots in $(\T_k \cup \sT_k) \cap \{(i,j) \mid i<j \leq k+1\}$.%

	Since $(k,k+1)$ is initial in $c$, we have that $k+1 \in \lheart_c$ and $k \in \rheart_c$, so that the roots in $\T_{(k,\bullet)}$ that are charmed form an initial subsequence of $\T_{(k,\bullet)}$.  Similarly, the roots in $\T_{(\bullet,k+1)}$ that are charmed form an initial subsequence of $\T_{(\bullet,k+1)}$. Note that the toggles $\t_t$ for $t \in \T_{(k,\bullet)}$ commute with the toggles $\t_r$ for $r \in \T_{(\bullet,k+1)}$.
	Furthermore for $i<k$, a root $(i,k)$ is $c$-charmed (or $c^\prime$-charmed resp.) if and only if $(i,k+1)$ is not $c$-charmed (or $c^\prime$-charmed resp.).	Using this, we can verify our assertion in the base case: \Cref{fig: induction basis} shows how the restriction of $\L$ to the region $(\T_k \cup \sT_k) \cap \{(i,j) \mid i<j \leq k+1\}$ is mapped to $\L^t$. Note, that it suffices to regard the cases shown in \Cref{fig: induction basis} because we assume that $\L$  is an intimate family of paths; for example the left-most path in  \Cref{fig: induction basis} would not be intimate if there would be an ordinary root below the dashed line.

	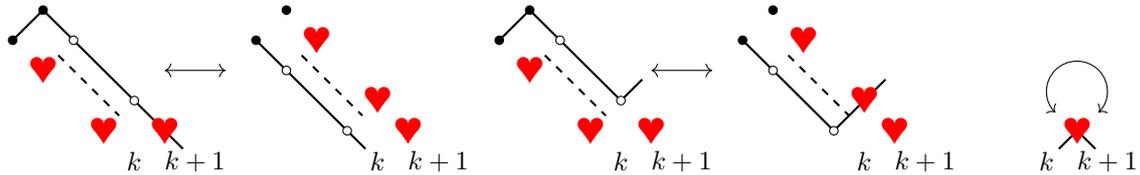
\begin{figure}[!ht]
		\begin{center}
			\begin{tikzpicture}[scale=.8]
				\filldraw (3.5,2.5) circle (2pt);
				\filldraw (3,2) circle (2pt);
				\draw[thick] (3,2) -- (3.5,2.5) -- (5.8,.2);
				\draw[fill=white] (4,2) circle (2pt);
				\draw[fill=white] (5,1) circle (2pt);
				\draw[thick, dashed] (3.75,1.75) -- (4.75,.75);
				\marked{2}{5}
				\marked{4}{5}
				\marked{5}{6}
				\node at (5,0) {$k$};
				\node at (6,0) {$k+1$};
				\draw[<->] (5.5,1.5) -- (6.5,1.5);
				\begin{scope}[xshift=4cm]
					\filldraw (3.5,2.5) circle (2pt);
					\filldraw (3,2) circle (2pt);
					\draw[thick] (3,2) -- (4.8,.2);
					\draw[fill=white] (3.5,1.5) circle (2pt);
					\draw[fill=white] (4.5,.5) circle (2pt);
					\draw[thick, dashed] (3.75,1.75) -- (4.75,.75);
					\marked{2}{6}
					\marked{4}{6}
					\marked{5}{6}
					\node at (5,0) {$k$};
					\node at (6,0) {$k+1$};
				\end{scope}

				\begin{scope}[xshift=8cm]
					\filldraw (3.5,2.5) circle (2pt);
					\filldraw (3,2) circle (2pt);
					\draw[thick] (3,2) -- (3.5,2.5) -- (5,1) -- (5.35,1.35);
					\draw[fill=white] (4,2) circle (2pt);
					\draw[fill=white] (5,1) circle (2pt);
					\draw[thick, dashed] (3.75,1.75) -- (4.75,.75);
					\marked{2}{5}
					\marked{4}{5}
					\marked{5}{6}
					\node at (5,0) {$k$};
					\node at (6,0) {$k+1$};
					\draw[<->] (5.5,1.5) -- (6.5,1.5);
					\begin{scope}[xshift=4cm]
						\filldraw (3.5,2.5) circle (2pt);
						\filldraw (3,2) circle (2pt);
						\draw[thick] (3,2) -- (4.5,.5) -- (5.35,1.35);
						\draw[fill=white] (3.5,1.5) circle (2pt);
						\draw[fill=white] (4.5,.5) circle (2pt);
						\draw[thick, dashed] (3.75,1.75) -- (4.75,.75);
						\marked{2}{6}
						\marked{4}{6}
						\marked{5}{6}
						\node at (5,0) {$k$};
						\node at (6,0) {$k+1$};
					\end{scope}
				\end{scope}

				\begin{scope}[xshift=15cm]
					\draw[thick] (5.2,.2) -- (5.5,.5) -- (5.8,.2);
					\node at (5,0) {$k$};
					\node at (6,0) {$k+1$};
					\marked{5}{6}
					\draw (5.85,.8) arc(315:585:.5)[<->];
				\end{scope}
			\end{tikzpicture}
		\end{center}
		\caption{\label{fig: induction basis}
			We use the following conventions to depict the different cases concisely.  The roots along the dashed lines are all charmed on one side and all ordinary on the other side. The root $(i,k)$ represented by a black dot can be either charmed or ordinary if it is reached by a north-east step. If the root $(i,k)$ is reached by a south-east step then it is charmed if and only if the root $(i-1,k)$ is ordinary.}
	\end{figure}

	For the induction step, we regard first the top path in $\T_k \cup \sT_k \cap \{(i,j) \mid i<j \leq k+1\}$, see \Cref{fig: bot path} for all possible cases. We observe that in all cases, the second path from the top behaves similarly to the top path, i.e., it would be mapped to the same path if we omit the top path.  This shows the claim by the induction hypothesis.\qedhere

	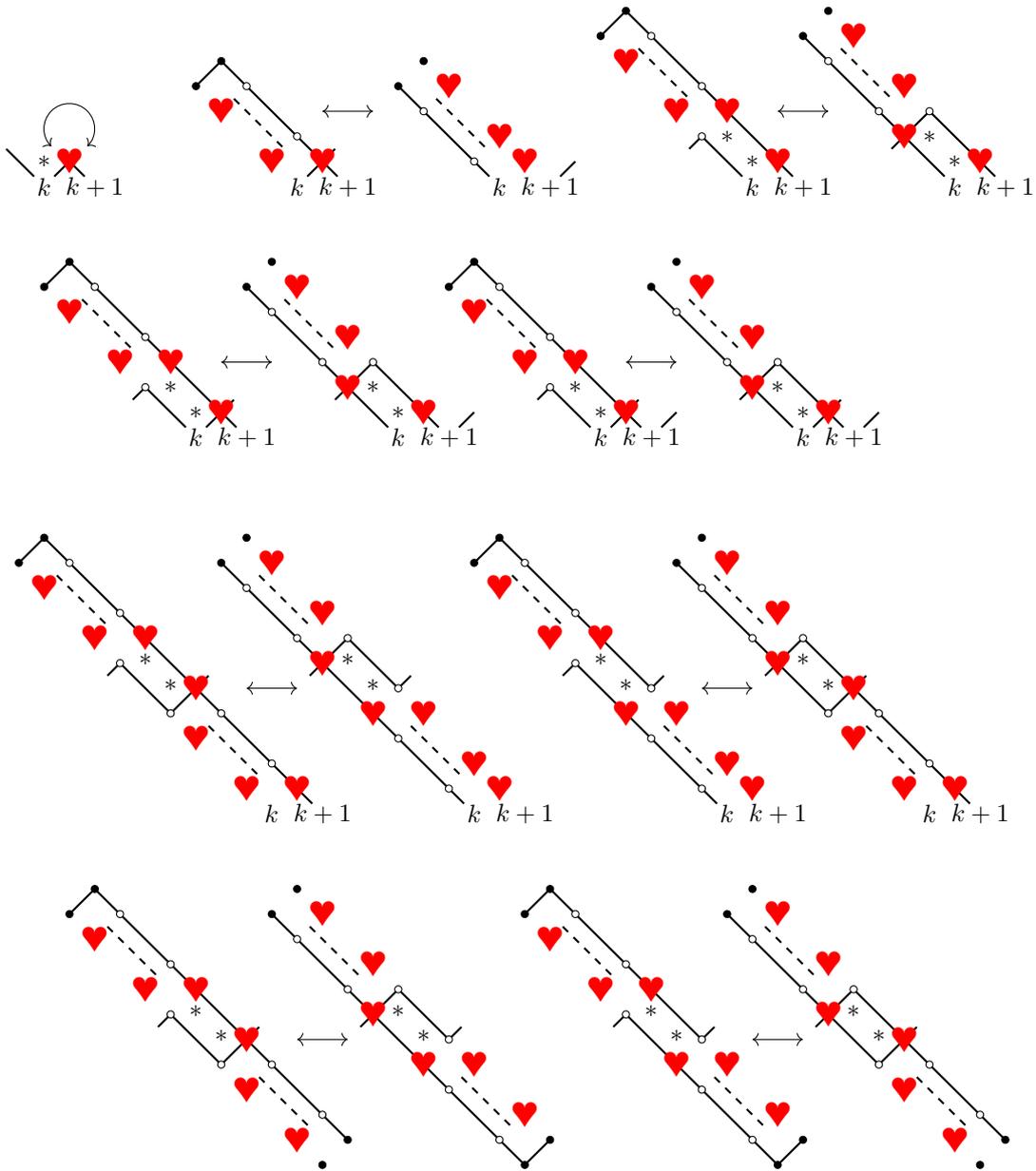
\begin{figure}
		\begin{center}
			\begin{tikzpicture}[scale=.7]

				\draw[thick] (5.2,.2) -- (5.5,.5) -- (5.8,.2);
				\draw[thick] (4.25,.75)  -- (4.8,.2);
				\node at (5,0) {$k$};
				\node at (6,0) {$k+1$};
				\node at (5,.5) {$*$};
				\marked{5}{6}
				\draw (5.85,.8) arc(315:585:.5)[<->];

				\begin{scope}[xshift=5cm]
					\filldraw (3.5,2.5) circle (2pt);
					\filldraw (3,2) circle (2pt);
					\draw[thick] (3,2) -- (3.5,2.5) -- (5.8,.2);
					\draw[thick] (5.2,.2) -- (5.75,.75);
					\draw[fill=white] (4,2) circle (2pt);
					\draw[fill=white] (5,1) circle (2pt);
					\draw[thick, dashed] (3.75,1.75) -- (4.75,.75);
					\marked{2}{5}
					\marked{4}{5}
					\marked{5}{6}
					\node at (5,0) {$k$};
					\node at (6,0) {$k+1$};
					\draw[<->] (5.5,1.5) -- (6.5,1.5);
					\begin{scope}[xshift=4cm]
						\filldraw (3.5,2.5) circle (2pt);
						\filldraw (3,2) circle (2pt);
						\draw[thick] (3,2) -- (4.8,.2);
						\draw[thick] (6.2,.2) -- (6.5,.5);
						\draw[fill=white] (3.5,1.5) circle (2pt);
						\draw[fill=white] (4.5,.5) circle (2pt);
						\draw[thick, dashed] (3.75,1.75) -- (4.75,.75);
						\marked{2}{6}
						\marked{4}{6}
						\marked{5}{6}
						\node at (5,0) {$k$};
						\node at (6,0) {$k+1$};
					\end{scope}
				\end{scope}

				\begin{scope}[xshift=12cm]
					\draw[thick] (4,3) -- (4.5,3.5) -- (7.8,.2);
					\draw[thick] (5.75,.75) -- (6,1) -- (6.8,.2);
					\draw[thick, dashed] (4.75,2.75) -- (5.75,1.75);
					\marked{2}{7}
					\marked{4}{7}
					\marked{5}{8}
					\marked{7}{8}
					\filldraw (4,3) circle (2pt);
					\filldraw (4.5,3.5) circle (2pt);
					\draw[fill=white] (5,3) circle (2pt);
					\draw[fill=white] (6,2) circle (2pt);
					\draw[fill=white] (6,1) circle (2pt);
					\node at (6.5,1) {$*$};
					\node at (7,.5) {$*$};
					\node at (7,0) {$k$};
					\node at (8,0) {$k+1$};
					\draw[<->] (7.5,1.5) -- (8.5,1.5);
					\begin{scope}[xshift=4cm]
						\draw[thick] (4,3) -- (6.8,.2);
						\draw[thick] (5.75,.75)-- (6.5,1.5) -- (7.8,.2);
						\draw[thick, dashed] (4.75,2.75) -- (5.75,1.75);
						\marked{2}{8}
						\marked{4}{8}
						\marked{5}{7}
						\marked{7}{8}
						\filldraw (4,3) circle (2pt);
						\filldraw (4.5,3.5) circle (2pt);
						\draw[fill=white] (4.5,2.5) circle (2pt);
						\draw[fill=white] (5.5,1.5) circle (2pt);
						\draw[fill=white] (6.5,1.5) circle (2pt);
						\node at (6.5,1) {$*$};
						\node at (7,.5) {$*$};
						\node at (7,0) {$k$};
						\node at (8,0) {$k+1$};
					\end{scope}
				\end{scope}

				\begin{scope}[yshift=-5cm, xshift=1cm]

					\draw[thick] (4,3) -- (4.5,3.5) -- (7.8,.2);
					\draw[thick] (5.75,.75) -- (6,1) -- (6.8,.2);
					\draw[thick] (7.2,.2) -- (7.75,.75);
					\draw[thick, dashed] (4.75,2.75) -- (5.75,1.75);
					\marked{2}{7}
					\marked{4}{7}
					\marked{5}{8}
					\marked{7}{8}
					\filldraw (4,3) circle (2pt);
					\filldraw (4.5,3.5) circle (2pt);
					\draw[fill=white] (5,3) circle (2pt);
					\draw[fill=white] (6,2) circle (2pt);
					\draw[fill=white] (6,1) circle (2pt);
					\node at (6.5,1) {$*$};
					\node at (7,.5) {$*$};
					\node at (7,0) {$k$};
					\node at (8,0) {$k+1$};
					\draw[<->] (7.5,1.5) -- (8.5,1.5);
					\begin{scope}[xshift=4cm]
						\draw[thick] (4,3) -- (6.8,.2);
						\draw[thick] (5.75,.75)-- (6.5,1.5) -- (7.8,.2);
						\draw[thick] (8.2,.2) -- (8.5,.5);
						\draw[thick, dashed] (4.75,2.75) -- (5.75,1.75);
						\marked{2}{8}
						\marked{4}{8}
						\marked{5}{7}
						\marked{7}{8}
						\filldraw (4,3) circle (2pt);
						\filldraw (4.5,3.5) circle (2pt);
						\draw[fill=white] (4.5,2.5) circle (2pt);
						\draw[fill=white] (5.5,1.5) circle (2pt);
						\draw[fill=white] (6.5,1.5) circle (2pt);
						\node at (6.5,1) {$*$};
						\node at (7,.5) {$*$};
						\node at (7,0) {$k$};
						\node at (8,0) {$k+1$};
					\end{scope}

					\begin{scope}[xshift=8cm]
						\draw[thick] (4,3) -- (4.5,3.5) -- (7.8,.2);
						\draw[thick] (5.75,.75) -- (6,1) -- (6.8,.2);
						\draw[thick] (7.2,.2) -- (7.75,.75);
						\draw[thick] (8.2,.2) -- (8.5,.5);
						\draw[thick, dashed] (4.75,2.75) -- (5.75,1.75);
						\marked{2}{7}
						\marked{4}{7}
						\marked{5}{8}
						\marked{7}{8}
						\filldraw (4,3) circle (2pt);
						\filldraw (4.5,3.5) circle (2pt);
						\draw[fill=white] (5,3) circle (2pt);
						\draw[fill=white] (6,2) circle (2pt);
						\draw[fill=white] (6,1) circle (2pt);
						\node at (6.5,1) {$*$};
						\node at (7,.5) {$*$};
						\node at (7,0) {$k$};
						\node at (8,0) {$k+1$};
						\draw[<->] (7.5,1.5) -- (8.5,1.5);
						\begin{scope}[xshift=4cm]
							\draw[thick] (4,3) -- (6.8,.2);
							\draw[thick] (5.75,.75)-- (6.5,1.5) -- (7.8,.2);
							\draw[thick] (7.2,.2) -- (7.75,.75);
							\draw[thick] (8.2,.2) -- (8.5,.5);
							\draw[thick, dashed] (4.75,2.75) -- (5.75,1.75);
							\marked{2}{8}
							\marked{4}{8}
							\marked{5}{7}
							\marked{7}{8}
							\filldraw (4,3) circle (2pt);
							\filldraw (4.5,3.5) circle (2pt);
							\draw[fill=white] (4.5,2.5) circle (2pt);
							\draw[fill=white] (5.5,1.5) circle (2pt);
							\draw[fill=white] (6.5,1.5) circle (2pt);
							\node at (6.5,1) {$*$};
							\node at (7,.5) {$*$};
							\node at (7,0) {$k$};
							\node at (8,0) {$k+1$};
						\end{scope}
					\end{scope}
				\end{scope}

				\begin{scope}[yshift=-13cm, xshift=-2cm]
					\draw[thick] (6.5,5.5) -- (7,6) -- (12.3,.7);
					\draw[thick] (8.25,3.25) -- (8.5,3.5) -- (9.5,2.5) -- (10.25,3.25);
					\draw[thick, dashed] (7.25,5.25) -- (8.25,4.25);
					\draw[thick, dashed] (10.25,2.25) -- (11.25,1.25);
					\node at (9,3.5) {*};
					\node at (9.5,3) {*};
					\filldraw (6.5,5.5) circle (2pt);
					\filldraw (7,6) circle (2pt);
					\marked{2}{12}
					\marked{4}{12}
					\marked{5}{13}
					\marked{7}{13}
					\marked{8}{12}
					\marked{10}{12}
					\marked{11}{13}
					\draw[fill=white] (7.5,5.5) circle (2pt);
					\draw[fill=white] (8.5,4.5) circle (2pt);
					\draw[fill=white] (8.5,3.5) circle (2pt);
					\draw[fill=white] (9.5,2.5) circle (2pt);
					\draw[fill=white] (10.5,2.5) circle (2pt);
					\draw[fill=white] (11.5,1.5) circle (2pt);
					\draw[<->] (11,3) -- (12,3);
					\node at (11.5,.5) {$k$};
					\node at (12.5,.5) {$k+1$};

					\begin{scope}[xshift=4cm]
						\draw[thick] (6.5,5.5) -- (11.3,.7);
						\draw[thick] (8.25,3.25) -- (9,4) -- (10,3) -- (10.25,3.25);
						\draw[thick, dashed] (7.25,5.25) -- (8.25,4.25);
						\draw[thick, dashed] (10.25,2.25) -- (11.25,1.25);
						\node at (9,3.5) {*};
						\node at (9.5,3) {*};
						\filldraw (6.5,5.5) circle (2pt);
						\filldraw (7,6) circle (2pt);
						\marked{2}{13}
						\marked{4}{13}
						\marked{5}{12}
						\marked{7}{12}
						\marked{8}{13}
						\marked{10}{13}
						\marked{11}{13}
						\draw[fill=white] (7,5) circle (2pt);
						\draw[fill=white] (8,4) circle (2pt);
						\draw[fill=white] (9,4) circle (2pt);
						\draw[fill=white] (10,3) circle (2pt);
						\draw[fill=white] (10,2) circle (2pt);
						\draw[fill=white] (11,1) circle (2pt);
						\node at (11.5,.5) {$k$};
						\node at (12.5,.5) {$k+1$};
					\end{scope}

					\begin{scope}[xshift=9cm]
						\draw[thick] (6.5,5.5) -- (7,6) -- (10,3) -- (10.25,3.25);
						\draw[thick] (8.25,3.25) -- (8.5,3.5) -- (11.3,.7);
						\draw[thick, dashed] (7.25,5.25) -- (8.25,4.25);
						\draw[thick, dashed] (10.25,2.25) -- (11.25,1.25);
						\node at (9,3.5) {*};
						\node at (9.5,3) {*};
						\filldraw (6.5,5.5) circle (2pt);
						\filldraw (7,6) circle (2pt);
						\marked{2}{12}
						\marked{4}{12}
						\marked{5}{13}
						\marked{7}{12}
						\marked{8}{13}
						\marked{10}{13}
						\marked{11}{13}
						\draw[fill=white] (7.5,5.5) circle (2pt);
						\draw[fill=white] (8.5,4.5) circle (2pt);
						\draw[fill=white] (8.5,3.5) circle (2pt);
						\draw[fill=white] (10,3) circle (2pt);
						\draw[fill=white] (10,2) circle (2pt);
						\draw[fill=white] (11,1) circle (2pt);
						\draw[<->] (11,3) -- (12,3);
						\node at (11.5,.5) {$k$};
						\node at (12.5,.5) {$k+1$};

						\begin{scope}[xshift=4cm]
							\draw[thick] (6.5,5.5) --  (9.5,2.5) -- (10.25,3.25);
							\draw[thick] (8.25,3.25) -- (9,4) -- (12.3,.7);
							\draw[thick, dashed] (7.25,5.25) -- (8.25,4.25);
							\draw[thick, dashed] (10.25,2.25) -- (11.25,1.25);
							\node at (9,3.5) {*};
							\node at (9.5,3) {*};
							\filldraw (6.5,5.5) circle (2pt);
							\filldraw (7,6) circle (2pt);
							\marked{2}{13}
							\marked{4}{13}
							\marked{5}{12}
							\marked{7}{13}
							\marked{8}{12}
							\marked{10}{12}
							\marked{11}{13}
							\draw[fill=white] (7,5) circle (2pt);
							\draw[fill=white] (8,4) circle (2pt);
							\draw[fill=white] (9,4) circle (2pt);
							\draw[fill=white] (9.5,2.5) circle (2pt);
							\draw[fill=white] (10.5,2.5) circle (2pt);
							\draw[fill=white] (11.5,1.5) circle (2pt);
							\node at (11.5,.5) {$k$};
							\node at (12.5,.5) {$k+1$};
						\end{scope}
					\end{scope}
				\end{scope}

				\begin{scope}[yshift=-20cm, xshift=-1cm]
					\draw[thick] (6.5,5.5) -- (7,6) -- (10,3) -- (10.25,3.25);
					\draw[thick] (8.25,3.25) -- (8.5,3.5)  -- (9.5,2.5) -- (10,3) -- (12,1);
					\draw[thick, dashed] (7.25,5.25) -- (8.25,4.25);
					\draw[thick, dashed] (10.25,2.25) -- (11.25,1.25);
					\node at (9,3.5) {*};
					\node at (9.5,3) {*};
					\filldraw (6.5,5.5) circle (2pt);
					\filldraw (7,6) circle (2pt);
					\marked{2}{12}
					\marked{4}{12}
					\marked{5}{13}
					\marked{7}{13}
					\marked{8}{12}
					\marked{10}{12}
					\draw[fill=white] (7.5,5.5) circle (2pt);
					\draw[fill=white] (8.5,4.5) circle (2pt);
					\draw[fill=white] (8.5,3.5) circle (2pt);
					\draw[fill=white] (9.5,2.5) circle (2pt);
					\draw[fill=white] (10.5,2.5) circle (2pt);
					\draw[fill=white] (11.5,1.5) circle (2pt);
					\filldraw (11.5,.5) circle (2pt);
					\filldraw (12,1) circle (2pt);
					\draw[<->] (11,3) -- (12,3);

					\begin{scope}[xshift=4cm]
						\draw[thick] (6.5,5.5) -- (8.5,3.5) -- (9,4) -- (10,3) -- (10.25,3.25);
						\draw[thick] (8.25,3.25) -- (8.5,3.5)  --  (11.5,.5) -- (12,1);
						\draw[thick, dashed] (7.25,5.25) -- (8.25,4.25);
						\draw[thick, dashed] (10.25,2.25) -- (11.25,1.25);
						\node at (9,3.5) {*};
						\node at (9.5,3) {*};
						\filldraw (6.5,5.5) circle (2pt);
						\filldraw (7,6) circle (2pt);
						\marked{2}{13}
						\marked{4}{13}
						\marked{5}{12}
						\marked{7}{12}
						\marked{8}{13}
						\marked{10}{13}
						\draw[fill=white] (7,5) circle (2pt);
						\draw[fill=white] (8,4) circle (2pt);
						\draw[fill=white] (9,4) circle (2pt);
						\draw[fill=white] (10,3) circle (2pt);
						\draw[fill=white] (10,2) circle (2pt);
						\draw[fill=white] (11,1) circle (2pt);
						\filldraw (11.5,.5) circle (2pt);
						\filldraw (12,1) circle (2pt);
					\end{scope}

					\begin{scope}[xshift=9cm]
						\draw[thick] (6.5,5.5) -- (7,6) -- (10,3) -- (10.25,3.25);
						\draw[thick] (8.25,3.25) -- (8.5,3.5)  -- (11.5,.5) -- (12,1);
						\draw[thick, dashed] (7.25,5.25) -- (8.25,4.25);
						\draw[thick, dashed] (10.25,2.25) -- (11.25,1.25);
						\node at (9,3.5) {*};
						\node at (9.5,3) {*};
						\filldraw (6.5,5.5) circle (2pt);
						\filldraw (7,6) circle (2pt);
						\marked{2}{12}
						\marked{4}{12}
						\marked{5}{13}
						\marked{7}{12}
						\marked{8}{13}
						\marked{10}{13}
						\draw[fill=white] (7.5,5.5) circle (2pt);
						\draw[fill=white] (8.5,4.5) circle (2pt);
						\draw[fill=white] (8.5,3.5) circle (2pt);
						\draw[fill=white] (10,3) circle (2pt);
						\draw[fill=white] (10,2) circle (2pt);
						\draw[fill=white] (11,1) circle (2pt);
						\filldraw (11.5,.5) circle (2pt);
						\filldraw (12,1) circle (2pt);
						\draw[<->] (11,3) -- (12,3);

						\begin{scope}[xshift=4cm]
							\draw[thick] (6.5,5.5) -- (8.5,3.5) -- (9,4) -- (10,3) -- (10.25,3.25);
							\draw[thick] (8.25,3.25) -- (8.5,3.5)  -- (9.5,2.5) -- (10,3) -- (12,1);
							\draw[thick, dashed] (7.25,5.25) -- (8.25,4.25);
							\draw[thick, dashed] (10.25,2.25) -- (11.25,1.25);
							\node at (9,3.5) {*};
							\node at (9.5,3) {*};
							\filldraw (6.5,5.5) circle (2pt);
							\filldraw (7,6) circle (2pt);
							\marked{2}{13}
							\marked{4}{13}
							\marked{5}{12}
							\marked{7}{13}
							\marked{8}{12}
							\marked{10}{12}
							\draw[fill=white] (7,5) circle (2pt);
							\draw[fill=white] (8,4) circle (2pt);
							\draw[fill=white] (9,4) circle (2pt);
							\draw[fill=white] (9.5,2.5) circle (2pt);
							\draw[fill=white] (10.5,2.5) circle (2pt);
							\draw[fill=white] (11.5,1.5) circle (2pt);
							\filldraw (11.5,.5) circle (2pt);
							\filldraw (12,1) circle (2pt);
						\end{scope}
					\end{scope}

				\end{scope}
			\end{tikzpicture}
		\end{center}
		\caption{\label{fig: bot path} All possible configurations of the topmost path(s) in $\mathcal{L}$ together with their counterparts in $\mathcal{L}^t$. We continue the same conventions as in \Cref{fig: induction basis}. Further, the roots in the regions indicated by $*$s are charmed or ordinary arbitrarily. The root $(i,k+1)$ represented by a black dot at the bottom of the figure can be charmed or ordinary if the path leaves it by a north-east step. If the path leaves $(i,k+1)$ by a south-east step, then $(i,k+1)$ is charmed if and only if $(i+1,k+1)$ is ordinary.}
	\end{figure}
\end{proof}

We can give an alternative description for reading the blocks of the noncrossing partition from the corresponding intimate family by following the above proof.

\begin{corollary}
	\label{cor:tails}
	Let $\L$ be an $\heart_c$-intimate family and $\nc \in \NC(\S_n,c)$ the corresponding noncrossing partition with $(\Out(\L),\In(\L))=(O(\nc),I(\nc))$. The blocks of $\nc$ consist of the integers which are connected by paths in $\L$ after reinterpreting each kiss between a pair of paths in $\L$ at a charmed root as a crossing.
\end{corollary}

\begin{proof}
	We follow analogous steps as in the above proof.
	For the linear Coxeter element $\co = s_1s_2 \cdots s_{n-1}$, the paths in an $\heart_{c_1}$-intimate family are nonintersecting and by comparing with \Cref{prop:bal_to_nc} we see that the assertion is true. Next we show the general case by Cambrian induction. For $c \xrightarrow{k} c^\prime$ let $\L^\prime$ be the $\heart_{c^\prime}$-intimate family as defined in the above proof. Then it suffices to show that two integers $i,j$ are connected in $\L$ after reinterpreting kissing paths as crossing paths if and only if $s_k(i),s_k(j)$ are connected in $\L^\prime$ after reinterpreting kissing paths as crossing paths. By the proof of \Cref{prop:inductive step} it suffices to check the cases shown in \Cref{fig: induction basis} and \Cref{fig: bot path} which imply the assertion.
\end{proof}

\begin{theorem}
	\label{thm:charm}
	For all Coxeter elements $c$, the bijection $\Charm_c$ is the unique support-preserving bijection between $\NC(\S_n,c)$ and $\NN(\S_n)$ satisfying
	$\Krow_{c} \circ \; \Charm_{c} = \Charm_{c} \circ \Krew_{c}$.
\end{theorem}

\newsavebox\cubea\sbox\cubea{
	\raisebox{-.5\height}{\scalebox{0.3}{
			\begin{tikzpicture}
				\draw[red, fill=red!40!white,opacity=.3] (1,1) -- (1,2) -- (2,2) -- (2,1);
				\node (00) at (0,0) {$\bullet$};
				\node (30) at (3,0) {$\bullet$};
				\node (11) at (1,1) {$\bullet$};
				\node (21) at (2,1) {$\bullet$};
				\node (03) at (0,3) {$\bullet$};
				\node (33) at (3,3) {$\bullet$};
				\node (12) at (1,2) {$\bullet$};
				\node (22) at (2,2) {$\bullet$};
				\draw (00) -- (30);
				\draw (11) -- (21);
				\draw (12) -- (22);
				\draw (03) -- (33);
				\draw (03) -- (00);
				\draw (12) -- (03);
				\draw (22) -- (33);
				\draw (11) -- (00);
				\draw[->, ultra thick, red] (21) -- (30);
				\draw (12) -- (11);
				\draw (33) -- (30);
				\draw[->, ultra thick, red] (22) -- (21);
			\end{tikzpicture}}}}

\newsavebox\cubeaa\sbox\cubeaa{
	\raisebox{-.5\height}{\scalebox{0.3}{
			\begin{tikzpicture}
				\node (00) at (0,0) {$\bullet$};
				\node (30) at (3,0) {$\bullet$};
				\node (11) at (1,1) {$\bullet$};
				\node (21) at (2,1) {$\bullet$};
				\node (03) at (0,3) {$\bullet$};
				\node (33) at (3,3) {$\bullet$};
				\node (12) at (1,2) {$\bullet$};
				\node (22) at (2,2) {$\bullet$};
				\draw (00) -- (30);
				\draw (11) -- (21);
				\draw (12) -- (22);
				\draw (03) -- (33);
				\draw (03) -- (00);
				\draw (12) -- (03);
				\draw (22) -- (33);
				\draw (11) -- (00);
				\draw[->, ultra thick, red] (21) -- (30);
				\draw (12) -- (11);
				\draw (33) -- (30);
				\draw[->, ultra thick, red] (22) -- (21);
			\end{tikzpicture}}}}
\newsavebox\cubeb\sbox\cubeb{\raisebox{-.5\height}{
		\scalebox{0.3}{
			\begin{tikzpicture}
				\draw[red, fill=red!40!white,opacity=.3] (0,0) -- (1,1) -- (2,1) -- (3,0) -- (0,0);
				\node (00) at (0,0) {$\bullet$};
				\node (30) at (3,0) {$\bullet$};
				\node (11) at (1,1) {$\bullet$};
				\node (21) at (2,1) {$\bullet$};
				\node (03) at (0,3) {$\bullet$};
				\node (33) at (3,3) {$\bullet$};
				\node (12) at (1,2) {$\bullet$};
				\node (22) at (2,2) {$\bullet$};
				\draw (00) -- (30);
				\draw[->, ultra thick, red] (11) -- (21);
				\draw[->, ultra thick, red] (22) -- (12);
				\draw (03) -- (33);
				\draw (03) -- (00);
				\draw (12) -- (03);
				\draw (22) -- (33);
				\draw (11) -- (00);
				\draw[->, ultra thick, red] (21) -- (30);
				\draw[->, ultra thick, red] (12) -- (11);
				\draw (33) -- (30);
				\draw (22) -- (21);
			\end{tikzpicture}}}}
\newsavebox\cubebb\sbox\cubebb{\raisebox{-.5\height}{
		\scalebox{0.3}{
			\begin{tikzpicture}
				\node (00) at (0,0) {$\bullet$};
				\node (30) at (3,0) {$\bullet$};
				\node (11) at (1,1) {$\bullet$};
				\node (21) at (2,1) {$\bullet$};
				\node (03) at (0,3) {$\bullet$};
				\node (33) at (3,3) {$\bullet$};
				\node (12) at (1,2) {$\bullet$};
				\node (22) at (2,2) {$\bullet$};
				\draw (00) -- (30);
				\draw[->, ultra thick, red] (11) -- (21);
				\draw[->, ultra thick, red] (22) -- (12);
				\draw (03) -- (33);
				\draw (03) -- (00);
				\draw (12) -- (03);
				\draw (22) -- (33);
				\draw (11) -- (00);
				\draw[->, ultra thick, red] (21) -- (30);
				\draw[->, ultra thick, red] (12) -- (11);
				\draw (33) -- (30);
				\draw (22) -- (21);
			\end{tikzpicture}}}}
\newsavebox\cubec\sbox\cubec{\raisebox{-.5\height}{
		\scalebox{0.3}{
			\begin{tikzpicture}
				\draw[red, fill=red!40!white,opacity=.3] (0,0) -- (0,3) -- (1,2) -- (1,1) -- (0,0);
				\node (00) at (0,0) {$\bullet$};
				\node (30) at (3,0) {$\bullet$};
				\node (11) at (1,1) {$\bullet$};
				\node (21) at (2,1) {$\bullet$};
				\node (03) at (0,3) {$\bullet$};
				\node (33) at (3,3) {$\bullet$};
				\node (12) at (1,2) {$\bullet$};
				\node (22) at (2,2) {$\bullet$};
				\draw[->, ultra thick, red] (00) -- (30);
				\draw (11) -- (21);
				\draw[->, ultra thick, red] (22) -- (12);
				\draw (03) -- (33);
				\draw (03) -- (00);
				\draw (12) -- (03);
				\draw (22) -- (33);
				\draw[->, ultra thick, red] (11) -- (00);
				\draw (21) -- (30);
				\draw[->, ultra thick, red] (12) -- (11);
				\draw (33) -- (30);
				\draw (22) -- (21);
			\end{tikzpicture}}}}
\newsavebox\cubecc\sbox\cubecc{\raisebox{-.5\height}{
		\scalebox{0.3}{
			\begin{tikzpicture}
				\node (00) at (0,0) {$\bullet$};
				\node (30) at (3,0) {$\bullet$};
				\node (11) at (1,1) {$\bullet$};
				\node (21) at (2,1) {$\bullet$};
				\node (03) at (0,3) {$\bullet$};
				\node (33) at (3,3) {$\bullet$};
				\node (12) at (1,2) {$\bullet$};
				\node (22) at (2,2) {$\bullet$};
				\draw[->, ultra thick, red] (00) -- (30);
				\draw (11) -- (21);
				\draw[->, ultra thick, red] (22) -- (12);
				\draw (03) -- (33);
				\draw (03) -- (00);
				\draw (12) -- (03);
				\draw (22) -- (33);
				\draw[->, ultra thick, red] (11) -- (00);
				\draw (21) -- (30);
				\draw[->, ultra thick, red] (12) -- (11);
				\draw (33) -- (30);
				\draw (22) -- (21);
			\end{tikzpicture}}}}
\newsavebox\cubed\sbox\cubed{\raisebox{-.5\height}{
		\scalebox{0.3}{
			\begin{tikzpicture}
				\draw[red, fill=red!40!white,opacity=.3] (0,0) -- (0,3) -- (3,3) -- (3,0);
				\node (00) at (0,0) {$\bullet$};
				\node (30) at (3,0) {$\bullet$};
				\node (11) at (1,1) {$\bullet$};
				\node (21) at (2,1) {$\bullet$};
				\node (03) at (0,3) {$\bullet$};
				\node (33) at (3,3) {$\bullet$};
				\node (12) at (1,2) {$\bullet$};
				\node (22) at (2,2) {$\bullet$};
				\draw[->, ultra thick, red] (00) -- (30);
				\draw (11) -- (21);
				\draw[->, ultra thick, red] (22) -- (12);
				\draw (03) -- (33);
				\draw[->, ultra thick, red] (03) -- (00);
				\draw[->, ultra thick, red] (12) -- (03);
				\draw (22) -- (33);
				\draw (11) -- (00);
				\draw (21) -- (30);
				\draw (12) -- (11);
				\draw (33) -- (30);
				\draw (22) -- (21);
			\end{tikzpicture}}}}
\newsavebox\cubedd\sbox\cubedd{\raisebox{-.5\height}{
		\scalebox{0.3}{
			\begin{tikzpicture}
				\node (00) at (0,0) {$\bullet$};
				\node (30) at (3,0) {$\bullet$};
				\node (11) at (1,1) {$\bullet$};
				\node (21) at (2,1) {$\bullet$};
				\node (03) at (0,3) {$\bullet$};
				\node (33) at (3,3) {$\bullet$};
				\node (12) at (1,2) {$\bullet$};
				\node (22) at (2,2) {$\bullet$};
				\draw[->, ultra thick, red] (00) -- (30);
				\draw (11) -- (21);
				\draw[->, ultra thick, red] (22) -- (12);
				\draw (03) -- (33);
				\draw[->, ultra thick, red] (03) -- (00);
				\draw[->, ultra thick, red] (12) -- (03);
				\draw (22) -- (33);
				\draw (11) -- (00);
				\draw (21) -- (30);
				\draw (12) -- (11);
				\draw (33) -- (30);
				\draw (22) -- (21);
			\end{tikzpicture}}}}
\newsavebox\cubee\sbox\cubee{\raisebox{-.5\height}{
		\scalebox{0.3}{
			\begin{tikzpicture}
				\draw[red, fill=red!40!white,opacity=.3] (1,2) -- (0,3) -- (3,3) -- (2,2) -- (1,2);
				\node (00) at (0,0) {$\bullet$};
				\node (30) at (3,0) {$\bullet$};
				\node (11) at (1,1) {$\bullet$};
				\node (21) at (2,1) {$\bullet$};
				\node (03) at (0,3) {$\bullet$};
				\node (33) at (3,3) {$\bullet$};
				\node (12) at (1,2) {$\bullet$};
				\node (22) at (2,2) {$\bullet$};
				\draw (00) -- (30);
				\draw (11) -- (21);
				\draw[->, ultra thick, red] (22) -- (12);
				\draw[->, ultra thick, red] (03) -- (33);
				\draw (03) -- (00);
				\draw[->, ultra thick, red] (12) -- (03);
				\draw (22) -- (33);
				\draw (11) -- (00);
				\draw (21) -- (30);
				\draw (12) -- (11);
				\draw[->, ultra thick, red] (33) -- (30);
				\draw (22) -- (21);
			\end{tikzpicture}}}}
\newsavebox\cubeee\sbox\cubeee{\raisebox{-.5\height}{
		\scalebox{0.3}{
			\begin{tikzpicture}
				\node (00) at (0,0) {$\bullet$};
				\node (30) at (3,0) {$\bullet$};
				\node (11) at (1,1) {$\bullet$};
				\node (21) at (2,1) {$\bullet$};
				\node (03) at (0,3) {$\bullet$};
				\node (33) at (3,3) {$\bullet$};
				\node (12) at (1,2) {$\bullet$};
				\node (22) at (2,2) {$\bullet$};
				\draw (00) -- (30);
				\draw (11) -- (21);
				\draw[->, ultra thick, red] (22) -- (12);
				\draw[->, ultra thick, red] (03) -- (33);
				\draw (03) -- (00);
				\draw[->, ultra thick, red] (12) -- (03);
				\draw (22) -- (33);
				\draw (11) -- (00);
				\draw (21) -- (30);
				\draw (12) -- (11);
				\draw[->, ultra thick, red] (33) -- (30);
				\draw (22) -- (21);
			\end{tikzpicture}}}}
\newsavebox\cubef\sbox\cubef{\raisebox{-.5\height}{
		\scalebox{0.3}{
			\begin{tikzpicture}
				\draw[red, fill=red!40!white,opacity=.3] (2,1) -- (2,2) -- (3,3) -- (3,0) -- (2,1);
				\node (00) at (0,0) {$\bullet$};
				\node (30) at (3,0) {$\bullet$};
				\node (11) at (1,1) {$\bullet$};
				\node (21) at (2,1) {$\bullet$};
				\node (03) at (0,3) {$\bullet$};
				\node (33) at (3,3) {$\bullet$};
				\node (12) at (1,2) {$\bullet$};
				\node (22) at (2,2) {$\bullet$};
				\draw (00) -- (30);
				\draw (11) -- (21);
				\draw (12) -- (22);
				\draw (03) -- (33);
				\draw (03) -- (00);
				\draw (12) -- (03);
				\draw[->, ultra thick, red] (22) -- (33);
				\draw (11) -- (00);
				\draw (21) -- (30);
				\draw (12) -- (11);
				\draw[->, ultra thick, red] (33) -- (30);
				\draw (22) -- (21);
			\end{tikzpicture}}}}

\begin{proof}
	We deduce the equivariance of $\Charm_{c'}$ given equivariance of $\Charm_c$ using the commutativity of the other faces of the diagram in~\Cref{fig:outline}.  We compute:%
	\begin{align*} \mrho_{c'} \circ \; \Charm_{c'} & = \mrho_{c'} \circ \beta_{c,k} \circ \Charm_c \circ \alpha_{c,k}^{-1} & \text{by~\Cref{prop:inductive step}} \\
                                               & = \beta_{c,k} \circ \mrho_c \circ \; \Charm_c \circ \alpha_{c,k}^{-1} & \text{by \Cref{prop:nn_induction}}   \\
                                               & = \beta_{c,k} \circ \Charm_c \circ \Kr_c \circ \alpha_{c,k}^{-1}      & \text{by Cambrian induction}         \\
                                               & = \Charm_{c'} \circ \alpha_{c,k} \circ \Kr_c \circ \alpha_{c,k}^{-1}  & \text{by~\Cref{prop:inductive step}} \\
                                               & = \Charm_{c'} \circ \Kr_{c^\prime}                                    & \text{by~\Cref{prop:nc_induction}}.
	\end{align*}
	Therefore, if \Cref{thm:main_theorem} holds for $c$, then \Cref{thm:main_theorem} holds for $c'$.  Since we know that the theorem holds for $\co$, by~\Cref{lem:initial_conj} we conclude \Cref{thm:main_theorem} for all (standard) Coxeter elements.

	To understand it pictorially, we have
	\[ \usebox\cubeaa \xrightarrow[\usebox\cubea]{\textit{Pr.}~\ref{prop:inductive step}} \usebox\cubebb \xrightarrow[\usebox\cubeb]{\textit{Pr.}~\ref{prop:nn_induction}} \usebox\cubecc \xrightarrow[\usebox\cubec]{\textit{Sec.}~\ref{sec:cambrian_induction}} \usebox\cubedd \xrightarrow[\usebox\cubed]{\textit{Pr.}~\ref{prop:inductive step}} \usebox\cubeee \xrightarrow[\usebox\cubee]{\textit{Pr.}~\ref{prop:nc_induction}} \usebox\cubef \]
	where, with respect to the statements of \Cref{prop:inductive step} and \Cref{prop:nc_induction}, we can reverse some arrows because the involved functions are all bijections.

	To establish the uniqueness of these bijections, we show that $w \in \NC(\S_n,c)$ is uniquely identified by
	\[ \orbsupp(w) \coloneqq (\Supp(w),\Supp(\Kr_c(w)),\Supp(\Kr_c^2(w)),\dots,\Supp(\Kr_c^{2n-1}(w))). \]
	Suppose this claim is established. Making a similar definition for $x\in\NN(\S_n)$,
	\[ \orbsupp(x) \coloneqq (\Supp(x),\Supp(\rho_c(x)),\Supp(\rho_c^2(x)),\dots,\Supp(\rho_c^{2n-1}(x))), \]
	the rest of the main theorem tells us that $\phi_c(w)$ satisfies $\orbsupp(w)=\orbsupp(\phi_c(w))$. Since, by the claim, distinct noncrossing partitions $w$ yield distinct values for $\orbsupp(w)$, it must also be that the values of $\orbsupp(x)$ are distinct for distinct $x\in \NN(\S_n)$. Then $\phi_c$ is the unique map sending $w$ to the $x\in \NN(\S_n)$ such that $\orbsupp(w)=\orbsupp(x)$.

	We now establish the claim that $w\in \NC(\S_n,c)$ is uniquely identified by $\orbsupp(w)$ (i.e., that the map $\orbsupp$ is injective on $\NC(\S_n,c)$). Let $w\in \NC(\S_n,c)$. Let $1\leq i <j\leq n$, and suppose that $i$ and $j$ are in different blocks of $w$. There exists a coarsening $w'$ of $w$ which has only two parts in which $i$ and $j$ are still in different blocks. (Such a coarsening can be obtained by successively choosing a block that contains neither $i$ nor $j$, but which is adjacent to a block that does contain one of them, and merging
	the selected block with the adjacent block.)
	Suppose that in $w'$ the block containing $i$ has $k$ elements, and the block containing $j$ has $n-k$ elements. There is a rotation of $w'$, say $c^{-t} wc^t$  in which the image of the block containing $i$ consists of the numbers $\{1,\dots,k\}$, while the image of the block containing $j$ consists of $\{k+1,\dots,n\}$. (This is true since there is a rotation of $w'$ such that the block containing $i$ consists of any $k$ numbers consecutive around the circle, and $\{1,\dots,k\}$ are the labels of one such set.) It follows that $\alpha_k$ is not in the support of $c^{-t} w'c^{t}$, and thus it is not in the support of $c^{-t} wc^t$ either. This fact can be read off from $\orbsupp(w)$.
	This allows us to deduce that if $\widetilde w$ is a noncrossing partition such that $\orbsupp(\widetilde w)=\orbsupp(w)$, then $i$ and $j$ must be in different blocks, and thus, by applying the same idea with any pair $i,j$ in different blocks, it follows that $\widetilde w$ must refine $w$. Reversing the argument, $w$ must refine $\widetilde w$. Thus $\widetilde w=w$.
\end{proof}

\section{Future Work}

It is natural to wish to extend our bijections to all Dynkin types.  A first step is to extend the notion of $c$-charmed---to this end, let $W$ be an irreducible finite Weyl group with positive roots $\Phi^+$, and let $c$ be a standard Coxeter element.  For $s$ a simple reflection, we say that its corresponding simple root $\alpha_s$ is \defn{$c$-charmed} if $s$ is either initial or final in $c$. Then a general positive root $\alpha \in \Phi^+$ is \defn{$c$-charmed} if and only if $s(\alpha)$ is $scs$-charmed.

\subsection{Type \texorpdfstring{$C$}{C}}
\label{sec:typec}
In this short subsection, we prove that our techniques immediately extend to type $C_n$ by folding.  We note that although the root posets of types $B_n$ and $C_n$ are isomorphic as posets, the labeling of these posets by roots differs in a fundamental way---the labeling of $\Phi^+(C_n)$ is a folding of $\Phi^+(A_{n-1})$, but this is not the case for $B_n$.

The \defn{hyperoctahedral group} $W(C_n)$ is the group of signed permutations---that is, bijections $w$ from $[\pm n] \coloneqq \{-n,\ldots,-1,1,\ldots,n\}$ such that $w(i)=-w(-i)$ for all $1 \leq i \leq n$.  As a Coxeter group, $W(C_n)$ is generated by the simple reflections $s_i^C \coloneqq (i,i+1)(-i,-i-1)=(\;(i,i+1)\;)$ for $1 \leq i <n$ and $s_n^C=(n,-n)$.  A simple folding argument gives an embedding $\iota \colon W(C_n) \hookrightarrow \S_{2n}$:
\begin{align*}
	\iota(s_i^C) & = s_i s_{2n-i}=(i,i+1)(2n-i,2n+1-i) \text{ for } 1 \leq i <n \text{ and} \\
	\iota(s_n^C) & = s_n=(n,n+1).
\end{align*}
Under this folding, any Coxeter element $c^C \in W(C_n)$ is sent to a Coxeter element $c=\iota(c^C)$ of $\S_{2n}$.  It is easy to see that $\AR(c^C)$ is a folding of $\AR(c)$ horizontally, and that $\Phi^+(C_n)$ is a folding $\Phi^+(A_{2n})$ vertically---thus, the $c^C$-Kroweras complement $\Krow_{c^C}$ is obtained by folding $\Krow_c$.  Similarly, $\iota \colon \NC(W(C_n),c^C) \hookrightarrow \NC(\S_{2n},c)$ sends a noncrossing partition $\pi^C$ to a noncrossing partition $\pi=\iota(\pi^C)$ that is centrally symmetric, and so the $c^C$-Kreweras complement is obtained by folding $\Krew_c$.  We conclude the following.

\begin{corollary}
	Let $W(C_n)$ be the group of signed permutations, and fix a standard Coxeter element $c \in W(C_n)$.  Then there is a unique bijection $\Charm_c \colon \NC(W(C_n),c) \to \NN(C_n)$ satisfying
	\begin{itemize}
		\item $\Charm_c \circ \Krew_c = \Krow_c \circ \; \Charm_c$ and
		\item $\Supp = \Supp \circ \Charm_c$.
	\end{itemize}
\end{corollary}

\bibliographystyle{amsalpha}
\bibliography{bibliography}

\end{document}